\documentclass{amsart}
\usepackage{times,mathrsfs,array,amssymb,pb-diagram}

\newcounter{lemma}
\newtheorem{Theorem}{Theorem}
\newtheorem{Lemma}[lemma]{Lemma}
\newtheorem{Corollary}[lemma]{Corollary}
\newtheorem{Proposition}[lemma]{Proposition}

\theoremstyle{definition}

\newtheorem{Definition}[lemma]{Definition}
\newtheorem{Notation}[lemma]{Notation}

\newtheorem{Remark}[lemma]{Remark}

\def\C{\mathbb C}
\def\H{\mathbb H}
\def\P{\mathbb P}
\def\Q{\mathbb Q}
\def\S{\mathcal S}
\def\Z{\mathbb Z}
\def\new{\mathrm{new}}
\def\mod{\  \mathrm{mod}\ }
\def\ord{\mathrm{ord}}
\def\sgn{\mathrm{sgn\,}}
\def\Im{\mathrm{Im\,}}
\def\tr{\mathit{tr\,}}
\def\JS#1#2{\left(\frac{#1}{#2}\right)}
\def\MM{\mathscr M}
\def\QQ{\mathcal Q}
\def\SL{\mathrm{SL}}
\def\M#1#2#3#4{\begin{pmatrix}#1&#2\\#3&#4\end{pmatrix}}
\def\SM#1#2#3#4{\left(\begin{smallmatrix}#1&#2\\#3&#4\end{smallmatrix}
  \right)}

\begin{document}
\title{Modular forms of half-integral weights on $\SL(2,\Z)$}
\author{Yifan Yang}
\address{Department of Applied Mathematics, National Chiao Tung
  University and National Center for Theoretical Sciences, Hsinchu
  300, Taiwan}
\email{yfyang@math.nctu.edu.tw}
\date\today
\subjclass[2000]{Primary 11F37; secondary 11F20, 11F27}

\begin{abstract} In this paper, we prove that, for an integer $r$ with
  $(r,6)=1$ and $0<r<24$ and a nonnegative even integer $s$, the set
  $$
    \{\eta(24\tau)^rf(24\tau):f(\tau)\in M_s(1)\}
  $$
  is isomorphic to
  $$
    S_{r+2s-1}^\new\left(6,-\JS8r,-\JS{12}r\right)\otimes\JS{12}\cdot
  $$
  as Hecke modules under the Shimura correspondence. Here
  $M_s(1)$ denotes the space of modular forms of weight $s$ on
  $\Gamma_0(1)=\SL(2,\Z)$, $S_{2k}^\new(6,\epsilon_2,\epsilon_3)$ is
  the space of newforms of weight $2k$ on $\Gamma_0(6)$ that are
  eigenfunctions with eigenvalues $\epsilon_2$ and $\epsilon_3$ for
  Atkin-Lehner involutions $W_2$ and $W_3$, respectively, and the
  notation $\otimes\JS{12}\cdot$ means the twist by the quadratic
  character $\JS{12}\cdot$. There is also an analogous result for the
  cases $(r,6)=3$.
\end{abstract}
\maketitle

\begin{section}{Introduction}
Let
$$
  \theta(\tau)=\sum_{n\in\Z}q^{n^2}, \qquad q=e^{2\pi i\tau},
$$
be the Jacobi theta function. Then Shimura's theory of modular forms
of half-integral weights can be described as follows. Let $k$ be a
positive integer, $N$ a positive integer and $\chi$ a Dirichlet
character modulo $4N$. We say a holomorphic function
$f:\H=\{\tau:\Im\,\tau>0\}\to\C$ is a \emph{modular form of half
  integral weight} $k+1/2$ on $\Gamma_0(4N)$ with character $\chi$ if
it is holomorphic at each cusp and satisfies 
\begin{equation*}
  \frac{f(\gamma\tau)}{f(\tau)}=\chi(d)\frac{\theta(\gamma\tau)^{2k+1}}
  {\theta(\tau)^{2k+1}}
\end{equation*}
for all $\gamma=\SM abcd\in\Gamma_0(4N)$. Let $M_{k+1/2}(4N,\chi)$
denote the space of these functions. In \cite{Shimura-correspondence},
Shimura showed how the Hecke theory can be extended to these spaces.
More importantly, he showed that if $f\in M_{k+1/2}(4N,\chi)$ is a
Hecke eigenform, then there is a corresponding Hecke eigenform of
integral weight $2k$ with character $\chi^2$ that shares the same
eigenvalues. Moreover, he conjectured that the level of this modular
form of integral weight can be taken to be $2N$. This conjecture was
later proved by Niwa \cite{Niwa}. (See also \cite{Shintani}.) In
literature, this correspondence between modular forms of half-integral
weights and modular forms of integral weights is called the
\emph{Shimura correspondence}.

In \cite{Shimura-correspondence}, the correspondence was proved by
using Weil's characterization of Hecke eigenforms in terms of
$L$-functions. From the representation-theoretical point of view, this
correspondence amounts to a correspondence from certain automorphic
representations of the metaplectic double cover of
$\mathrm{GL}(2,\mathbb A_\Q)$ to automorphic representations of
$\mathrm{GL}(2,\mathbb A_\Q)$, where $\mathbb A_\Q$ denotes the adele
ring of $\Q$. See \cite{Flicker,Gelbart,Waldspurger} for more details.

In general, the Shimura correspondence is not one-to-one. In
particular, the multiplicity-one theorem does not hold for modular
forms of half-integral weights in general. In order to get a
multiplicity-one result, in the fundamental works
\cite{Kohnen-level1,Kohnen-newform}, Kohnen introduced a subspace of
$M_{k+1/2}(4N,\chi)$ and developed a newform theory for this subspace
that is parallel to the Atkin-Lehner-Li theory of newforms for modular
forms of integral weights. To state Kohnen's result, let $\chi$ be a
Dirichlet character modulo $4N$ and set $\epsilon=\chi(-1)$. Let
$S_{k+1/2}^+(4N,\JS{4\epsilon}\cdot\chi)$ be the subspace consisting
of cusp forms of half-integral weight $k+1/2$ and character
$\JS{4\epsilon}\cdot\chi$ on $\Gamma_0(4N)$ whose Fourier expansions
are of the form
$$
  \sum_{\epsilon(-1)^kn\equiv 0,1\mod 4}a_ne^{2\pi in\tau}.
$$
Then Kohnen proved that, under the assumptions that $N$ is odd and
squarefree and $\chi$ is a quadratic character, the image of
$S_{k+1/2}^+(4N,\JS{4\epsilon}\cdot\chi)$ under the Shimura
correspondence is $S_{2k}(N)$. Moreover, there is a canonically
defined subspace $S_{k+1/2}^\new(4N,\JS{4\epsilon}\cdot\chi)\subset
S_{k+1/2}^+(4N,\JS{4\epsilon}\cdot\chi)$ such that
$S_{k+1/2}^\new(4N,\JS{4\epsilon}\cdot\chi)\simeq S_{k+1/2}^\new(N)$ as
Hecke modules. In particular, the strong multiplicity-one theorem
holds for $S_{k+1/2}^\new(4N,\JS{4\epsilon}\cdot\chi)$. 
Kohnen's work was later extended by several authors in various
directions \cite{Frechette,Frechette2,Ueda-Yamana}.

Modular forms of half-integral weights are closely related to many
problems in number theory. For example, let $p(n)$ denote the number
of ways to write a positive integer $n$ as unordered sums of positive
integers. Then the generating function of the partition function
$p(n)$ is equal to
$$
  \sum_{n=0}^\infty p(n)q^n=\prod_{m=1}^\infty\frac1{1-q^m}.
$$
If we set $q=e^{2\pi i\tau}$, then the infinite product above is
essentially the reciprocal of the Dedekind eta function, which is
well-known to be a modular form of weight $1/2$ on $\Gamma_0(576)$
with character $\JS{12}\cdot$. Using this fact, along with the Shimura
correspondence and properties of Galois representations attached to
cusp forms, Ono \cite{Ono-Annals} proved that for every prime $m$
greater than $3$, there is a positive proportion of primes $\ell$ such
that the congruence
$$
  p\left(\frac{m\ell^3n+1}{24}\right)\equiv 0\mod m
$$
holds for all integers $n$ relatively prime to $\ell$. This result was
later extended by several authors
\cite{Ahlgren-MAnn,Ahlgren-Ono,Yang-partition}. For example, in
\cite{Yang-partition}, the author of the present paper showed that for
every prime $m$ greater than $3$ and every prime different from $2$,
$3$, and $m$, there is an explicitly computable integer $k$ such that
$$
  p\left(\frac{m\ell^kn+1}{24}\right)\equiv 0\mod m
$$
for all integers $n$ relatively prime to $\ell$. A key ingredient in
\cite{Yang-partition} is the Hecke invariance of the space
\begin{equation} \label{equation: Srs}
  \S_{r,s}=\{\eta(24\tau)^rf(24\tau):~f(\tau)\in M_s(1)\},
\end{equation}
where $r$ is an odd integer between $0$ and $24$, $s$ is a nonnegative
even integer, and $M_s(1)$ is the space of modular forms of weight $s$
on $\Gamma_0(1)=\SL(2,\Z)$. That is, even though the space
$M_{k+1/2}(576,\JS{12}\cdot)$ itself has a huge dimension, it contains
several subspaces of small dimensions that are invariant under the
action of Hecke algebra. The invariance of these spaces was first
proved by Garvan \cite{Garvan} and later rediscovered by the author of
the present paper independently. (See the arxiv version of
\cite{Yang-partition} for a proof of the invariance.)

Now recall that a well-known result of Waldspurger \cite{Waldspurger2}
states that if $f$ is a Hecke eigenform of half-integral weight
$k+1/2$ and $F$ is the corresponding Hecke eigenform of integral
weight $2k$, then for squarefree $n$, the square of the $n$th Fourier
coefficient of $f$ is essentially proportional to the special value at
$s=k$ of $L(F\otimes\chi_{(-1)^kn},s)$, where $\chi_{(-1)^kn}$ is the
Kronecker character of the quadratic field $\Q(\sqrt{(-1)^kn})$.
(See also \cite{Kohnen-Zagier}.) Using this result of Waldspurger, Guo
and Ono \cite{Guo-Ono} related the arithmetic of the partition
function $p(n)$ to the arithmetic of certain motives. Specifically,
let $13\le\ell\le 31$ be a prime. Let $r$ be the unique integer
between $0$ and $24$ such that $r\equiv-\ell\mod 24$ and let
$s=(\ell-r-2)/2$. Then the space $\S_{r,s}$ defined in
\eqref{equation: Srs} is one-dimensional and spanned by
$g_\ell(\tau)=\eta(24\tau)^rE_s(24\tau)$, where $E_s(\tau)$ denotes
the Eisenstein series. It is known that
\begin{equation} \label{equation: congruence}
  g_\ell(\tau)\equiv \sum_{n=0}^\infty p\left(\frac{\ell
      n+1}{24}\right)q^n\mod \ell.
\end{equation}
Then Guo and Ono showed that if we let $G_\ell(\tau)$ be the unique
Hecke eigenform in $S_{\ell-3}(6)$ with Fourier expansion
\begin{equation*} 
  G_\ell(\tau)=q+\JS8r 2^{(\ell-5)/2}q^2+\JS{12}r 3^{(\ell-5)/2}q^3+\cdots,
\end{equation*}
then the image of $g_\ell(\tau)$ under the Shimura correspondence is
$G_\ell\otimes\JS{12}\cdot$, which is a newform of level $144$. (Note
that $g_\ell(\tau)$ is contained in Kohnen's $+$-space.) In view of
Waldspurger's result and \eqref{equation: congruence}, this means that
the values of the partition function modulo $\ell$ are related to the
values at the center point of the $L$-function of $G_\ell$ twisted by
quadratic Dirichlet characters. Thus, assuming the truth of the
Bloch-Kato conjecture, the arithmetic properties of the partition
function are related to those of certain motives associated to
$G_\ell$.

Now observe that, by \cite[Theorem 3]{Atkin-Lehner}, the function
$G_\ell(\tau)$ is contained in the Atkin-Lehner eigensubspace of
$S_{\ell-3}^\new(6)$ with eigenvalues $-\JS8r$ and $-\JS{12}r$ for the
Atkin-Lehner involutions $W_2$ and $W_3$, respectively. In other
words, for the few cases considered in \cite{Guo-Ono}, the Shimura
correspondence yields an isomorphism
$$
  \S_{r,s}\simeq S_{r+2s-1}^\new\left(6,-\JS8r,-\JS{12}r\right)
  \otimes\JS{12}\cdot
$$
as Hecke modules, where $S_{2k}^\new(6,\epsilon_2,\epsilon_3)$ denotes
the space of newforms of weight $2k$ on $\Gamma_0(6)$ that are
eigenfunctions with eigenvalues $\epsilon_2$ and $\epsilon_3$ for
$W_2$ and $W_3$. (Note that the Hecke algebras on the two sides are
isomorphic. Thus, we may talk about isomorphisms as Hecke modules.)
It is natural to ask whether this isomorphism holds in general. The
purpose of this paper is to prove that this is indeed the case.

\begin{Theorem} \label{theorem: 1} Let $r$ be an integer satisfying
  $(r,6)=1$ and $0<r<24$ and $s$ be a nonnegative even integer. Let
  $$
    \S_{r,s}=\{\eta(24\tau)^rf(24\tau):~f(\tau)\in M_s(1)\}\subset
    S_{r/2+s}\left(576,\JS{12}\cdot\right),
  $$
  where $M_s(1)$ denotes the space of modular forms of weight $s$ on
  $\Gamma_0(1)=\SL(2,\Z)$. Then the Shimura correspondence yields an
  isomorphism
  $$
    \S_{r,s}\simeq S_{r+2s-1}^\new\left(6,-\JS8r,-\JS{12}r\right)
    \otimes\JS{12}\cdot
  $$
  as Hecke modules.
\end{Theorem}

For odd integers $r$ that are divisible by $3$, we have also an
analogous result.

\begin{Theorem} \label{theorem: 2} Let $r$ be an odd integer
  satisfying $0<r<8$ and $s$ be a nonnegative even integer. Let
  $$
    \S_{3r,s}=\{\eta(8\tau)^{3r}f(8\tau):~f(\tau)\in M_s(1)\}
    \subset S_{3r/2+s}\left(64,\JS{-4}\cdot\right).
  $$
  Then the Shimura correspondence yields an isomorphism
  $$
    \S_{3r,s}\simeq S_{3r+2s-1}^\new\left(2,-\JS8r\right)
    \otimes\JS{-4}\cdot
  $$
  as Hecke modules.
\end{Theorem}

\begin{Corollary} The multiplicity-one property holds for the spaces
  $\S_{r,s}$ defined in Theorems \ref{theorem: 1} and \ref{theorem: 2}.
\end{Corollary}

\begin{Remark} Note that the space
  $S_{2k}^\new(6,\epsilon_2,\epsilon_3)\otimes\JS{12}\cdot$ is
  contained in $S_{2k}^\new(144,-,-)$, regardless of whether
  $\epsilon_2,\epsilon_3$ are $1$ or $-1$. Also,
  $S_{2k}^\new(2,\epsilon_2)\otimes\JS{-4}\cdot$ is a subspace of
  $S_{2k}^\new(16,-)$ for both $\epsilon_2=1$ and $\epsilon_2=-1$.
\end{Remark}

It turns out that the Hecke invariance of $\S_{r,s}$ and the explicit
Shimura correspondence in Theorems \ref{theorem: 1} and \ref{theorem:
  2} are best explained in terms of modular forms of half-integral
weight of $\eta$-type. Namely, in Shimura's setting, a function is
called a modular form of half-integral weight if its transformation is
comparable with the Jacobi theta function. In a similar way, we say a
function $f(\tau)$ is a modular form of $\eta$-type if its
transformation is comparable with the Dedekind eta function, that is,
if $f(\tau)$ satisfies
$$
  \frac{f(\gamma\tau)}{f(\tau)}=(c\tau+d)^s\frac{\eta(\gamma\tau)^r}{\eta(\tau)^r}
$$
for all $\gamma=\SM abcd$ in a subgroup $\Gamma$ of $\SL(2,\Z)$, where
$s$ is assumed to be a nonnegative even integer and $r$ is an odd
integer between $0$ and $24$. Then it is easy to show that modular
forms of $\eta$-type on $\SL(2,\Z)$ are essentially just the functions
in $\S_{r,s}$ defined in Theorems \ref{theorem: 1} and \ref{theorem:
  2}. (See Proposition \ref{proposition: Srs(1)} below.) This explains
the Hecke invariance of the spaces $\S_{r,s}$.

At the first sight, the introduction of the notion of modular forms of
$\eta$-type is superficial since if $f(\tau)$ is such a function, then
$f(24\tau)$ is just a modular form of half-integral weight with
character $\JS{12}\cdot$ in the sense of Shimura, and we do not get
any new modular forms in this way. However, if a modular
form of half-integral weight on a congruence subgroup $\Gamma_0(4N)$
in the sense of Shimura happens to be a modular form of
$\eta$-type on a larger group, then the extra symmetries from this
larger group will give us additional information about the function.
This is the reason why we can determine such a precise image of
$S_{r,s}$ under the Shimura correspondence. (In the proof of Theorems
\ref{theorem: 1} and \ref{theorem: 2}, we will work with modular forms
of $\eta$-type on $\SL(2,\Z)$ instead of modular forms on the much
smaller group $\Gamma_0(576)$ in the sense of Shimura.)

Our proof of Theorems \ref{theorem: 1} and \ref{theorem: 2} is
classical. That is, since the Hecke modules involved are all
semisimple, to prove the theorems, it suffices to show that the traces
coincide for all Hecke operators. It will be interesting to have a
representation-theoretical proof of the results. Note that here we
only prove Theorem \ref{theorem: 1}; the proof of Theorem
\ref{theorem: 2} is similar, but much simpler, and will be omitted.

The rest of the paper is organized as follows. In Section
\ref{section: preliminaries}, we first define modular forms of
$\eta$-type in more details. We then define Hecke operators and
introduce several basic properties of them. We also describe Shimura's
abstract trace formula \cite{Shimura-trace}. In Section \ref{section:
  trace, half-integral weight}, we compute the traces of
Hecke operators on the space of modular forms of $\eta$-type. This
constitutes the main bulk of the paper. In Section \ref{section:
  trace, integral weight}, we determine the traces of Hecke operators
on $S_{2k}^\new(6,\epsilon_2,\epsilon_3)$. In Section \ref{section:
  comparison}, we show that the traces coincide and thereby establish
Theorem \ref{theorem: 1}.
\end{section}

\begin{section}{Preliminaries}
\label{section: preliminaries}
In this section, we first give a more detailed definition of modular
forms of $\eta$-type. We then define Hecke operators on these modular
forms and review Shimura's trace formula for Hecke operators.

\begin{subsection}{Modular forms of $(\eta^r,s)$-type}

\begin{Notation} Throughout the rest of the paper, we let $r$ and $s$
  be fixed integers with $(r,6)=1$, $0<r<24$ and $s$ even. Set also
  $k=(r-1)/2+s$. Let $\mathfrak G_{k+1/2}$ be the group of pairs
  $(A,\phi(\tau))$, where $A=\SM abcd\in\mathrm{GL}^+(2,\Q)$,
  $\phi(\tau)$ is a holomorphic function $\H\mapsto\C$ satisfying
  $$
    |\phi(\tau)|=(\det A)^{-k/2-1/4}|c\tau+d|^{k+1/2},
  $$
  and the group law is defined by
  $$
    (A,\phi(\tau))(B,\psi(\tau))=(AB,\phi(B\tau)\psi(\tau)).
  $$
  Consider the subgroup $\Gamma^\ast$ of $\mathfrak G_{k+1/2}$ defined
  by
  $$
    \Gamma^\ast=\Gamma_{r,s}^\ast=\left\{\left(\gamma,
    \frac{\eta(\gamma\tau)^r}{\eta(\tau)^r}(c\tau+d)^s\right):
   ~\gamma=\M abcd\in\SL(2,\Z)\right\}
  $$
  For an element $\gamma$ in $\SL(2,\Z)$, we let $\gamma^\ast$ denote
  the element in $\Gamma^\ast$ whose first component is $\gamma$.
  Naturally, if $G$ is a subgroup of $\SL(2,\Z)$, then we let $G^\ast$
  be the subgroup $\{\gamma^\ast:\gamma\in G\}$.
\end{Notation}

Here let us recall a well-known formula for
$\eta(\gamma\tau)/\eta(\tau)$.

\begin{Lemma}[{\cite[pp.125--127]{Weber}}] \label{lemma: eta}
  Let $\gamma=\SM abcd\in\SL(2,\Z)$ with $c\ge 0$. Then we have
  $$
    \frac{\eta(\gamma\tau)}{\eta(\tau)}=\epsilon(a,b,c,d)
   (c\tau+d)^{1/2},
  $$
  where
  \begin{equation} \label{equation: epsilon}
    \epsilon(a,b,c,d)=\begin{cases}
    \JS dc i^{(1-c)/2}e^{2\pi i\left(bd(1-c^2)+c(a+d)-3\right)/24},
      &\text{if }c\text{ is odd}, \\
    \JS cde^{2\pi i\left(ac(1-d^2)+d(b-c+3)-3\right)/24},
      &\text{if }c\text{ is even}. \end{cases}
  \end{equation}
\end{Lemma}

We now define modular forms of $\eta$-type.

\begin{Definition} Let $f:\H\to\C$ be a holomorphic function on the
  upper half-plane. We define the action of
  $\gamma^\ast=(\gamma,\phi_\gamma(\tau))$ on $f$ by
  $$
    (f|\gamma^\ast)(\tau)=\phi_\gamma(\tau)^{-1}f(\gamma\tau).
  $$
  Let $G$ be a subgroup of $\SL(2,\Z)$ of finite index. If the
  function $f$ satisfies
  $$
    (f|\gamma^\ast)(\tau)=f(\tau)
  $$
  for all $\gamma^\ast\in G^\ast$ and is holomorphic at each cusp of
  $G$, then we say $f$ is a \emph{modular form of $(\eta^r,s)$-type}
  on $G$. If, in addition, $f$ vanishes at each cusp of $G$, we say
  $f$ is a \emph{cusp form} of $(\eta^r,s)$-type. The space of cusp
  forms of $(\eta^r,s)$-type on $G$ will be denoted by $\S_{r,s}(G)$.
  If $G=\Gamma_0(N)$, we simply write it as $\S_{r,s}(N)$.
\end{Definition}

In the case $N=1$, the space $\S_{r,s}(1)$ has a very simple
description.

\begin{Proposition} \label{proposition: Srs(1)}
  For $N=1$, we have
  $$
    \S_{r,s}(1)=\{\eta(\tau)^rf(\tau):~f\in M_s(1)\},
  $$
  where $M_s(1)$ is the space of modular forms of weight $s$ on
  $\SL(2,\Z)$.
\end{Proposition}

\begin{proof} Assume that $g(\tau)\in\S_{r,s}(1)$. We have
  $$
    g(\tau+1)=e^{2\pi ir/24}g(\tau).
  $$
  Thus, the Fourier expansion of $g(\tau)$ takes the form
  $q^{r/24}(a_0+a_1q+\cdots)$. Since $\eta(\tau)$ is nonvanishing
  throughout $\H$, the function $g(\tau)/\eta(\tau)^r$ is holomorphic
  on $\H$. Moreover, it is easy to see that for
  $\gamma=\SM abcd\in\SL(2,\Z)$, one has
  $$
    \frac{g(\gamma\tau)}{\eta(\gamma\tau)^r}
   =(c\tau+d)^s\frac{g(\tau)}{\eta(\tau)^r}.
  $$
  Therefore, $g(\tau)/\eta(\tau)^r$ is a modular form of weight $s$ on
  $\SL(2,\Z)$. This proves the proposition.
\end{proof}

\end{subsection}

\begin{subsection}{Hecke operators on $\S_{r,s}(N)$}

\begin{Notation}
  Let $N$ be a positive integer. For a positive integer $n$, let
  \begin{equation*}
  \begin{split}
    \MM_n(N)&=\Gamma_0(N)\M100n\Gamma_0(N) \\
  &=\left\{\M abcd:~a,b,c,d\in\Z,~ad-bc=n,~N|c,~(a,N)=1,(a,b,c,d)=1\right\},
  \end{split}
  \end{equation*}
  and let $\MM_n(N)^\ast$ denote the subset
  $$
    \MM_n(N)^\ast=\Gamma_0(N)^\ast\left(\M 100n,n^{k/2+1/4}\right)
    \Gamma_0(N)^\ast
  $$
  of $\mathfrak G_{k+1/2}$.
\end{Notation}

\begin{Lemma} If $n$ is a positive integer relatively prime to $6$,
  then for each $\gamma\in\MM_{n^2}(N)$, there exists a unique element
  $\gamma^\ast$ in $\MM_{n^2}(N)^\ast$ such that the first component of
  $\gamma^\ast$ is $\gamma$.
\end{Lemma}

\begin{proof} It suffices to prove the case $\gamma=\SM100{n^2}$.
  We are required to show that if $A,B\in\Gamma_0(N)$ are matrices
  such that $A\SM100{n^2}B^{-1}=\SM100{n^2}$, then
  \begin{equation} \label{equation: lemma uniqueness 1}
    A^\ast\left(\M100{n^2},n^{k+1/2}\right)
   =\left(\M100{n^2},n^{k+1/2}\right)B^\ast.
  \end{equation}
  Assume that $A=\SM abcd$. By Lemma \ref{lemma: eta}, we have
  $$
    A^\ast\left(\M100{n^2},n^{k+1/2}\right)
   =\left(\M a{bn^2}c{dn^2},\epsilon(a,b,c,d)^r(c\tau/n^2+d)^{s+r/2}
    n^{k+1/2}\right),
  $$
  where $\epsilon(a,b,c,d)$ is defined by \eqref{equation: epsilon}.
  Now the assumption that $A\SM100{n^2}B^{-1}=\SM100{n^2}$ implies
  that if $A=\SM abcd$, then $B=\SM a{bn^2}{c/n^2}d$. In particular,
  we have $n^2|c$. Thus,
  \begin{equation*}
  \begin{split}
   &\left(\M100{n^2},n^{k+1/2}\right)B^\ast \\
   &\qquad\qquad=\left(\M a{bn^2}c{dn^2},n^{k+1/2}\epsilon(a,bn^2,c/n^2,d)^r
    (c\tau/n^2+d)^{s+r/2}\right).
  \end{split}
  \end{equation*}
  Since $n$ is assumed to be relatively prime to $6$, we have
  $n^2\equiv 1\mod 24$ and hence
  $$
    \epsilon(a,b,c,d)=\epsilon(a,bn^2,c/n^2,d)
  $$
  This establishes \eqref{equation: lemma uniqueness 1} and the lemma.
\end{proof}

The group $\Gamma_0(N)$ acts on $\MM_{n^2}(N)$ by matrix
multiplication on the left. It is clear that if
$f\in\S_{r,s}(N)$ and $\alpha$ and $\beta$ are two elements
of $\MM_{n^2}(N)$ that are equivalent under the left action of
$\Gamma_0(N)$, then
$$
  f\big|\alpha^\ast=f\big|\beta^\ast.
$$
Moreover, since $\SM100{n^2}^{-1}\Gamma_0(N)\SM100{n^2}$ and
$\Gamma_0(N)$ are commensurable, there are finitely many right cosets
in $\Gamma_0(N)\backslash\MM_{n^2}(N)$. Thus, for each positive
integer $n$ with $(n,6)=1$, we can define a linear
operator on $\S_{r,s}(N)$.

\begin{Lemma} \label{lemma: action of M} The mapping
$$
  [\MM_{n^2}(N)^\ast]:f\longmapsto f\big|[\MM_{n^2}(N)^\ast]=
   \sum_{\gamma\in\Gamma_0(N)\backslash\MM_{n^2}(N)}f\big|\gamma^\ast
$$
is a linear operator on $\S_{r,s}(N)$.
\end{Lemma}

We now define Hecke operators on $\S_{r,s}(N)$.

\begin{Definition}
  For a positive integer $n$ with $(n,6)=1$, the
  \emph{Hecke operator} $T_{n^2}$ on $\S_{r,s}(N)$ is defined by
$$
  T_{n^2}:f\longmapsto n^{k-3/2}\sum_{ad=n,a|d}a f\big|[\MM_{(d/a)^2}(N)]
$$
\end{Definition}

\begin{Proposition} Let $p$ be a prime such that $p\nmid 6N$. Then
  for $f(\tau)=\sum_{n=1}^\infty a_f(n)q^{n/24}\in\S_{r,s}(N)$, we have
  $$
    T_{p^2}:f(\tau)\mapsto\sum_{n=1}^\infty\left(
    a_f(p^2n)+\JS{12}p\left(\frac{(-1)^kn}{p}\right)
    p^{k-1}a_f(n)+p^{2k-1}a_f(n/p^2)\right)q^{n/24}.
  $$
\end{Proposition}

\begin{proof} One way to prove the proposition would be to utilize the
  standard coset representatives of
  $\Gamma_0(N)\backslash\MM_{p^2}(N)$ given by
  $$
    \M{p^2}001, \quad \M pa0p, \quad \M1b0{p^2},\quad
    a=1,\ldots,p-1,~b=0,\ldots,p^2-1,
  $$
  and then apply formulas similar to those in \eqref{equation:
    parabolic 2} in the proof of Lemma \ref{lemma: parabolic prelim}
  below to get the conclusion. Here, however, because it is well-known
  that $f(24\tau)$ is a modular form of half-integral weight on
  $\Gamma_0(576)$ with character $\JS{12}\cdot$ in the sense
  of Shimura, we can actually skip those tedious computations. Indeed,
  from the commutativity of the diagram
  $$
  \begin{diagram}
  \node{\S_{r,s}}\arrow[3]{e,t}{T_{p^2}}\arrow{s,l}{\SM{24}001}
  \node[3]{\S_{r,s}}\arrow{s,r}{\SM{24}001} \\
  \node{S_{k+1/2}(576,\JS{12}\cdot)} \arrow[3]{e,t}{T_{p^2}}
  \node[3]{S_{k+1/2}(576,\JS{12}\cdot)}
  \end{diagram}
  $$
  and the formula given in \cite[Page 450]{Shimura-correspondence} for the
  Hecke operator $T_{p^2}$ on $S_{k+1/2}(576,\JS{12}\cdot)$ in terms
  of Fourier coefficients, we immediately get the conclusion.
\end{proof}

\begin{Remark} The linear operators $[\MM_n(N)]$ can also be defined
  for nonsquare integers $n$ and integers that are not relatively
  prime to $6$, but it turns out that they are actually the zero
  operator. The reason is that for such integers $n$, there are more
  than one elements in $\MM_n(N)$ with the same first component and
  the actions of these elements cancel out each other.
\end{Remark}
\end{subsection}

\begin{subsection}{Shimura's trace formula}

Here we state a trace formula of Shimura \cite{Shimura-trace}, adapted
to our setting. Our description of the trace formula mostly
follows that of Kohnen \cite{Kohnen-newform}.

\begin{Definition} \label{definition: equivalence}
  Let $N$ and $n$ be positive integers. For
  $\gamma\in\MM_n(N)$, we say $\gamma$ is
  \begin{enumerate}
  \item a \emph{scalar} element if $\gamma=\SM a00a$ for
    some integer $a$ (this only happens when $n=1$),
  \item a \emph{parabolic} element if the fixed point of $\gamma$ 
    is a single cusp in $\P^1(\Q)$,
  \item a \emph{hyperbolic} element if the fixed points of $\gamma$ is
    two distinct real numbers, and
  \item an \emph{elliptic} element if the fixed points of $\gamma$ is
    a pair of conjugate complex numbers.
  \end{enumerate}
  Two elements $\gamma_1$ and $\gamma_2$ in $\MM_n(N)$ are
  \emph{equivalent} if
  \begin{enumerate}
  \item $\gamma_1$ and $\gamma_2$ are scalars and $\gamma_1=\gamma_2$,
  \item $\gamma_1$ and $\gamma_2$ are hyperbolic or elliptic and there
    exists an element $\sigma\in\Gamma_0(N)$ such that
    $\sigma\gamma_1\sigma^{-1}=\gamma_2$, or
  \item $\gamma_1$ and $\gamma_2$ are parabolic and there exist
    $\sigma\in\Gamma_0(N)$ and $\alpha$ in the stabilizer subgroup
    inside $\Gamma_0(N)$ of the cusp fixed by $\gamma_2$ such that
    $\sigma\gamma_1\sigma^{-1}=\alpha\gamma_2$.
  \end{enumerate}
\end{Definition}

Now for $\gamma\in\MM_{n^2}(N)$, we define a number $J(\gamma)$ as
follows.
\begin{enumerate}
\item If $\gamma$ is a scalar, then we set
  $$
    J(\gamma)=\frac1{24}\left(k-\frac12\right)[\SL(2,\Z):\Gamma_0(N)].
  $$
\item Assume that $\gamma$ is parabolic with fixed point
  $a/c\in\P^1(\Q)$. Let $\sigma\in\SL(2,\Z)$ be a matrix such that
  $\sigma\infty=a/c$. Then the stabilizer subgroup of $a/c$ inside
  $\Gamma_0(N)$ is generated by $\sigma\SM1w01\sigma^{-1}$ and
  $\SM{-1}00{-1}$, where $w=N/(c^2,N)$ is the width of the cusp
  $a/c$. Now we write 
  $$
    \M1w01^\ast=\left(\M1w01,e^{-2\pi i\mu}\right)
  $$
  with $0\le\mu<1$. If
  $$
    {\sigma^\ast}^{-1}\gamma^\ast\sigma^\ast
    =\left(\pm\M n{nuw}0n,\eta\right),
  $$
  then let
  $$
    J(\gamma)=\begin{cases}
   -\frac1{2\eta}e^{-2\pi iu\mu}(1-2\mu), &\text{if }u\in\Z, \\
   -\frac1{2\eta}e^{-2\pi iu\mu}(1-i\cot\pi u), &\text{if }u\not\in\Z.
   \end{cases}
  $$
\item If $\gamma$ is hyperbolic and the fixed points are not cusps,
  then set $J(\gamma)=0$.
\item Assume that $\gamma$ is hyperbolic fixing (two distinct) cusps.
  Then the eigenvalues of $\gamma$ are two integers $\lambda$ and
  $\lambda'$. We assume that $|\lambda|>|\lambda'|$. Let
  $\left(\begin{smallmatrix}a\\c\end{smallmatrix}\right)$ be an
  eigenvector associated to $\lambda'$ with $a,c\in\Z$ and $(a,c)=1$.
  Find an element $\sigma\in\SL(2,\Z)$ such that
  $\sigma=\SM abcd$. Then
  $\sigma^{-1}\gamma\sigma=\SM{\lambda'}x0{\lambda}$ for some integer
  $x$. If
  $$
    \sigma^\ast\gamma^\ast{\sigma^{-1}}^\ast
   =\left(\M{\lambda'}x0\lambda,\eta\right),
  $$
  then set
  $$
    J(\gamma)=\frac12\left(\eta\left(\frac{\lambda'}\lambda-1\right)
    \right)^{-1}.
  $$
\item Assume that $\gamma=\SM abcd$ is elliptic. Then the eigenvalues
  of $\gamma$ are a pair of conjugate complex numbers $\rho$ and
  $\overline\rho$. We assume that $\sgn\Im\rho=\sgn c$. If
  $$
    \gamma^\ast=\left(\M abcd,u(c\tau+d)^{k+1/2}\right),
  $$
  then we set
  $$
    J(\gamma)=\left(wu\rho^{k-1/2}(\rho-\overline\rho)\right)^{-1},
  $$
  where $w$ denotes the number of elements in $\Gamma_0(N)$ that
  commute with $\gamma$.
\end{enumerate}

Then according to Shimura's formula \cite[Page 273]{Shimura-trace},
the trace of the operator $[\MM_{n^2}(N)]$ on $\S_{r,s}(N)$ is as
follows. (Cf. \cite[Page 49]{Kohnen-newform}.)

\begin{Proposition} \label{proposition: Shimura trace}
  For positive integers $N$ and $n$ with $(n,6N)=1$,
  the trace of the linear operator $[\MM_{n^2}(N)^\ast]$ on $\S_{r,s}(N)$
  is given by
$$
  \tr[\MM_{n^2}(N)^\ast]=\sum_\gamma J(\gamma),
$$
where the sum runs over representatives of equivalence classes as per
Definition \ref{definition: equivalence} and $J(\gamma)$ are defined
as in the paragraph preceding the proposition.
\end{Proposition}

\end{subsection}

\end{section}

\begin{section}{Traces of Hecke operators on $\S_{r,s}(1)$}
\label{section: trace, half-integral weight}

In this section, we will compute the trace of Hecke operators on
$\S_{r,s}(1)$. The contributions of scalar, parabolic, hyperbolic, and
elliptic classes will be determined separately in individual subsections. 

Throughout this section, we write $\MM_{n^2}(1)$ and
$\MM_{n^2}(1)^\ast$ simply as $\MM_{n^2}$ and $\MM_{n^2}^\ast$,
respectively. All equivalences mentioned here refer to the equivalence
relation described in Definition \ref{definition: equivalence}.

\begin{subsection}{Scalar cases}
Since any element $\SM abcd$ in $\MM_{n^2}$ satisfies $(a,b,c,d)=1$.
Scalar elements exist only in $\MM_1^\ast$ and they are
$\left(\pm\SM1001,1\right)$.

\begin{Proposition}
  \label{proposition: scalar}
  The contribution of scalar elements in $\MM_{n^2}$
  to the trace of $[\MM_{n^2}^\ast]$ is
  $$
  \begin{cases}
  \frac1{12}\left(k-\frac12\right), &\text{if }n=1, \\
  0, &\text{else}. \end{cases}
  $$
\end{Proposition}
\end{subsection}

\begin{subsection}{Parabolic cases}

The contribution of the parabolic classes to the trace of $\MM_{n^2}$
is summarized in Proposition \ref{proposition: parabolic}. The proof
is divided into several steps.

\begin{Lemma} \label{lemma: inequivalent parabolic}
  The inequivalent parabolic elements in $\MM_{n^2}$ are
  $$
    \M{n}a0n, \qquad a=1,\ldots,n,~(a,n)=1.
  $$
\end{Lemma}

\begin{proof}
  Since $\SL(2,\Z)$ has only one inequivalent cusp $\infty$, a
  parabolic element is conjugate to
  $$
    \pm\M{n}b0n
  $$
  for some $b$ with $(b,n)=1$. Since the stabilizer subgroup of $\infty$
  inside $\SL(2,\Z)$ is generated by $\SM1101$ and $\SM{-1}00{-1}$, we
  see that the inequivalent parabolic elements are given as in the statement.
\end{proof}



\begin{Lemma} \label{lemma: parabolic prelim}
  Let $a$ be an integer relatively prime to $n$. The
  contribution of the class of $\M{n}a0n$ to the trace of
  $[\MM_{n^2}^\ast]$ on $\S_{r,s}(1)$ is
  $$
    \begin{cases}
    -(r-12)/24, &\text{if }n=a=1, \\ \displaystyle
    -(-i)^{r(n-1)/2}\JS{-a}n\frac{e^{-2\pi i(r\ell+1)a/n}}{1-e^{-2\pi ia/n}},
    &\text{if }n>1. \end{cases}
  $$
  where $\ell=(n^2-1)/24$.
\end{Lemma}

\begin{proof} Assume first that $n=a=1$. We have
  $\SM1101^\ast=\left(\SM1101,e^{2\pi ir/24}\right)$. Then the numbers
  $\mu$, $\eta$, and $u$ in the definition of $J(\gamma)$ in
  Proposition \ref{proposition: Shimura trace} are
  $$
    \mu=\frac{24-r}{24}, \quad \eta=e^{2\pi ir/24}, \quad u=1,
  $$
  respectively. Thus, the contribution to the trace is $-(r-12)/24$.

  Now assume that $n>1$. We have $(n,a)=1$. Let $\alpha$ and $\beta$
  be integers such that $\alpha n+\beta a=1$ and $\beta>0$. We have
  $$
    \M{n}a0n=\M a{-\alpha}n\beta\M{-1}00{-n^2}\M{-n\beta}{-1}10.
  $$
  By Lemma \ref{lemma: eta}, we have
  $$
    \M a{-\alpha}n\beta^\ast=\left(\M a{-\alpha}n\beta,
    \JS ani^{r(1-n)/2}e^{2\pi ir(n(a+\beta)-3)/24}(n\tau+\beta)^{k+1/2}\right)
  $$
  and
  $$
    \M{-n\beta}{-1}10^\ast=\left(\M{-n\beta}{-1}10,e^{2\pi
        ir(-n\beta-3)/24}\tau^{k+1/2}\right).
  $$
  Then
  $$
    \M{-1}00{-n^2}^\ast\M{-n\beta}{-1}10^\ast
   =\left(\M{n\beta}1{-n^2}0,e^{2\pi ir(-n\beta-3)/24}
    (n\tau)^{k+1/2}\right)
  $$
  and
  \begin{equation} \label{equation: parabolic 2}
  \begin{split}
    \M{n}a0n^\ast
  &=\left(\M na0n,\JS ani^{r(1-n)/2}e^{2\pi ir(na-6)/24}
    \left(-\frac1{n\tau}\right)^{k+1/2}(n\tau)^{k+1/2}\right) \\
  &=\left(\M na0n,\JS ane^{2\pi ir(n(a-3)+3)/24}\right).
  \end{split}
  \end{equation}

  Now the stabilizer of $\infty$ is generated by
  $\left(\SM1101,e^{2\pi ir/24}\right)$ and $\left(\SM{-1}00{-1},1\right)$. Thus,
  according to Proposition \ref{proposition: Shimura trace} and
  \eqref{equation: parabolic 2}, the contribution of the class of
  $\SM{n}a0n$ to the trace is
  $$
    -\frac12\JS{a}ne^{-2\pi ir(n(a-3)+3)/24}
    e^{-2\pi ia(24-r)/(24n)}(1-i\cot(\pi a/n)).
  $$
  Simplifying the expression, we get the lemma.
\end{proof}

The proof of Proposition \ref{proposition: parabolic} involves sums of
the form $\sum_{u\le N/j}\chi(u)$ for some Dirichlet character $\chi$.
Here we recall a formula from \cite{SUZ} for such sums. Note that for
a Dirichlet character $\chi$ modulo $M$, the generalized Bernoulli
numbers $B_{m,\chi}$ are defined by the power series
$$
  \sum_{a=1}^M\frac{\chi(a)te^{at}}{e^{Mt}-1}=\sum_{m=0}^\infty
  B_{m,\chi}\frac{t^m}{m!}.
$$
If $\chi$ is the trivial character modulo $1$, then $B_{m,\chi}$ is
just the Bernoulli numbers $B_m$. We have
\begin{equation} \label{equation: B0}
  B_{0,\chi}=\frac1M\sum_{a=1}^M\chi(a)=\begin{cases}
  \phi(M)/M, &\text{if }\chi\text{ is principal}, \\
  0, &\text{else}. \end{cases}
\end{equation}
Also, $B_{1,\chi}=0$ if $\chi$ is an even character. Moreover, if
$\chi$ is an imprimitive Dirichlet character induced from $\chi_1$,
then we have the relation
\begin{equation} \label{equation: Bernoulli relation}
   B_{1,\chi}=B_{1,\chi_1}\sum_{d|M}\mu(d)\chi_1(d)
   =B_{1,\chi_1}\prod_{p|M}(1-\chi_1(p)).
\end{equation}
Also, if $d<0$ is a fundamental discriminant, then we have
\begin{equation} \label{equation: class number formula}
  B_{1,\JS d\cdot}=-H(d),
\end{equation}
where $H(d)$ is the Hurwitz class number.

\begin{Lemma}[{\cite[Equation (6), Page 276]{SUZ}}]
  \label{lemma: SUZ}
  Let $M$ be a positive integer and $\chi$ be a Dirichlet character
  modulo $M$. Let $N>0$ be a multiple of $M$ and $j$ be a positive
  integer relatively prime to $N$. Then we have
  $$
    \sum_{0\le u<N/j}\chi(u)=-B_{1,\chi}
   +\frac{\overline{\chi(j)}}{\phi(j)}
    \sum_{\psi\mod j}\overline{\psi(-N)}B_{1,\chi\psi}(N),
  $$
  where the sum runs over all Dirichlet characters $\psi$ modulo
  $j$ and
  \begin{equation*}
  \begin{split}
    B_{1,\chi\psi}(N)&=B_{0,\chi\psi}N+B_{1,\chi\psi} \\
  &=\begin{cases}
    B_{1,\chi\psi}, &\text{if }\chi\psi\text{ is a nonprincipal
      character modulo }jM, \\
    N\phi(jM)/jM, &\text{if }\chi\psi\text{ is the principal
      character modulo }jM.\end{cases}
  \end{split}
  \end{equation*}
\end{Lemma}

\begin{proof} If $M>1$, then the formula is just Equation (6) of
  \cite{SUZ} on Page 276 with $m=1$. If $M=1$, that is, if $\chi(u)=1$
  for all $u\in\Z$, then we should modify the definition of $\mathcal
  L_\chi(t)$ on Page 274 of \cite{SUZ} to $\mathcal
  L_\chi(t)=\sum_{n=0}^\infty\chi(n)e^{nt}$. Then following the argument
  from Page 274 up to Equation (6) of Page 276 of \cite{SUZ}, we see that
  our formula holds. The details of proof are omitted.
\end{proof}

\begin{Lemma} \label{lemma: parabolic prelim 0}
  Let $n$ be a positive odd integer greater than $1$. If $n\equiv
  1\mod4$, then
  $$
    \sum_{a\mod n}\JS an\frac1{1-e^{2\pi ia/n}}
   =\begin{cases}\phi(n)/2, &\text{if }n\text{ is a square},  \\
    0, &\text{else}. \end{cases}
  $$
  If $n\equiv 3\mod 4$, then
  $$
    \sum_{a\mod n}\JS an\frac1{1-e^{2\pi ia/n}}=i\sqrt nH(-n).
  $$
  Here $H(-n)$ is the Hurwitz class number, i.e., $H(-3)=1/3$,
  $H(-4)=1/2$, and $H(-n)$ is the ideal class number of the quadratic
  order of discriminant $-n$ if $n\neq3,4$.
\end{Lemma}

\begin{proof}
  Assume that $n\equiv 1\mod 4$. Then $\JS\cdot n$ is an even function.
  Let $S$ denote the sum. We have
  $$
    2S=\sum_{a\mod n}\JS an\left(\frac1{1-e^{2\pi
        ia/n}}+\frac1{1-e^{-2\pi ia/n}}\right)
      =\sum_{a\mod n}\JS an.
  $$
  If $n$ is a square, then $\JS\cdot n$ is the principal Dirichlet
  character modulo $n$ and we have $S=\phi(n)/2$. If $n$ is not a square,
  then $\JS\cdot n$ is a nonprincipal Dirichlet character modulo $n$ and
  the sum $S$ vanishes.

  Now assume that $n\equiv 3\mod 4$. Set $z=e^{2\pi i/n}$. We first
  consider the case $n$ is squarefree. For an integer $t$, set
  $$
    S(t)=\sum_{a\mod n}\JS an\frac{z^{ta}}{1-z^a}, \qquad
    T(t)=\sum_{a\mod n}\JS anz^{ta}.
  $$
  The two sums are related by
  $$
    S(t)-S(t+1)=\sum_{a\mod n}\JS an\frac{z^{ta}(1-z^a)}{1-z^a}
      =\sum_{a\mod n}\JS anz^{ta}=T(t).
  $$
  Since $n$ is assumed to be squarefree, we have
  $T(t)=\JS tn i\sqrt n$. It follows that
  \begin{equation*}
  \begin{split}
    \sum_{t=1}^n(S(0)-S(t))&=\sum_{t=1}^n\sum_{a=0}^{t-1}T(a)
   =\sum_{a=0}^{t-1}T(a)(n-a) \\
  &=i\sqrt n\sum_{a=0}^{n-1}\JS an(n-a)=in\sqrt nH(-n).
  \end{split}
  \end{equation*}
  Since
  $$
    \sum_{t=1}^nS(t)=\sum_{a=1}^n\JS an\frac{1+z^a+\cdots+z^{(n-1)a}}{1-z^a}=0,
  $$
  we conclude that $S=S(0)=i\sqrt nH(-n)$ for the case $n$ is
  squarefree.

  Now if $n$ is not squarefree, we write $n$ as $m^2n_0$ with
  squarefree $n_0$. Then using the partial fraction decomposition
  $$
    \frac\ell{1-x^{\ell}}=\sum_{j=0}^{\ell-1}\frac1{1-xe^{2\pi ij/\ell}},
  $$
  we get
  \begin{equation*}
  \begin{split}
    S&=\sum_{d|m}\mu(d)\sum_{a=1}^{n/d}\JS{ad}{n_0}\frac1{1-z^{ad}} \\
     &=\sum_{d|m}\mu(d)\JS d{n_0}\sum_{a=1}^{n_0-1}\JS a{n_0}
       \sum_{j=0}^{n/dn_0-1}\frac1{1-e^{2\pi id(a+jn_0)/n}} \\
     &=\sum_{d|m}\mu(d)\JS d{n_0}\sum_{a=1}^{n_0}\JS a{n_0}
       \frac{n/dn_0}{1-e^{2\pi ia/n_0}} \\
     &=im^2\sqrt{n_0}H(-n_0)\sum_{d|m}\frac{\mu(d)}d\JS d{n_0},
  \end{split}
  \end{equation*}
  where in the last step we use the result for the squarefree
  case computed earlier. Now recall that, for $d>0$ with $d\equiv
  0,3\mod 4$, the Hurwitz class numbers
  $H(-m^2d)$ and $H(-d)$ are related by the formula
  \begin{equation} \label{equation: class number}
  \begin{split}
    H(-m^2d)=H(-d)m\prod_{p|m}\left(1-\frac1p\JS{-d}p\right)
   =H(-d)m\sum_{a|m}\frac{\mu(a)}a\JS{-d}a.
  \end{split}
  \end{equation}
  Therefore,
  $$
    S=im\sqrt{n_0}H(-n)=i\sqrt nH(-n).
  $$
  This completes the proof of the lemma.
\end{proof}

\begin{Lemma} For a positive integer $n>1$ with $(n,6)=1$, let
  $\ell=(n^2-1)/24$. If $n\equiv 1\mod 4$, then
  \begin{equation} \label{equation: parabolic n mod 4=1}
  \begin{split}
  &\frac1{\sqrt n}\sum_{a\mod n}\JS an\frac{e^{2\pi i(r\ell+1)a/n}}
     {1-e^{2\pi ia/n}} \\
  &\qquad=-\frac18\JS{24}n\sum_{u=-3,-4,-8,-24}\JS ur
   \left(1-\JS{un}2\right)\left(1-\JS{un}3\right)H(un).
  \end{split}
  \end{equation}
  If $n\equiv 3\mod 4$, then
  \begin{equation} \label{equation: parabolic n mod 4=3}
  \begin{split}
  &\frac1{i\sqrt n}\sum_{a\mod n}\JS an\frac{e^{2\pi i(r\ell+1)a/n}}
     {1-e^{2\pi ia/n}} \\
  &\qquad=\frac18\JS{24}n\sum_{u=1,8,12,24}\JS ur
   \left(1-\JS{-un}2\right)\left(1-\JS{-un}3\right)H(-un).
  \end{split}
  \end{equation}
\end{Lemma}

\begin{proof} Let $S$ denote the sum in question and for a positive
  integer $\ell$, let $z_\ell$ denote the $\ell$th primitive root of
  unity $e^{2\pi i/\ell}$. Also, write $n$ as $n=m^2n_0$ with
  squarefree $n_0$. We have
  \begin{equation} \label{equation: parabolic 1}
  \begin{split}
    S&=-\sum_{a\mod n}\JS an\frac{1-z_n^{(r\ell+1)a}}{1-z_n^a}
     +\sum_{a\mod n}\JS an\frac1{1-z_n^a} \\
     &=-\sum_{t=0}^{r\ell}\sum_{a\mod n}\JS anz_n^{ta}+
      \sum_{a\mod n}\JS an\frac1{1-z_n^a}=-\sum_{t=0}^{r\ell}T(t)+S',
  \end{split}
  \end{equation}
  say. The sum $S'$ has been evaluated in Lemma \ref{lemma: parabolic
    prelim 0}. We now consider $T(t)$. We have
  \begin{equation*}
  \begin{split}
    T(t)&=\sum_{d|m}\mu(d)\sum_{a\mod n/d}\JS{ad}{n_0}z_n^{tad} \\
    &=\sum_{d|m}\mu(d)\JS d{n_0}\sum_{a\mod n_0}\JS a{n_0}z_n^{tad}
      \sum_{b=0}^{m^2/d-1}z_{m^2/d}^{bt}.
  \end{split}
  \end{equation*}
  The inner sum is $0$ if $(m^2/d)\nmid t$. Then
  \begin{equation} \label{equation: T(t)}
  \begin{split}
    T(t)&=m^2\sum_{d|m,(m^2/d)|t}\frac{\mu(d)}d\JS d{n_0}\sum_{a\mod n_0}\JS a{n_0}
    z_{n_0}^{at/(m^2/d)} \\
  &=\epsilon m^2\sqrt{n_0}\sum_{d|m,(m^2/d)|t}\frac{\mu(d)}d\JS d{n_0}
    \JS{t/(m^2/d)}{n_0},
  \end{split}
  \end{equation}
  where
  $$
    \epsilon=\begin{cases}1&\text{if }n\equiv 1\mod 4, \\
    i, &\text{if }n\equiv 3\mod 4. \end{cases}
  $$
  From \eqref{equation: T(t)}, we obtain
  \begin{equation*}
  \begin{split}
    \sum_{t=0}^{r\ell}T(t)=\epsilon m^2\sqrt{n_0}\sum_{d|m}
    \frac{\mu(d)}d\JS d{n_0}\sum_{0\le u\le r\ell/(m^2/d)}\JS{u}{n_0}.
  \end{split}
  \end{equation*}
  Now $\ell=(n^2-1)/24$ and $r$ is assumed to be in the range $0<r<24$.
  Therefore, the sum above can also be written as
  \begin{equation} \label{equation: sum T(t)}
    \sum_{t=0}^{r\ell}T(t)=\epsilon m^2\sqrt{n_0}\sum_{d|m}
    \frac{\mu(d)}d\JS d{n_0}\sum_{0\le u<(rn^2d/m^2)/24}\JS u{n_0}.
  \end{equation}
  Now we apply Lemma \ref{lemma: SUZ} to
  \eqref{equation: sum T(t)} with $\chi=\chi_{n_0}=\JS\cdot{n_0}$,
  $M=n_0$, $N=rn^2d/m^2$, and $j=24$. We consider the three cases
  \begin{enumerate}
  \item $n_0=1$,
  \item $n\equiv 1\mod 4$ and $n_0\neq 1$,
  \item $n\equiv 3\mod 4$,
  \end{enumerate}
  separately.

  When $n_0=1$, an application of Lemma \ref{lemma: SUZ} yields
  \begin{equation} \label{equation: sum T(t) 2}
  \begin{split}
    \sum_{t=0}^{r\ell}T(t)=m^2\sum_{d|m}\frac{\mu(d)}d
    \left(-B_1+\frac18\sum_{\psi\mod 24}\overline{\psi(-rm^2d)}
    B_{1,\psi}(rm^2d)\right) \\
   =\frac12m\phi(m)+\frac18m^2\sum_{d|m}\frac{\mu(d)}d
    \sum_{\psi\mod 24}\overline{\psi(-rm^2d)}B_{1,\psi}(rm^2d).
  \end{split}
  \end{equation}
  If we let $\chi_0$ denote the principal Dirichlet
  character modulo $24$, then the Dirichlet characters $\psi$ modulo
  $24$ are given by
  \begin{equation} \label{equation: definition psiu}
    \psi_u=\chi_0(\cdot)\JS u\cdot, \qquad u=1,8,12,24,-3,-4,-8,-24.
  \end{equation}
  By \eqref{equation: B0}, \eqref{equation: Bernoulli relation}, and
  \eqref{equation: class number formula}, we have
  \begin{equation} \label{equation: n0=1 temp}
    B_{1,\psi_u}(rm^2d)=\begin{cases}
    rm^2d\phi(24)/24=rm^2d/3, &\text{if }u=1, \\
    0, &\text{if }u=8,12,24, \\
    -\left(1-\JS u2\right)\left(1-\JS u3\right)H(u),
    &\text{if }u=-3,-4,-8,-24.
    \end{cases}
  \end{equation}
  The contribution from the character $\psi_1$ to \eqref{equation: sum
    T(t) 2} is
  \begin{equation} \label{equation: psi1}
    \frac{rm^4}{24}\sum_{d|m}\mu(d)=0,
  \end{equation}
  since $m=\sqrt n$ is assumed to be greater than $1$. The
  contributions from $\psi_u$, for $u=-3,-4,-8,-24$, are
  \begin{equation} \label{equation: psiu}
  \begin{split}
    &\frac{m^2}8\JS ur\left(1-\JS u2\right)\left(1-\JS u3\right)H(u)
     \sum_{d|m}\frac{\mu(d)}d\JS ud \\
    &\qquad\qquad=\frac m8\JS ur\left(1-\JS u2\right)\left(1-\JS
      u3\right)H(un),
  \end{split}
  \end{equation}
  where we have utilized formula \eqref{equation: class number}.
  Combining \eqref{equation: parabolic 1}, \eqref{equation: sum T(t)
    2}, \eqref{equation: n0=1 temp}, \eqref{equation: psi1}, and
  \eqref{equation: psiu} and using the formula for $S'$ given in Lemma
  \ref{lemma: parabolic prelim 0}, we get formula \eqref{equation:
    parabolic n mod 4=1} for the case $n=m^2$.

  Now let us consider the case $n\equiv 1\mod 4$ but not a perfect
  square. That is, $n=m^2n_0$ with squarefree $n_0\neq 1$ and
  $n_0\equiv 1\mod 4$. In this case, we have
  $B_{1,\JS\cdot{n_0}}=B_{1,\JS\cdot{n_0}\psi_u}=0$
  for $u=1,8,12,24$, where $\psi_u$ are defined by \eqref{equation:
    definition psiu}. Then an application of Lemma \ref{lemma:
    SUZ} to \eqref{equation: sum T(t)} yields
  \begin{equation*}
  \begin{split}
    \sum_{t=0}^{r\ell}T(t)
  &=\frac{m^2\sqrt{n_0}}8\sum_{d|m}\frac{\mu(d)}d\JS d{n_0}
    \JS{24}{n_0}\sum_{u=-3,-4,-8,-24}\psi_u(-rn^2d/m^2)
    B_{1,\JS\cdot{n_0}\psi_u} \\
  &=-\frac{m\sqrt{n}}8\JS{24}n\sum_{d|m}\frac{\mu(d)}d\JS{n_0}d
    \sum_{u=-3,-4,-8,-24}\JS u{rd}B_{1,\JS{n_0u}\cdot} \\
  &\qquad\qquad\qquad\times
   \left(1-\JS{un_0}2\right)\left(1-\JS{un_0}3\right),
  \end{split}
  \end{equation*}
  where we have used \eqref{equation: Bernoulli relation}. Then by
  \eqref{equation: class number formula} and \eqref{equation: class
    number}, we get
  \begin{equation*}
  \begin{split}
  &\sum_{t=0}^{r\ell}T(t)=\frac{\sqrt n}8\JS{24}n
   \sum_{u=-3,-4,-8,-24}\JS ur\left(1-\JS{un}2
   \right)\left(1-\JS{un}3\right)H(un).
  \end{split}
  \end{equation*}
  This gives the evaluation of the sum of $T(t)$ in \eqref{equation:
    parabolic 1}. The term $S'$ in \eqref{equation: parabolic 1} is
  shown to be $0$ in Lemma \ref{lemma: parabolic prelim 0}. This
  establishes \eqref{equation: parabolic n mod 4=1} for the case $n$
  is not a square.

  Now assume that $n\equiv 3\mod 4$. Then $B_{1,\JS\cdot n\psi_u}=0$
  for $u=-3,-4,-8,-24$ and an application of Lemma
  \ref{lemma: SUZ} to \eqref{equation: sum T(t)} gives us
  \begin{equation*}
  \begin{split}
    \sum_{t=0}^{r\ell}T(t)&=im\sqrt n\sum_{d|m}\frac{\mu(d)}d\JS d{n_0}
      \bigg(-B_{1,\JS\cdot{n_0}} \\
  &\qquad\qquad+\frac18\JS{24}{n_0}
      \sum_{u=1,8,12,24}\psi_u(-rn^2d/m^2)B_{1,\JS\cdot{n_0}\psi_u}\bigg)
  \end{split}
  \end{equation*}
  Using \eqref{equation: Bernoulli relation}, \eqref{equation: class
    number formula}, and \eqref{equation: class number} again, we get
  \begin{equation*}
  \begin{split}
    \sum_{t=0}^{r\ell}T(t)&=i\sqrt n\Bigg(H(-n) \\
  &\qquad-\frac18\JS{24}n
    \sum_{u=1,8,12,24}\JS{u}r\left(1-\JS{-un}2\right)
    \left(1-\JS{-un}3\right)H(-un)\Bigg).
  \end{split}
  \end{equation*}
  Combining this, \eqref{equation: parabolic 1}, and Lemma \ref{lemma:
    parabolic prelim 0}, we arrive at the claimed formula. This
  completes the proof of the lemma.
\end{proof}

\begin{Proposition} \label{proposition: parabolic}
  The total contribution of the parabolic classes of $\MM_{n^2}$ to
  the trace of the linear operator $[\MM_{n^2}^\ast]$ on $\S_{r,s}(1)$
  is
  \begin{equation*}
  \begin{split}
  \frac{\sqrt n}8\JS{12}n\sum_{e=1,2,3,6}\JS{-4e}r
  \left(1-\JS{-en}3\right)(H(-4en)-H(-en))
  \end{split}
  \end{equation*}
  where, for a negative integer $-d$, we let
  $H(-d)$ denote the Hurwitz class number of the imaginary quadratic
  order of discriminant $-d$. (If $-d$ is not a discriminant, then set
  $H(-d)=0$.)
\end{Proposition}

\begin{proof} We first consider the case $n=1$. By Lemmas \ref{lemma:
    inequivalent parabolic} and \ref{lemma: parabolic prelim}, we find
  that the total contribution is $-(r-12)/24$. We then verify case by
  case that
  $$
   -\frac{r-12}{24}=\frac18\sum_{e=1,2,3,6}\JS{-4e}r
    \left(1-\JS{-e}3\right)(H(-4e)-H(-e)).
  $$

  We next consider the cases $n>1$. Again, using Lemmas \ref{lemma:
    inequivalent parabolic} and \ref{lemma: parabolic prelim}, we find
  that the total contribution is
  $$
    -(-i)^{r(n-1)/2}\sum_{a=1}^n\JS{-a}n
    \frac{e^{-2\pi i(r\ell+1)a/n}}{1-e^{-2\pi ia/n}}
  $$
  Now for $n\equiv 1\mod 4$, we have
  $$
    (-i)^{r(n-1)/2}=\JS8n.
  $$
  Thus, by \eqref{equation: parabolic n mod 4=1}, the contribution to
  the trace is
  \begin{equation*}
  \begin{split}
  &\frac{\sqrt n}8\JS{12}n\sum_{u=-3,-4,-8,-24}\JS ur
   \left(1-\JS{un}2\right)\left(1-\JS{un}3\right)H(un).
  \end{split}
  \end{equation*}
  Since $H(-12n)=H(-3n)\left(2-\JS{-3n}2\right)$,
  $H(-n)=H(-2n)=H(-6n)=0$, the formula above is equal to that in
  the statement of the proposition.

  Now assume that $n\equiv 3\mod 4$. We have
  $$
    i(-i)^{r(n-1)/2}=-\JS{-4}r\JS8n.
  $$
  Then from Lemma \ref{lemma: parabolic prelim}, \eqref{equation:
    parabolic n mod 4=3}, and $H(-4n)=H(-n)\left(2-\JS{-n}2\right)$,
  we get the claimed formula. This proves the proposition.
\end{proof}

\end{subsection}

\begin{subsection}{Hyperbolic cases}

\begin{Lemma} \label{lemma: hyperbolic}
  A complete set of representatives of inequivalent hyperbolic
  elements in $\MM_{n^2}$ whose fixed points are cusps is
  $$
    \left\{\pm\M a{b+ma}0d: \,b=1,\ldots,h,\,(b,h)=1,\,m=1,\ldots,(d-a)/h
      \text{ with }h=(a,d)\right\},
  $$
  where in the set we let $a$ and $d$ run over all integers satisfying
  $ad=n^2$ and $0<a<d$.
\end{Lemma}

\begin{proof} Let $M$ be a hyperbolic element in $\MM_{n^2}$ whose
  fixed points are cusps. Then the eigenvalues of $M$ are two integers
  $a$ and $d$ satisfying $ad=n^2$. Without loss of generality, we
  assume that $|d|>|a|>0$. Let $\alpha$ be the cusp corresponding to the
  eigenvalue $a$ and choose an element $\gamma\in\SL(2,\Z)$ such that
  $\sigma\infty=\alpha$. Then
  $$
    \sigma M\sigma^{-1}=\M ab0d
  $$
  for some integer $b$. The integer $b$ must satisfy $(a,b,d)=1$. In
  other words, a hyperbolic element in $\MM_{n^2}$ whose fixed points
  are cusps is conjugate to a matrix of the form $\SM ab0d$ with
  $ad=n^2$, $|d|>|a|>0$, and $(b,(a,d))=1$. It is clear that if two such
  matrices $\SM ab0d$ and $\SM{a'}{b'}0{d'}$
  are conjugate, then we must have $a=a'$ and $d=d'$. Furthermore, by
  considering the eigenvector
  $\left(\begin{smallmatrix}1\\0\end{smallmatrix}\right)$, it is easy
  to see that two matrices $\SM ab0d$ and $\SM a{b'}0d$ are conjugate
  if and only they are conjugate by $\SM1m01$ for some $m\in\Z$, that
  is, if and only if $(d-a)|(b'-b)$. With these informations, it is
  straightforward to verify that the set in the lemma is a complete
  set of representatives. This proves the lemma.
\end{proof}

\begin{Proposition}
  \label{proposition: hyperbolic}
  The total contribution of hyperbolic elements in
  $\MM_{n^2}$ to the trace is $0$.
\end{Proposition}

\begin{proof} By Shimura's formula, the contribution from hyperbolic
  elements whose fixed points are not cusp is $0$. Furthermore, by
  Lemma \ref{lemma: hyperbolic}, a complete set of inequivalent
  hyperbolic elements whose fixed points are cusps is
  $$
    \left\{\pm\M a{b+ma}0d: \,b=1,\ldots,h,\,(b,h)=1,\,m=1,\ldots,(d-a)/h
      \text{ with }h=(a,d)\right\},
  $$
  where in the set we let $a$ and $d$ run over all integers satisfying
  $ad=n^2$ and $0<a<d$. Now we have
  $$
    \M a{b+ma}0d=\M ab0d\M1m01.
  $$
  Thus, if the contribution from the conjugacy class $\SM ab0d$ is
  $A$, then the contribution from the class $\SM a{b+ma}0d$ is
  $e^{2\pi irm/24}A$. Since $24|(d-a)/h$, as $m$ runs from $1$ to
  $(d-a)/h$, the contributions from the classes $\SM a{b+ma}0d$ cancel
  out each other. We therefore conclude that the total contribution
  from hyperbolic elements is $0$.
\end{proof}
\end{subsection}

\begin{subsection}{Elliptic cases}
\begin{subsubsection}{\bf Quadratic forms}
Assume that $\gamma=\SM abcd\in\MM_{n^2}$ is elliptic. Then $t=a+d$,
$f=(b,c,d-a)$, and $\sgn c$ are invariants under conjugation by
elements of $\SL(2,\Z)$. This can be seen from the one-to-one
correspondence between the set
$$
  \Gamma_{n,t,f}=\left\{\M abcd\in\MM_{n^2}:~a+d=t,~f=(b,c,d-a),
  c>0\right\}
$$
and the set $\QQ_{(t^2-4n^2)/f^2}$ of all primitive positive definite
quadratic forms of discriminant $(t^2-4n^2)/f^2$, where
$$
  \QQ_D=\{Ax^2+Bxy+Cy^2:~B^2-4AC=D,~A>0\}.
$$
and the correspondence is given by
\begin{equation} \label{equation: correspondence}
  \M abcd\longleftrightarrow
  \frac1f(cx^2+(d-a)xy-by^2).
\end{equation}
Elements $\gamma$ of $\SL(2,\Z)$ act on $\Gamma_{n,t,f}$ by
conjugation and on $\QQ_{(t^2-4n^2)/f^2}$ by change of variable $\left(
\begin{smallmatrix}x\\y\end{smallmatrix}\right)\mapsto\gamma\left(
\begin{smallmatrix}x\\y\end{smallmatrix}\right)$. Moreover, the two
group actions are compatible with respect to the correspondence above.

\begin{Lemma} \label{lemma: elliptic prelim}
  The total contribution of elliptic elements in $\MM_{n^2}^\ast$ to
  the trace of the linear operator $[\MM_{n^2}^\ast]$ on $\S_{r,s}(1)$
  is
  $$
    2\sum_{t,f}\sum_{\gamma\in\Gamma_{n,t,f}/\SL(2,\Z)}J(\gamma),
  $$
  where the outer sum runs over all integers $t$ with $t^2<4n^2$ and
  all positive integers $f$ such that $(t^2-4n^2)/f^2\equiv 0,1\mod 4$,
  the inner sum runs over all class representatives of
  $\Gamma_{n,t,f}$ under conjugation by $\SL(2,\Z)$, and $J(\gamma)$
  is defined as in Definition \ref{definition: equivalence}.
\end{Lemma}

\begin{proof} From the discussion preceding the lemma, we easily see
  that every elliptic element in $\MM_{n^2}$ falls in precisely one of
  $\Gamma_{n,t,f}$ and $-\Gamma_{n,t,f}$. Also, it is clear that if
  $\gamma_1,\ldots,\gamma_m$ are class representatives of
  $\Gamma_{n,t,f}$, then $-\gamma_1,\ldots,-\gamma_m$ are class
  representatives of $-\Gamma_{n,t,f}$. We now show that
  $J(-\gamma)=J(\gamma)$. Then the lemma follows.

  Assume that $\gamma=\SM abcd\in\Gamma_{n,t,f}$ and
  $$
    \gamma^\ast=\left(\M abcd,u(c\tau+d)^{k+1/2}\right).
  $$
  Then by Proposition \ref{proposition: Shimura trace}, the
  contribution of the class of $\gamma$ to the trace is
  $$
    J(\gamma)=\frac{\rho^{1/2-k}}{wu(\rho-\overline\rho)},
  $$
  where $\rho=(t+\sqrt{t^2-4n^2})/2$. Now we have
  \begin{equation*}
  \begin{split}
    (-\gamma)^\ast&=\left(\M{-a}{-b}{-c}{-d}, u(c\tau+d)^{k+1/2}\right)\\
  &=\left(\M{-a}{-b}{-c}{-d},ue^{2\pi i(2k+1)/4}(-c\tau-d)^{k+1/2}\right).
  \end{split}
  \end{equation*}
  Thus,
  $$
    J(-\gamma)=\frac{(-\rho)^{1/2-k}}
    {wue^{2\pi i(2k+1)/4}(-\rho+\overline\rho)}
   =\frac{e^{2\pi i(2k-1)/4}\rho^{1/2-k}}
    {wue^{2\pi i(2k+1)/4}(-\rho+\overline\rho)}=J(\gamma).
  $$
  This proves the lemma.
\end{proof}

\begin{Remark} \label{remark: (t,f)} Since for $\SM abcd\in\MM_{n^2}$,
  we have $(a,b,c,d)=1$, the integers $n$, $t$, and $f$ need to
  satisfy the following conditions.
  \begin{enumerate}
  \item The common divisor $(n,t,f)$ must be $1$.
  \item The integer $(t,f)$ can only be $1$ or $2$.
  Moreover, if $(t,f)=2$, then because $(t^2-4n^2)/4$ is a
  discriminant, we must have $2|t$, but $4\nmid t$.
  \item The integer $(f,t+2n,t-2n)$ can only be $1$, $2$, or $4$
  and if the latter two cases occur, then $2|t$, but $4\nmid t$.
  \end{enumerate}
\end{Remark}

We first choose a suitable representative $\gamma$ in each conjugacy
class in $\Gamma_{n,t,f}$ and determine $\gamma^\ast$.

\begin{Lemma} \label{lemma: class representatives}
  Each conjugacy class of $\Gamma_{n,t,f}$ contains an
  element $\SM abcd$ such that $c=fp$ for some prime $p>4n$ and
  $(c,d)=1$.

  Moreover, in the case $2|f$ (which occurs only when $2\|t$), we may
  further require that $d$ satisfies the following congruences
  $$
    d\equiv\begin{cases}
    t'\mod 8, &\text{if }2\|f, \\
    t'\mod 16, &\text{if }4\|f, \\
    t'\mod 8, &\text{if }8\|f\text{ and }(t^2-4n^2)/64
     \text{ is even}, \\
    t'+4\mod 8, &\text{if }8\|f\text{ and }(t^2-4n^2)/64
     \text{ is odd}, \\
    t'\mod 8, &\text{if }16|f,\end{cases}  
  $$
  where $t'=t/2$.
\end{Lemma}

\begin{proof} It is well-known that a primitive positive definite
  quadratic form represents infinity many primes, which implies that
  each class of quadratic forms in $\QQ_{(t^2-4n^2)/f^2}$
  contains elements of the form $px^2+uxy+vy^2$ for infinitely many
  primes $p$. The matrix corresponding to this form under
  \eqref{equation: correspondence} is
  $$
    \M{(t-fu)/2}{-fv}{fp}{(t+fu)/2}.
  $$
  Now if $p>4n$, then the relation
  $$
    f^2(u^2-4pv)=t^2-4n^2
  $$
  implies that $p\nmid(t+fu)$ and hence $(fp,(t+fu)/2)=(f,(t+fu)/2)$.
  By Remark \ref{remark: (t,f)}, this can only be $1$ or $2$. However,
  because the determinant of the matrix is the odd integer $n^2$,
  $(fp,(t+fu)/2)$ cannot be even. That is, $(fp,(t+fu)/2)=1$. This
  proves the first part of the statement.

  Now assume that $2\|f$. From
  \begin{equation} \label{equation: 2|f tmp}
    \M1{-u}01\M abcd\M1u01=\M{a-cu}{-cu^2+u(a-d)+b}c{d+cu}
  \end{equation}
  it is clear that we can further assume that $d\equiv t'\mod 8$.

  Assume that $4\|f$. Since $32|(t^2-4n^2)$ and $(t^2-4n^2)/16$ is a
  discriminant, we must have $64|(t^2-4n^2)$. Then the equality
  $t^2-4n^2=(d-a)^2+4bc$ shows that $64|(d-a)^2$, i.e., $8|(d-a)$. Then
  $$
    d=\frac12(t+(d-a))\equiv t'\mod 4.
  $$
  By \eqref{equation: 2|f tmp} again, we may assume that $d\equiv
  t'\mod 16$.

  Assume that $8\|f$. We have $64|(t^2-4n^2)$ and $8|b,c$. Then
  $(d-a)^2=t^2-4n^2-4bc\equiv (t^2-4n^2)\mod 256$. If $(t^2-4n^2)/64$ is
  even, then $d\equiv a\mod 16$ and $d=(t+d-a)/2\equiv t'\mod 8$. If
  $(t^2-4n^2)/64$ is odd, then $d\equiv a+8\mod 16$ and $d\equiv
  t'+4\mod 8$.

  Finally, if $16|f$, by a computation similar to the case $8\|f$, we
  find $d\equiv t'\mod 8$. This proves the lemma.
\end{proof}

For our purpose, we also need to recall some properties of genus
characters of integral binary quadratic forms. Let $D<0$ be a discriminant,
i.e., $D\equiv 0,1\mod 4$. For each odd prime divisor $p$ of $D$, one
can associate to $\QQ_D$ a character by
$$
  \chi:Ax^2+Bxy+Cy^2\longmapsto\JS A{p}=\JS{p^\ast}A, \qquad
  p^\ast=(-1)^{(p-1)/2}p.
$$
For $D\equiv 0\mod 4$, there may also exist characters of the form
$$
  Ax^2+Bxy+Cy^2\longmapsto\JS{u}A, \qquad u\in\{-4,8,-8\},
$$
depending on the residue of $D/4$ modulo $8$. Any product of these
characters is called a \emph{genus character}. All genus characters
can be written as
$$
  \chi_{D_1}:Ax^2+Bxy+Cy^2\longmapsto\JS{D_1}A
$$
for some (positive or negative) divisor $D_1$ of $D$ such that $D_1$
and $D/D_1$ are both discriminants and vice versa. General properties
of genus characters can be found in \cite[Chapter 1]{Cox}. Here we
only quote some properties relevant to our calculation of traces.

\begin{Lemma} \label{lemma: genus character} Let $D$ and $D_1$ be as
  above and $\chi_{D_1}:\QQ_D\to\{\pm 1\}$ be a genus character. Then
  we have the following properties.
  \begin{enumerate}
  \item The value of $\chi_{D_1}$ for a quadratic form
    $Ax^2+Bxy+Cy^2\in\QQ_D$ depends only on the genus in which the
    quadratic form lies. In particular, every quadratic form in the
    same class takes the same value of $\chi_{D_1}$.
  \item $\chi_{D_1}$ is a trivial character (i.e., mapping every
    quadratic form to $1$) if and only if $D_1$ or $D/D_1$ is a
    square.
  \item The sum of $\chi_{D_1}$ over a complete set of class
    representatives of $\QQ_D/\SL(2,\Z)$ is
  $$
    \sum_{Ax^2+Bxy+Cy^2\in\QQ_D/\SL(2,\Z)}\JS{D_1}A
   =\begin{cases}
    h(D), &\text{if }D_1\text{ or }D/D_1\text{ is a square}, \\
    0, &\text{else},
    \end{cases}
  $$
  where $h(D)$ is the number of classes in $\QQ_D$. 
  \end{enumerate}
\end{Lemma}

Finally, the following formula frequently occurs in our computation.

\begin{Lemma} \label{lemma: class sum}
  Let $D$ be the discriminant of an imaginary quadratic order.
  Let $u$ be a positive integer. Then we have
  $$
    \sum_{f|u}\JS DfH\left(\frac{u^2}{f^2}D\right)=uH(D).
  $$
\end{Lemma}

\begin{proof} We have
  \begin{equation*}
  \begin{split}
    \sum_{f|u}\JS DfH\left(\frac{u^2}{f^2}D\right)
  &=H(D)\sum_{f|u}\JS Df\frac uf\sum_{m|(u/f)}\frac{\mu(m)}m\JS Dm \\
  &=H(D)\sum_{n|u}\JS Dn\frac un\sum_{m|n}\mu(m)
  \end{split}
  \end{equation*}
  The sum $\sum_{m|n}\mu(m)$ is nonzero only when $n=1$. From this, we
  get the claimed formula.
\end{proof}

\end{subsubsection}

\begin{subsubsection}{\bf Preliminary calculation and notations}
\begin{Lemma} \label{lemma: eta multiplier}
  If $\SM abcd\in\MM_{n^2}$ satisfies $(c,d)=1$ and $c>0$, then
  $$
    \M abcd^\ast=\left(\M abcd,n^{-k-1/2}\epsilon(a,b,c,d)^r
  (c\tau+d)^{k+1/2}\right),
  $$
  where
  $$
    \epsilon(a,b,c,d)=\begin{cases}
    \JS dc e^{2\pi i\left(bd(1-c^2)+c(a+d-3)\right)/24},
      &\text{if }c\text{ is odd}, \\
    \JS cde^{2\pi i\left(ac(1-d^2)+d(b-c+3)-3\right)/24},
      &\text{if }c\text{ is even}. \end{cases}
  $$
\end{Lemma}

\begin{proof} Let $\alpha,\beta\in\Z$ be integers such that $\alpha
  d+\beta c=1$. We have
  $$
    \M abcd=\M 1{a\beta+b\alpha}01\M{n^2}001\M\alpha{-\beta}cd.
  $$
  Now
  $$
    \M\alpha{-\beta}cd^\ast=\left(\M\alpha{-\beta}cd,
    \epsilon(\alpha,-\beta,c,d)^r(c\tau+d)^{k+1/2}\right)
  $$
  and
  $$
    \M 1{a\beta+b\alpha}01^\ast=\left(\M1{a\beta+b\alpha}01,
    e^{2\pi ir(a\beta+b\alpha)/24}\right).
  $$
  Then from
  $$
    \M1{a\beta+b\alpha}01\M\alpha{-\beta}cd
   =\M{a+(1-n^2)\alpha}{b+(n^2-1)\beta}cd
  $$
  we deduce that
  \begin{equation*}
  \begin{split}
    e^{2\pi ir(a\beta+b\alpha)/24}\epsilon(\alpha,-\beta,c,d)^r
  &=\epsilon(a+(1-n^2)\alpha,b+(n^2-1)\beta,c,d)^r \\
  &=\epsilon(a,b,c,d)^r.
  \end{split}
  \end{equation*}
  Then the lemma follows.
\end{proof}


All the $24$th roots of unity can be expressed in terms of Jacobi
symbols, which we give in the next lemma for future reference.

\begin{Lemma} \label{lemma: roots of unity} If $(t,24)=1$, then
  $$
    e^{2\pi it/24}=\frac{\sqrt2}4\JS8{t}+\frac{\sqrt6}4\JS{24}{t}
   -\frac{i\sqrt 2}4\JS{-8}t+\frac{i\sqrt 6}4\JS{-24}{t}.
  $$
  If $(t,12)=1$, then
  $$
    e^{2\pi it/12}=\frac{\sqrt 3}2\JS{12}t+\frac i2\JS{-4}t.
  $$
  If $(t,8)=1$, then
  $$
    e^{2\pi it/8}=\frac{\sqrt 2}2\JS8t+\frac{i\sqrt 2}2\JS{-8}t.
  $$
  If $(t,6)=1$, then
  $$
    e^{2\pi it/6}=\frac12+\frac{i\sqrt 3}2\JS{-3}t.
  $$
  If $(t,4)=1$, then
  $$
    e^{2\pi it/4}=i\JS{-4}t.
  $$
  If $(t,3)=1$, then
  $$
    e^{2\pi it/3}=-\frac12+\frac{i\sqrt 3}2\JS{-3}t.
  $$
\end{Lemma}

Let class representatives
  $\SM abcd$ of $\Gamma_{n,t,f}$ be chosen as in Lemma \ref{lemma:
    class representatives}. Then by Proposition \ref{proposition:
    Shimura trace} and Lemma \ref{lemma: eta multiplier}, the
  contribution of the class of $\SM abcd$ to the trace is
  \begin{equation} \label{equation: J(gamma)}
    J(\gamma)=\frac{n^{k+1/2}}{w_{n,t,f}}
    \epsilon(a,b,c,d)^{-r}\frac{\rho^{1/2-k}}{\rho-\overline\rho},
  \end{equation}
  where
  \begin{equation} \label{equation: wntf}
    w_{n,t,f}=\begin{cases}
    6, &\text{if }(t^2-4n^2)/f^2=-3, \\
    4, &\text{if }(t^2-4n^2)/f^2=-4, \\
    2, &\text{else},
    \end{cases}
  \end{equation}
  is the number of elements in $\SL(2,\Z)$ commuting with $\gamma$,
  $$
    \epsilon(a,b,c,d)=\JS dce^{-2\pi ic/8}e^{2\pi i(bd(1-c^2)+ct)/24}
  $$
  and $\rho=(t+\sqrt{t^2-4n^2})/2$. In view of Lemma \ref{lemma: roots
    of unity}, sums of the form $\sum\JS dc\JS ec$,
  $e=\pm1,\pm2,\pm3,\pm6$, will appear frequently in our computation.
  So we first compute such sums.

\begin{Notation} \label{notation: LM}
  For $e=\pm1,\pm2,\pm3,\pm6$ and nonnegative integers $\ell$, we set
  \begin{equation} \label{equation: L(n,t)}
  \begin{split}
    L_{e,\ell}(n,t)&:=\frac{\rho^{1/2-k}}{\rho-\overline\rho}
      \Bigg(\sum_{f:2^\ell\|f,(n,t,f)=1}\frac1{w_{n,t,f}}
      \sum_{\SM abcd\in\Gamma_{n,t,f}/\SL(2,\Z)}\JS d{c'}\JS e{c'} \\
    &\qquad-3\sum_{f:2^\ell\|f,3|f,(n,t,f)=1}\frac1{w_{n,t,f}}
      \sum_{\SM abcd\in\Gamma_{n,t,f}/\SL(2,\Z)}\JS d{c'}\JS e{c'}\Bigg),
  \end{split}
  \end{equation}
  where in the inner sums, the representatives $\SM abcd$ are chosen
  according to Lemma \ref{lemma: class representatives},
  $c'=c/2^\ell$ (that is, $c'$ is the odd part of $c$), and
  $\rho=(t+\sqrt{t^2-4n^2})/2$.

  Moreover, for $e=1,2,3,6$ and integers $u$, we set
  \begin{equation} \label{equation: M(n,u)}
  \begin{split}
    M_{e,\ell}(n,u)&:=\frac{(\tau/\sqrt e)^{1-2k}}{\tau-\overline\tau}
      \Bigg(\sum_{g:2^\ell\|g,(n,u,g)=1}
      H\left(\frac{e^2u^2-4en}{g^2}\right) \\
    &\qquad\qquad-3\sum_{g:2^\ell\|g,3|g,(n,u,g)=1}
      H\left(\frac{e^2u^2-4en}{g^2}\right)\Bigg),
  \end{split}
  \end{equation}
  and
  \begin{equation} \label{equation: M'(n,u)}
  \begin{split}
    M'_{e,\ell}(n,u)&:=
      \frac{(\overline\tau/\sqrt e)^{1-2k}}{\tau-\overline\tau}
      \Bigg(\sum_{g:2^\ell\|g,(n,u,g)=1}
     H\left(\frac{e^2u^2-4en}{g^2}\right) \\
    &\qquad\qquad-3\sum_{g:2^\ell\|g,3|g,(n,u,g)=1}
      H\left(\frac{e^2u^2-4en}{g^2}\right)\Bigg),
  \end{split}
  \end{equation}
  where $\tau=(eu+\sqrt{e^2u^2-4en})/2$ and $g$ runs over all positive
  integers satisfying the given conditions such that $(e^2u^2-4en)/g^2$
  is a discriminant.
\end{Notation}

\begin{Lemma} \label{lemma: t,e conditions}
  Let $n$ be a positive integer relatively prime to $6$ and $t$ be an
  integer satisfying $t^2<4n^2$. Let
  $$
    e\in\begin{cases}
    \{\pm1,\pm3\}, &\text{if }(t,24)=1, \\
    \{\pm1,\pm2,\pm3,\pm6\}, &\text{if }(t,24)=2, \\
    \{\pm1\}, &\text{if }(t,24)=3, \\
    \{\pm2,\pm6\}, &\text{if }(t,24)=4\text{ or }(t,24)=8, \\
    \{\pm1,\pm2\}, &\text{if }(t,24)=6, \\
    \{\pm2\}, &\text{if }(t,24)=12\text{ or }(t,24)=24.
    \end{cases}
  $$
  If $e(t+2n)$ is a discriminant, i.e., if $e(t+2n)\equiv0,1\mod 4$,
  then there exists a rational number $s$ such that $t^2-4n^2$
  decomposes into a product of two discriminants $s^2e(t+2n)$ and
  $(t-2n)/(es^2)$.
\end{Lemma}

\begin{proof} Here we only consider the
  case $(t,24)=2$. When $e=\pm1$, we have $4|(t+2n),(t-2n)$ and the
  statement holds obviously. When $e=\pm 2$, we can decompose
  $t^2-4n^2$ as $2(t+2n)\cdot(t-2n)/2$ or $(t+2n)/2\cdot2(t-2n)$
  according to whether $8|(t-2n)$ or $8|(t+2n)$. When $e=\pm3$, the
  decomposition is $3(t+2n)\cdot(t-2n)/3$ or $(t+2n)/3\cdot3(t-2n)$
  according to whether $3|(t-2n)$ or $3|(t+2n)$. When $e=\pm 6$, the
  rational number $s$ can be one of $1$, $1/2$, $1/3$, or $1/6$,
  depending one whether $8|(t+2n)$ and whether $3|(t+2n)$.
\end{proof}

We now evaluate $L_{e,\ell}(n,t)$. The computation mostly follows that
in \cite{Kohnen-newform}.

\begin{Lemma} \label{lemma: elliptic main lemma}
  Let $n$ be a positive integer relatively prime to $6$. For
  $e\in\{1,2,3,6\}$ and integers $u$, let
  $$
    \mu_{e}(n,u)=\begin{cases}
    1, &\text{if }3\nmid u, \\ \displaystyle
    1+\JS{en}3, &\text{if }3|u.
    \end{cases}
  $$
  Then we have the following formulas.

  \begin{enumerate}
  \item Let $t$ and $e$ be integers satisfying the assumptions in Lemma
  \ref{lemma: t,e conditions}.

  If $e(t+2n)\equiv0,1\mod 4$, then
  \begin{equation} \label{equation: L(n,t) 1}
  \begin{split}
   L_{e,0}(n,t)=\frac12\begin{cases}
   0, &\text{if }e(t+2n)\text{ is not a square}, \\
   \mu_{e}(n,u)M_{e,0}(n,u),
   &\text{if }t+2n=eu^2. \end{cases}
  \end{split}
  \end{equation}
  where $u$ is the positive square root of $(t+2n)/e$.

  Also, if $-e(t+2n)\equiv0,1\mod 4$, then
  \begin{equation} \label{equation: L(n,t) 1-}
  \begin{split}
    L_{-e,0}(n,t)=\frac i2\JS{-4}r\begin{cases}
      0, &\text{if }e(2n-t)\text{ is not a square}, \\
      \mu_{e}(n,u)M'_{e,0}(n,u)
   &\text{if }2n-t=eu^2, \end{cases}
  \end{split}
  \end{equation}
  where $u$ is the positive square root of $(2n-t)/e$.
  \item Assume that $2\|t$ and $e\in\{\pm1,\pm3\}$. Then
  $$
    L_{e,1}(n,t)=\frac12 L_{e,0}(n,t).
  $$
  \item Assume that $2\|t$, $e\in\{1,3\}$, and $\ell\ge 2$.
  Then
  $$
    L_{e,\ell}(n,t)=\frac12\begin{cases}
    0, &\text{if }e(t+2n)\text{ is not a square}, \\
    \mu_{e}(n,u)M_{e,\ell-1}(n,u)/2,
   &\text{if }t+2n=eu^2\text{ and }2\|u, \\
    \mu_{e}(n,u)M_{e,1}(n,u)/2^{\ell-1},
   &\text{if }t+2n=eu^2\text{ and }2^{\ell-1}\|u, \\
    \mu_{e}(n,u)M_{e,0}(n,u)/2^\ell,
   &\text{if }t+2n=eu^2\text{ and }2^\ell|u. \end{cases}
  $$
  and
  $$
    L_{-e,\ell}(n,t)=\frac i2\JS{-4}r\begin{cases}
    0, &\text{if }e(2n-t)\text{ is not a square}, \\
    \mu_{e}(n,u)M'_{e,\ell-1}(n,u)/2,
   &\text{if }2n-t=eu^2\text{ and }2\|u, \\
    \mu_{e}(n,u)M'_{e,1}(n,u)/2^{\ell-1},
   &\text{if }2n-t=eu^2\text{ and }2^{\ell-1}\|u, \\
    \mu_{e}(n,u)M'_{e,0}(n,u)/2^\ell,
   &\text{if }2n-t=eu^2\text{ and }2^\ell|u, \end{cases}
  $$
  \end{enumerate}
\end{Lemma}

\begin{proof} We first prove \eqref{equation: L(n,t) 1}. Note that in
  the sum defining $L_{e,0}(n,t)$, the integers $f$ are always odd
  and according to the choice of representatives given in Lemma
  \ref{lemma: class representatives}, we have $c'=c$. Then from Remark 
  \ref{remark: (t,f)}, we know that $(f,t+2n,t-2n)=1$. Thus, if
  $f_1=(f,t+2n)$ and $f_2=(f,t-2n)$, then $f=f_1f_2$, $f_1^2|(t+2n)$ and
  $f_2^2|(t-2n)$. Write $\JS dc$ as
  $$
    \JS dc=\JS d{c/f}\JS d{f_1}\JS d{f_2}.
  $$
  Since
  $$
    \JS d{c/f}\JS{t+2n}{c/f}=\JS{ad+2nd+d^2}{c/f}
   =\JS{n^2+2nd+d^2}{c/f}=1,
  $$
  we have
  $$
    \JS d{c/f}=\JS{(t+2n)/f_1^2}{c/f}.
  $$
  By the same token, we also have
  $$
    \JS d{f_1}=\JS{(t-2n)/f_2^2}{f_1}, \qquad
    \JS d{f_2}=\JS{(t+2n)/f_1^2}{f_2},
  $$
  and hence
  \begin{equation} \label{equation: JS dc}
    \JS dc\JS ec=\JS{e(t+2n)/f_1^2}{c/f}\JS{e(t-2n)/f_2^2}{f_1}
    \JS{e(t+2n)/f_1^2}{f_2}.
  \end{equation}
  Now $e(t+2n)$ is assumed to be a discriminant. By Lemma \ref{lemma:
    t,e conditions}, we can decompose
  $t^2-4n^2$ into a product $es^2(t+2n)\cdot(t-2n)/es^2$ of two
  discriminants. So by Lemma \ref{lemma: genus character},
  \begin{equation} \label{equation: temp 1}
  \begin{split}
    &\frac1{w_{n,t,f}}\sum_{\SM abcd\in\Gamma_{n,t,f}/\SL(2,\Z)}
    \JS{e(t+2n)/f_1^2}{c/f} \\
  &\qquad=\begin{cases}
    H((t^2-4n^2)/f^2)/2, &\text{if }e(t+2n)=(eu)^2\text{ is a square}, \\
    0, &\text{else}. \end{cases}
  \end{split}
  \end{equation}
  In the case $e(t+2n)=(eu)^2$, we have
  \begin{equation} \label{equation: rho 1}
    \rho=\frac{t+\sqrt{t^2-4n^2}}2
   =\left(\frac{\sqrt{e(t+2n)}+\sqrt{e(t-2n)}}{2\sqrt e}\right)^2
   =\left(\frac\tau{\sqrt e}\right)^2
  \end{equation}
  and
  \begin{equation} \label{equation: rho 2}
    \rho-\overline\rho=\sqrt{t^2-4n^2}=u\sqrt{e^2u^2-4en}
 =u(\tau-\overline\tau),
  \end{equation}
  where $\tau=(eu+\sqrt\Delta)/2$ with $\Delta=e^2u^2-4en$. Then for
  the first sum $\sum_f$ defining $L_{e,0}(n,t)$ in \eqref{equation:
    L(n,t)}, we have
  \begin{equation} \label{equation: ell=0}
  \begin{split}
   &\sum_{(f,2)=1,(n,t,f)=1}\frac1{w_{n,t,f}}\sum_\gamma\JS dc\JS ec \\
   &\qquad=\frac12\sum_{(f,2)=1,(n,t,f)=1}\JS{e(t+2n)/f_1^2}{f_2}
    \JS{e(t-2n)/f_2^2}{f_1}H\left(\frac{t^2-4n^2}{f^2}\right) \\
   &\qquad=\frac12\sum_{(f_2,2)=1,(n,t,f_2)=1}\sum_{f_1|u,(f_1,2)=1,(n,t,f_1)=1}
    \JS{\Delta/f_2^2}{f_1}H\left(\frac{u^2}{f_1^2}
    \frac\Delta{f_2^2}\right).
  \end{split}
  \end{equation}
  The inner sum running over $f_1$ is the same as
  \begin{equation} \label{equation: temp 6}
    \sum_{f_1|u}\JS{\Delta/f_2^2}{f_1}H\left(\frac{u^2}{f_1^2}
    \frac\Delta{f_2^2}\right)
  \end{equation}
  since if $2|f_1$ or $(n,t,f_1)>1$ then $\JS{\Delta/f_2^2}{f_1}=0$.
  Then by Lemma \ref{lemma: class sum}, the first sum in
  \eqref{equation: L(n,t)} is equal to
  \begin{equation} \label{equation: L(n,t) sum 1}
    \sum_{(f,2)=1,(n,t,f)=1}\frac1{w_{n,t,f}}\sum_\gamma\JS dc\JS ec
   =\frac u2\sum_{(f_2,2)=1,(n,t,f_2)=1}H\left(\frac\Delta{f_2^2}\right).
  \end{equation}

  For the second sum in \eqref{equation: L(n,t)}, we have two cases
  $$
    \begin{cases}
    3|f_1,3\nmid f_2, &\text{if }3|u, \\
    3|f_2,3\nmid f_1, &\text{if }3\nmid u.
    \end{cases}
  $$
  If $3\nmid u$, then a computation similar to \eqref{equation:
    ell=0}--\eqref{equation: L(n,t) sum 1} yields
  \begin{equation} \label{equation: L(n,t) sum 2-1}
    \sum_{(f,2)=1,3|f,(n,t,f)=1}\frac1{w_{n,t,f}}\sum_\gamma\JS dc\JS
    ec=\frac u2\sum_{(f_2,2)=1,3|f_2,(n,t,f_2)=1}H\JS\Delta{f_2^2}.
  \end{equation}
  If $3|u$, then
  \begin{equation*}
  \begin{split}
    &\sum_{(f,2)=1,3|f,(n,t,f)=1}\frac1{w_{n,t,f}}\sum_\gamma\JS dc\JS
    ec \\
    &\qquad=\frac12\sum_{(f,6)=3,(n,t,f)=1}\JS{e(t+2n)/f_1^2}{f_2}
     \JS{e(t-2n)/f_2^2}{f_1}H\JS{t^2-4n}{f^2} \\
    &\qquad=\frac12\sum_{(f_2,2)=1,(n,t,f_2)=1}\sum_{f_1|u,(f_1,6)=3,(n,t,f)=1}
     \JS{\Delta/f_2^2}{f_1}H\left(\frac{u^2}{f_1^2}
     \frac{\Delta}{f_2^2}\right).
  \end{split}
  \end{equation*}
  By Lemma \ref{lemma: class sum} the inner sum is equal to
  $$
    \sum_{h|(u/3),(h,2)=1,(n,t,h)=1}\JS{\Delta/f_2^2}{3h}
   H\left(\frac{(u/3)^2}{f_1^2}\frac\Delta{f_2^2}\right)
   =\frac u3\JS{\Delta}3H\JS\Delta{f_2^2}
  $$
  and we have
  \begin{equation} \label{equation: L(n,t) sum 2-2}
    \sum_{(f,2)=1,3|f,(n,t,f)=1}\frac1{w_{n,t,f}}\sum_\gamma\JS dc\JS
    ec=\frac
    u6\JS\Delta3\sum_{(f_2,2)=1,(n,t,f_2)=1}H\JS\Delta{f_2^2}.
  \end{equation}
  Combining \eqref{equation: JS dc}--\eqref{equation: rho 2},
  \eqref{equation: L(n,t) sum 1}, \eqref{equation: L(n,t) sum 2-1},
  and \eqref{equation: L(n,t) sum 2-2}, we get the formula
  \eqref{equation: L(n,t) 1}.
  
  The proof of \eqref{equation: L(n,t) 1-} follows the same line of
  calculation. Equation \eqref{equation: JS dc}
  continues to hold when we replace $e$ by $-e$. By Lemma \ref{lemma:
    genus character}, \eqref{equation: temp 1} also remains valid,
  provided that the condition that $e(t+2n)$ is a square is change to
  that $e(2n-t)$ is a square. The computations in \eqref{equation: rho 2},
  \eqref{equation: L(n,t) sum 1}, \eqref{equation: L(n,t) sum 2-1},
  and  \eqref{equation: L(n,t) sum 2-2} are almost the same.
  The only significant difference lies at \eqref{equation: rho 1}.
  Instead of \eqref{equation: rho 1}, we have
  $$
    \frac{t+\sqrt{t^2-4n^2}}2
   =\left(i\frac{\sqrt{e(2n-t)}-\sqrt{-e(t+2n)}}{2\sqrt e}\right)^2
   =\left(i\frac{\overline\tau}{\sqrt e}\right)^2
  $$
  and
  \begin{equation} \label{equation: tau bar}
  \begin{split}
    \left(\frac{t+\sqrt{t^2-4n^2}}2\right)^{1/2-k}
   =i^{1-2k}\left(\frac{\overline\tau}{\sqrt e}\right)^{1-2k}
   =i\JS{-4}r\left(\frac{\overline\tau}{\sqrt e}\right)^{1-2k},
  \end{split}
  \end{equation}
  where $\tau=(eu+\sqrt\Delta)/2$ with $\Delta=e^2u^2-4en$. Here in
  the last step, we have used the assumption $k=(r-1)/2+s$. This
  establishes \eqref{equation: L(n,t) 1-}.

  We now prove the second part of the lemma. Here we only consider the
  case $e=1,3$. The integers $f$ in the sum defining $L_{e,1}(n,t)$
  satisfies $2\|f$. Let $f'=f/2$ and set $f_1=(f',t+2n)$ and
  $f_2=(f',t-2n)$. Then as before, we have $f'=f_1f_2$,
  $f_1^2|(t+2n)$, $f_2^2|(t-2n)$, and \eqref{equation: JS dc} remains
  valid if we replace $c$ by $c'$ and $f$ by $f'$. For the first sum
  in $L_{e,\ell}(n,t)$, the computation in
  \eqref{equation: temp 1}--\eqref{equation: rho 2} remain unchanged.
  Now instead of \eqref{equation: ell=0}, we have
  \begin{equation*}
  \begin{split}
   &\sum_{2\|f,(n,t,f)=1}\frac1{w_{n,t,f}}\sum_\gamma\JS dc\JS ec \\
   &\qquad=\frac12\sum_{2\|f,(n,t,f)=1}\JS{e(t+2n)/f_1^2}{f_2}
     \JS{e(t-2n)/f_2^2}{f_1}H\left(\frac{t^2-4n^2}{f^2}\right) \\
   &\qquad=\frac12\sum_{(f_2,2)=1,(n,t,f_2)=1}\sum_{f_1|u,(f_1,2)=1,(n,t,f_1)=1}
    \JS{\Delta/f_2^2}{f_1}H\left(\frac{(u/2)^2}{f_1^2}
    \frac\Delta{f_2^2}\right).
  \end{split}
  \end{equation*}
  Now since $4|(\Delta/f_2^2)=(e^2u^2-4en)/f_2^2$, we have
  $$
    H\left(\frac{(u/2)^2}{f_1^2}\frac{\Delta}{f_2^2}\right)
   =\frac12H\left(\frac{u^2}{f_1^2}\frac{\Delta}{f_2^2}\right).
  $$
  Following the remaining computation, we easily see that the first
  sum in $L_{e,1}(n,t)$ is equal to one half of that of
  $L_{e,0}(n,t)$. Similarly, the second sum in $L_{e,1}(n,t)$ is equal
  to one half of that of $L_{e,0}(n,t)$. This proves the second part
  of the lemma.

  The proof of the third part is also similar. Again, we only consider
  the case $e>0$. The computation for the first sum in
  $L_{e,\ell}(n,t)$ differs from that for the first sum in
  $L_{e,0}(n,t)$ mainly at \eqref{equation: ell=0}. In the 
  case $2\|u$, instead of \eqref{equation: ell=0}, we have
  \begin{equation*}
  \begin{split}
   &\sum_{2^\ell\|f,(n,t,f)=1}\frac1{w_{n,t,f}}\sum_\gamma\JS dc\JS ec \\
   &\qquad=\frac12\sum_{(f_2,2)=1,(n,t,f_2)=1}\sum_{f_1|u,(f_1,2)=1,(n,t,f_1)=1}
    \JS{\Delta/f_2^2}{f_1}H\left(\frac{(u/2)^2}{f_1^2}
    \frac\Delta{(2^{\ell-1}f_2)^2}\right),
  \end{split}
  \end{equation*}
  where $\Delta=e^2u^2-4en$. Then, by Lemma
  \ref{lemma: class sum}, it is equal to
  $$
    \sum_{2^\ell\|f,(n,t,f)=1}\frac1{w_{n,t,f}}\sum_\gamma\JS dc\JS ec
   =\frac u4\sum_{(f_2,2)=1,(n,t,f_2)=1}
   H\left(\frac{\Delta}{(2^{\ell-1}f_2)^2}\right),
  $$
  The computation for the case $2^{\ell-1}\|u$ is almost the same. In
  this case, we have
  $$
    \sum_{2^\ell\|f,(n,t,f)=1}\frac1{w_{n,t,f}}\sum_\gamma\JS dc\JS ec
   =\frac u{2^\ell}\sum_{(f_2,2)=1,(n,t,f_2)=1}
   H\left(\frac{\Delta}{(2f_2)^2}\right),
  $$
  Now if $2^\ell|u$, then instead of \eqref{equation: ell=0}, we have
  \begin{equation*}
  \begin{split}
   &\sum_{2^\ell\|f,(n,t,f)=1}\frac1{w_{n,t,f}}\sum_\gamma\JS dc\JS ec \\
   &\qquad=\frac12\sum_{(f_2,2)=1,(n,t,f_2)=1}\sum_{f_1|u,(f_1,2)=1,(n,t,f_1)=1}
    \JS{\Delta/f_2^2}{f_1}H\left(\frac{(u/2^\ell)^2}{f_1^2}
    \frac\Delta{f_2^2}\right).
  \end{split}
  \end{equation*}
  Since $4|(\Delta/f_2^2)$, we have
  $$
   H\left(\frac{(u/2^\ell)^2}{f_1^2}
    \frac\Delta{f_2^2}\right)=\frac1{2^\ell}
   H\left(\frac{u^2}{f_1^2}\frac{\Delta}{f_2^2}\right)
  $$
  and by Lemma \ref{lemma: class sum},
  \begin{equation*}
  \begin{split}
    \sum_{2^\ell\|f,(n,t,f)=1}\frac1{w_{n,t,f}}\sum_\gamma\JS dc\JS ec
   =\frac u{2^{\ell+1}}
    \sum_{(f_2,2)=1,(n,t,f_2)=1}H\JS\Delta{f_2^2}.
  \end{split}
  \end{equation*}
  Together with \eqref{equation: rho 1} and \eqref{equation: rho 2},
  this completes the computation for the first sum in
  $L_{e,\ell}(n,t)$. The computation for the second sum is analogous 
  and is skipped.
\end{proof}

We now utilize Lemma \ref{lemma: elliptic main lemma} to compute the
contribution of $\Gamma_{n,t,f}$ to the trace of $[\MM_{n^2}]$. This
will be done according to the greatest common divisor of $t$ and $24$.
To summarize our computation, we fix the following notations.

\begin{Notation} \label{notation: ABCD}
  Given a positive integer $n$ relatively prime to $6$,
  $e\in\{1,2,3,6\}$, and an integer $u$ with $e^2u^2<4en$, we let
  \begin{equation} \label{equation: P}
    P_k(e,n,u)=\frac{\tau^{2k-1}-\overline\tau^{2k-1}}{\tau-\overline\tau},
    \qquad\tau=\frac{eu+\sqrt{e^2u^2-4en}}2.
  \end{equation}
  For nonnegative integers $\ell$ and $m$, we set
  \begin{equation*}
  \begin{split}
    A_{\ell,m}(n)&=\JS{12}n\sum_{3\nmid u}\left(
      \sum_{g}H\JS{\Delta}{g^2}-3\sum_{3|g}H\JS\Delta{g^2}\right)
      P_k(1,n,u) \\
    &\qquad+\JS{12}n\sum_{3|u}\left(1-\JS{\Delta}3\right)
    \sum_gH\left(\frac{\Delta}{g^2}\right)P_k(1,n,u),
  \end{split}
  \end{equation*}
  where $\Delta=u^2-4n$, the outer sums run over all integers $u$ satisfying
  $$
    u^2<4n,\qquad 2^\ell\|u,
  $$
  and the given conditions, and the inner sums run
  over all positive integers $g$ such that
  $$
    \Delta/g^2\equiv 0,1\mod 4, \qquad 2^m\|g, \qquad (n,u,g)=1,
  $$
  and the specified conditions are met.
  We also set
  $$
    A_{\ell,m}^\ast(n)=\sum_{\ell\le j<\infty}A_{j,m}(n).
  $$
  Similarly, we define
  \begin{equation*}
  \begin{split}
    B_{\ell}(n)&=\frac1{2^{k-1}}\JS{12}n\JS8r\sum_{3\nmid u}\left(\sum_{g}
    H\JS\Delta{g^2}-3\sum_{3|g}H\JS\Delta{g^2}\right)P_k(2,n,u) \\
    &\qquad+\frac1{2^{k-1}}\JS{12}n\JS8r
    \sum_{3|u}\left(1-\JS{\Delta}{3}\right)\sum_gH\JS\Delta{g^2}P_k(2,n,u),
  \end{split}
  \end{equation*}
  where $\Delta=4u^2-8n$, the outer sums run over all integers $u$ such that
  $$
    4u^2<8n, \qquad 2^\ell\|u,
  $$
  and the inner sums run over all positive integers $g$ such that
  $$
    \Delta/g^2\equiv 0,1\mod 4, \qquad (n,u,g)=1.
  $$
  Also, we set
  $$
    B_{\ell}^\ast(n)=\sum_{\ell\le j<\infty}B_{\ell}(n).
  $$
  Likewise, define
  $$
    C_{\ell,m}(n)=\frac1{3^{k-1}}\JS{12}n\JS{12}r\sum_{9u^2<12n,2^\ell\|u}
    \sum_{g}H\JS\Delta{g^2}P_k(3,n,u),
  $$
  where $\Delta=9u^2-12n$, and the inner
  sum runs over all positive integers $g$ with
  $$
    \Delta/g^2\equiv0,1\mod 4, \qquad 2^m\|g, \qquad
    (n,u,g)=1,
  $$
  and let
  $$
    C_{\ell,m}^\ast(n)=\sum_{\ell\le j<\infty}C_{j,m}(n).
  $$
  Finally, we let
  $$
    D_{\ell}(n)=\frac1{6^{k-1}}\JS{12}n\JS{24}r
    \sum_{36u^2<24n,2^\ell\|u}\sum_gH\JS\Delta{g^2}P_k(6,n,u),
 $$
  where $\Delta=36u^2-24n$, and the inner sum runs over all positive
  integer $g$ such that
  $$
    \Delta/g^2\equiv0,1\mod 4, \qquad (n,u,g)=1.
  $$
  Also, let
  $$
    D_\ell^\ast(n)=\sum_{\ell\le j<\infty}D_j(n).
  $$
\end{Notation}

Here we give alternative expressions for $A_{\ell,m}(n)$ and
$B_\ell(n)$, which will be used in a later section.

\begin{Lemma} \label{lemma: alternative AB}
For a discriminant $\Delta$ of an imaginary quadratic order, we let
$\Delta_0$ denote the discriminant of the field $\Q(\sqrt\Delta)$. Then
we have
$$
  A_{\ell,m}(n)=\beta(n)\JS{12}n\sum_u\left(1-\JS{\Delta_0}3\right)
  \sum_{g,3\nmid\Delta/(\Delta_0g^2)}H\JS{\Delta}{g^2}P_k(1,n,u),
$$
where
$$
  \beta(n)=\begin{cases}
  1, &\text{if }n\equiv 1\mod 3, \\
  1/2, &\text{if }n\equiv 2\mod 3, \end{cases}
$$
the double sum runs over the same $u$ and $g$ as in the
definition of $A_{\ell,m}$ satisfying the additional condition that $3$ does
not divide $\Delta/(\Delta_0g^2)$, where $\Delta=u^2-4n$. Analogously,
we also have
$$
  B_\ell(n)=\frac{\gamma(n)}{2^{k-1}}\JS{12}n\JS8r\sum_u\left(1-\JS{\Delta_0}3\right)
  \sum_{g,3\nmid\Delta/(\Delta_0g^2)}H\JS\Delta{g^2}P_k(2,n,u),
$$
where
$$
  \gamma(n)=\begin{cases}
  1/2, &\text{if }n\equiv 1\mod 3, \\
  1, &\text{if }n\equiv 2\mod 3. \end{cases}
$$
\end{Lemma}

\begin{proof} Here we prove only the case $A_{\ell,m}(n)$. Consider
  the case $n\equiv 1\mod 3$ first. Assume that $3\nmid u$. We have
  $\Delta\equiv 0\mod 3$. If $3\|\Delta$, then $3|\Delta_0$ and the
  sum $\sum_{3|g}$ is empty, which yields
  $$
    \sum_{g}H\JS\Delta{g^2}-\sum_{3|g}H\JS\Delta{g^2}
   =\left(1-\JS{\Delta_0}3\right)\sum_gH\JS{\Delta}{g^2}
  $$
  If $9|\Delta$, we have
  \begin{equation*}
  \begin{split}
   &\sum_gH\JS\Delta{g^2}-3\sum_{3|g}H\JS{\Delta}{g^2}
   =\sum_hH(h^2\Delta_0)-3\sum_{3|h}H\JS{h^2\Delta_0}9 \\
  &\qquad\qquad=\sum_{(h,3)=1}H(h^2\Delta_0)+\sum_{3|h}
   \left(H(h^2\Delta_0)-3H\JS{h^2\Delta_0}9\right)\\
  &\qquad\qquad=\sum_{(h,3)=1}H(h^2\Delta_0)-\sum_{3|h}\JS{h^2\Delta_0/9}3
    H\JS{h^2\Delta_0}9 \\
  &\qquad\qquad=\sum_{(h,3)=1}H(h^2\Delta_0)-\sum_{(h,3)=1}\JS{h^2\Delta_0}3H(h^2\Delta_0) \\
  &\qquad\qquad=\left(1-\JS{\Delta_0}3\right)\sum_{(h,3)=1}H(h^2\Delta_0).
  \end{split}
  \end{equation*}
  Either way, we find that the first double sum in the definition of
  $A_{\ell,m}(n)$ is equal to
  $$
    \sum_{3\nmid
      u}\left(1-\JS{\Delta_0}3\right)\sum_{g,3\nmid\Delta/(g^2\Delta_0)}
    H\JS{\Delta}{g^2}P_k(1,n,u).
  $$
  Now if $3|u$, then $\Delta$ is not a multiple of $3$ and hence
  $$
    \JS\Delta3=\JS{\Delta_0}3.
  $$
  There is nothing to do in these cases. This proves the case $n\equiv
  1\mod 3$.

  Now assume that $n\equiv 2\mod 3$. We have $3\nmid\Delta$. Thus, the
  sum $\sum_{3|g}$ is empty and
  $$
    1-\JS{u^2-4n}3=\begin{cases}
    2, &\text{if }3\nmid u, \\
    0, &\text{if }3|u. \end{cases}
  $$
  That is, the second sum in the definition of $A_{\ell,m}(n)$
  vanishes. The assertion follows.
\end{proof}

\end{subsubsection}

\begin{subsubsection}{Case $(t,24)=1$}
\begin{Lemma} \label{lemma: (t,24)=1 summary}
  For the cases $(t,24)=1$, we have
  \begin{equation*}
  \begin{split}
   \sum_{(t,24)=1}\sum_f\sum_{\gamma\in\Gamma_{n,t,f}/\SL(2,\Z)}J(\gamma)
 =-\frac{n^{3/2-k}}{8}\begin{cases}
   \left(A_{0,0}(n)+C_{0,0}(n)\right), &\text{if }n\equiv 1\mod 3, \\
   C_{0,0}(n), &\text{if }n\equiv 2\mod 3. \end{cases}
  \end{split}
  \end{equation*}
  Here the outermost sums run over all integers $t$ satisfying
  $t^2<4n^2$ and $(t,24)=1$, the middle sums run over positive
  integers $f$ such that $(t^2-4n^2)/f^2$ is a discriminant, and the
  functions $A_{0,0}(n)$ and $C_{0,0}(n)$ are defined as in Notation
  \ref{notation: ABCD}.
\end{Lemma}

\begin{proof} Since $(t,24)=1$, the integers $f$ are always odd. Let
  class representatives $\SM abcd$ of $\Gamma_{n,t,f}/\SL(2,\Z)$ be
  chosen as per Lemma \ref{lemma: class representatives}. By Lemma
  \ref{lemma: eta multiplier}, we have
  $$
    \M abcd^\ast=\left(\M abcd,
      n^{-k-1/2}\epsilon(a,b,c,d)^r(c\tau+d)^{k+1/2}\right),
  $$
  where
  $$
      \epsilon(a,b,c,d)=\JS dce^{2\pi i(bd(1-c^2)+c(a+d-3))/24}
  $$
  If $3\nmid f$, then $(c,6)=1$ and $bd(1-c^2)\equiv 0\mod 24$. If
  $3|f$, then $3|b$ and $8|(1-c^2)$. Either way, we get
  $\epsilon(a,b,c,d)=\JS dce^{2\pi ict/24}e^{-2\pi ic/8}$
  and
  \begin{equation} \label{equation: (t,24)=1 1}
    J(\gamma)=\frac{n^{k+1/2}}{w_{n,t,f}}\JS dce^{2\pi irc/8}e^{-2\pi irct/24}
    \frac{\rho^{1/2-k}}{\rho-\overline\rho},
  \end{equation}
  where $\rho=(t+\sqrt{t^2-4n^2})/2$.

  Now for $3\nmid f$, we have, by the formulas in Lemma
  \ref{lemma: roots of unity},
  \begin{equation*}
  \begin{split}
  &e^{2\pi irc/8}e^{-2\pi irct/24} \\
  &\ =\frac14\left(\JS8{rc}+i\JS{-8}{rc}\right)
    \left(\JS8{-rct}+\sqrt3\JS{24}{-rct}-i\JS{-8}{-rct}
   +i\sqrt3\JS{-24}{-rct}\right) \\
  &\ =\frac14\left(\JS8t-\JS{-8}t\right)
   +\frac i4\JS{-4}{rc}\left(\JS8t+\JS{-8}t\right) \\
  &\qquad+\frac{\sqrt 3}4\JS{12}{rc}\left(\JS{24}t+\JS{-24}t\right)
   +\frac{i\sqrt3}4\JS{-3}{rc}\left(\JS{24}t-\JS{-24}t\right) \\
  &\ =\frac12\JS8t\left(\delta_3(t)+i\delta_1(t)\JS{-4}{rc}\right)
    +\frac{\sqrt3}2\JS{24}t\JS{12}{rc}\left(\delta_1(t)+i\delta_3(t)\JS{-4}{rc}\right),
  \end{split}
  \end{equation*}
  where
  \begin{equation} \label{equation: delta}
    \delta_j(t)=\begin{cases}
    1, &\text{if }t\equiv j\mod 4, \\
    0, &\text{else}. \end{cases}
  \end{equation}
  For $3|f$, we have $3|c$ and
  \begin{equation*}
  \begin{split}
    e^{2\pi irc/8}e^{-2\pi irct/24}
  &=\frac12\left(\JS8{rc}+i\JS{-8}{rc}\right)
    \left(\JS8{rct/3}-i\JS{-8}{rct/3}\right) \\
  &=\frac12\left(-\JS8t+\JS{-8}t\right)
   -\frac i2\JS{-4}{rc}\left(\JS8t+\JS{-8}t\right) \\
  &=-\delta_3(t)\JS8t-i\delta_1(t)\JS8t\JS{-4}{rc}.
  \end{split}
  \end{equation*}
  Thus,
  \begin{equation} \label{equation: (t,24)=1 3}
  \begin{split}
  &\sum_{(t,24)=1}\sum_{f}\sum_{\gamma\in\Gamma_{n,t,f}/\SL(2,\Z)}J(\gamma) \\
  &=\frac{n^{k+1/2}}2\sum_{(t,24)=1}\Bigg(
   \delta_3(t)\JS8tL_{1,0}(n,t)+i\delta_1(t)\JS{-4}r\JS8tL_{-1,0}(n,t) \\
  &\qquad+\sqrt 3\delta_1(t)\JS{12}r\JS{24}tL_{3,0}(n,t)
  +i\sqrt3\delta_3(t)\JS{-3}r\JS{24}tL_{-3,0}(n,t)\Bigg),
  \end{split}
  \end{equation}
  where $L_{e,\ell}(n,t)$ are defined by \eqref{equation:
    L(n,t)}. (Note that the second sums in the definition of
  $L_{3,0}(n,t)$ and $L_{-3,0}(n,t)$ are empty.)

  For the term $\delta_3(t)L_{1,0}(n,t)$, $\delta_3(t)\neq 0$ implies
  that $t+2n$ is a discriminant and Lemma \ref{lemma:
    elliptic main lemma} applies. That is, $\delta_3(t)L_{1,0}(n,t)$
  is nonzero only when $t+2n=u^2$ for some (odd) integer $u$. If this
  situation occurs, then $t+2n\equiv 1\mod 8$ and
  \begin{equation} \label{equation: 8t=-4n}
    \JS8t=(-1)^{(t^2-1)/8}=(-1)^{n(n-1)/2}=\JS{-4}n.
  \end{equation}
  Also observe that since $(t,24)=1$, if $n\equiv 2\mod 3$, any
  integer $u$ such that $u^2=t+2n$ must be divisible by $3$. These
  informations, together with Lemma \ref{lemma: elliptic main lemma},
  yield
  \begin{equation} \label{equation: (t,24)=1 sum 1}
  \begin{split}
   &\sum_{(t,24)=1}\delta_3(t)\JS8tL_{1,0}(n,t)
    =\frac{\lambda_1(n)}2\JS{-4}n\sum_{(u,6)=1}M_{1,0}(n,u)\\
   &\qquad\qquad+\frac12\JS{-4}n\sum_{(u,6)=3}
    \left(1+\JS{n}3\right)M_{1,0}(n,u),
  \end{split}
  \end{equation}
  where the sums run over all \emph{positive} integers $u$ such
  that $u^2<4n$ satisfying the given conditions,
  $M_{e,\ell}(n,u)$ are defined by \eqref{equation: M(n,u)} and for $j=1,2$, 
  \begin{equation} \label{equation: lambda}
    \lambda_j(n)=\begin{cases}
    1, &\text{if }n\equiv j\mod 3, \\
    0, &\text{else}. \end{cases}
  \end{equation}

  Likewise, for the term $\delta_1(t)L_{-1,0}(n,t)$ in
  \eqref{equation: (t,24)=1 3}, $\delta_1(t)\neq 0$ implies that
  $-(t+2n)$ is a discriminant. Therefore, assuming $\delta_1(t)=1$,
  Lemma \ref{lemma: elliptic main lemma} shows that $L_{-1,0}(n,t)$
  is nonzero only when $2n-t=u^2$ is a square. In such cases,
  \eqref{equation: 8t=-4n} continues to hold. Then Lemma \ref{lemma:
    elliptic main lemma} yields
  \begin{equation} \label{equation: (t,24)=1 sum 2}
  \begin{split}
   &\sum_{(t,24)=1}\delta_1(t)\JS8tL_{-1,0}(n,t)
    =i\frac{\lambda_1(n)}2\JS{-4}{rn}\sum_{(u,6)=1}M'_{1,0}(n,u) \\
   &\qquad\qquad+\frac i2\JS{-4}{rn}\sum_{(u,6)=3}
    \left(1+\JS{n}3\right)M'_{1,0}(n,u).
  \end{split}
  \end{equation}

  The computation for the rest of terms in \eqref{equation: (t,24)=1
    3} is similar. We find for $t$ with $\delta_1(t)=1$,
  $3(t+2n)$ is a discriminant. For such $t$, $L_{3,0}(n,t)$
  is nonzero only when $3(t+2n)=(3u)^2$ is a
  square. Then $t+2n\equiv 3\mod 24$ and
  \begin{equation*}
  \begin{split}
    \JS{24}t&=\JS{-3}t\JS{-4}t\JS8t=\JS{-3}n(-1)^{(t-1)/2+(t^2-1)/8} \\
   &=\JS{-3}n(-1)^{n(n-1)/2}=\JS{12}n.
  \end{split}
  \end{equation*}
  Then by Lemma \ref{lemma: elliptic main lemma},
  \begin{equation} \label{equation: (t,24)=1 sum 3}
  \begin{split}
    \sum_{(t,24)=1}\delta_1(t)\JS{24}tL_{3,0}(n,t)
    =\frac12\JS{12}{n}\sum_{(u,2)=1}M_{3,0}(n,u)
  \end{split}
  \end{equation}
  where the sum runs over all \emph{positive} integers $u$ satisfying
  $9u^2<12n$ and the given conditions. By a similar computation and
  the same lemma, we also have
  \begin{equation} \label{equation: (t,24)=1 sum 4}
  \begin{split}
    \sum_{(t,24)=1}\delta_3(t)\JS{24}tL_{-3,0}(n,t)
  =\frac i2\JS{-4}r\JS{12}{n}\sum_{(u,2)=1}M'_{3,0}(n,u).
  \end{split}
  \end{equation}
  Combining \eqref{equation: (t,24)=1 3},
  \eqref{equation: (t,24)=1 sum 1}, \eqref{equation: (t,24)=1 sum 2},
  \eqref{equation: (t,24)=1 sum 3}, and \eqref{equation: (t,24)=1 sum
    4}, we get
  \begin{equation} \label{equation: (t,24)=1 summary 1}
  \begin{split}
   &\sum_{(t,24)=1}\sum_{f}
    \sum_{\gamma\in\Gamma_{n,t,f}/\SL(2,\Z)}J(\gamma) \\
   &\ =\frac{\lambda_1(n)n^{k+1/2}}4\JS{-4}n
    \sum_{(u,6)=1}\left(M_{1,0}(n,u)-M'_{1,0}(n,u)\right) \\
   &\ \qquad+\frac{n^{k+1/2}}4\JS{-4}n\sum_{(u,6)=3}
    \left(1+\JS{n}3\right)\left(M_{1,0}(n,u)-M'_{1,0}(n,u)\right) \\
   &\ \qquad+\frac{3^kn^{k+1/2}}4\JS{12}n\JS{12}r
    \sum_{(u,2)=1}\left(M_{3,0}(n,u)-M'_{3,0}(n,u)\right).
  \end{split}
  \end{equation}
  Notice that for $\tau=(eu+\sqrt{e^2u^2-4en})/2$ we have
  $$
    \frac{\tau^{1-2k}-\overline\tau^{1-2k}}{\tau-\overline\tau}
  =-(en)^{1-2k}\frac{\tau^{2k-1}-\overline\tau^{2k-1}}{\tau-\overline\tau}.
  $$
  Now if $n\equiv 1\mod 3$, then $\JS{-4}n=\JS{12}n$ and
  $$
    \sum_{(t,24)=1}\sum_f\sum_\gamma J(\gamma)=
   -\frac{n^{3/2-k}}8\left(A_{0,0}(n)+C_{0,0}(n)\right).
  $$
  (Note that in the definition of $A_{\ell,m}(n)$, the integers $u$
  can be positive or negative, but in \eqref{equation: (t,24)=1
    summary 1}, the integers $u$ are always positive. This explains
  the additional factor $1/2$ above.) If $n\equiv2\mod 3$, then the
  factor $1+\JS{n}3$ in the middle sum in \eqref{equation: (t,24)=1
    summary 1} is equal to $0$ and we have
  $$
    \sum_{(t,24)=1}\sum_f\sum_\gamma J(\gamma)=
   -\frac{n^{3/2-k}}8C_{0,0}(n).
  $$
  This completes the proof.
\end{proof}

\end{subsubsection}

\begin{subsubsection}{Case $(t,24)=3$}
\begin{Lemma} \label{lemma: (t,24)=3 summary} We have
$$
  \sum_{(t,24)=3}\sum_f\sum_{\gamma\in\Gamma_{n,t,f}/\SL(2,\Z)}J(\gamma)
 =-\frac{n^{3/2-k}}4\begin{cases}
  0, &\text{if }n\equiv 1\mod 3, \\
  A_{0,0}(n), &\text{if }n\equiv 2\mod 3,
  \end{cases}
$$
where $A_{0,0}(n)$ is defined as in Notation \ref{notation: ABCD}.
\end{Lemma}

\begin{proof} Since $(t,24)=3$, we have $(f,6)=1$. Then for class
  representatives $\gamma=\SM abcd$ given in Lemma
  \ref{lemma: class representatives}, we have, by Lemma \ref{lemma: eta
    multiplier},
  $$
    J(\gamma)=\frac{n^{k+1/2}}{w_{n,t,f}}\JS dc
    e^{2\pi irc/8}e^{-2\pi irc(t/3)/8}\frac{\rho^{1/2-k}}{\rho-\overline\rho}.
  $$
  Now
  \begin{equation*}
  \begin{split}
    e^{2\pi irc/8}e^{-2\pi irc(t/3)/8}
  &=\frac12\left(\JS8{rc}+i\JS{-8}{rc}\right)
    \left(\JS8{rc(t/3)}-i\JS{-8}{rc(t/3)}\right) \\
  &=\frac12\left(-\JS8t+\JS{-8}t\right)
   +\frac i2\JS{-4}{rc}\left(-\JS8t-\JS{-8}t\right) \\
  &=-\delta_3(t)\JS8t-i\delta_1(t)\JS{-4}{rc}\JS8t,
  \end{split}
  \end{equation*}
  where $\delta_1(t)$ and $\delta_3(t)$ are defined as
  \eqref{equation: delta}. Then
  \begin{equation*}
  \begin{split}
   &\sum_{(t,24)=3}\sum_f\sum_{\gamma\in\Gamma_{n,t,f}/\SL(2,\Z)}J(\gamma) \\
   &\qquad=n^{k+1/2}\sum_{(t,24)=3}\left(-\delta_3(t)\JS8tL_{1,0}(n,t)
    -i\delta_1(t)\JS{-4}r\JS8tL_{-1,0}(n,t)\right).
  \end{split}
  \end{equation*}
  By Lemma \ref{lemma: elliptic main lemma}, the term
  $\delta_3(t)L_{1,0}(n,t)$ is nonzero only when $t+2n$ is a
  square and $\delta_1(t)L_{1,0}(n,t)$ is nonzero only when
  $2n-t$ is a square. Since $3|t$, if $n\equiv 1\mod 3$, then $2n\pm
  t\equiv 2\mod 3$ and $2n\pm t$ can never be a square. Therefore, if
  $n\equiv 1\mod 3$, then the sum is $0$. If $n\equiv 2\mod 3$ and
  $2n\pm t$ is a square, we have, as in \eqref{equation: 8t=-4n},
  $$
    \JS8t=\JS{-4}n=-\JS{12}n.
  $$
  Then following the computation in \eqref{equation: (t,24)=1 sum 1}
  and \eqref{equation: (t,24)=1 sum 2}, we get
  \begin{equation*}
  \begin{split}
    \sum_{(t,24)=3}\sum_f\sum_\gamma J(\gamma)
  &=\frac{n^{k+1/2}}2\JS{12}n\sum_{(u,6)=1}\left(M_{1,0}(n,u)-M_{1,0}'(n,u)\right) \\
  &=-\frac{n^{3/2-k}}4A_{0,0}(n).
  \end{split}
  \end{equation*}
  (Note that here since $3|t$, the integer $u$ such that $2n\pm t=u^2$
  is never a multiple of $3$, while the second sum in the definition
  of $A_{0,0}(n)$ is just $0$. Also, since $n\equiv 2\mod 3$,
  $\Delta=u^2-4n$ is never a multiple of $3$. So the sum over $g$ with
  $3|g$ in the definition of $A_{0,0}(n)$ is empty.) This completes
  the proof.
\end{proof}
\end{subsubsection}

\begin{subsubsection}{Case $(t,24)=4$}

\begin{Lemma} \label{lemma: (t,24)=4 summary}
  We have
  $$
    \sum_{(t,24)=4}\sum_f\sum_{\gamma\in\Gamma_{n,t,f}/\SL(2,\Z)}J(\gamma)
   =-\frac{n^{3/2-k}}{16}
    \begin{cases}
    D_0(n), &\text{if }n\equiv 1\mod 4, \\
    0, &\text{if }n\equiv 7 \mod 12, \\
    B_0(n), &\text{if }n\equiv 11\mod 12.
    \end{cases}
  $$
\end{Lemma}

\begin{proof} Let class representatives $\SM abcd$ be chosen according
  to Lemma \ref{lemma: class representatives}. Since $(t,24)=4$, the
  integers $f$ must be odd, and the entries $c$ in the class
  representatives are always odd. Regardless of whether $3|f$ or not,
  we have $bd(1-c^2)\equiv 0\mod 24$ so that the term $J(\gamma)$ in
  Proposition \ref{proposition: Shimura trace} is equal to
  $$
    J(\gamma)=\frac{n^{k+1/2}}{w_{n,t,f}}\JS dc
    e^{2\pi irc/8}e^{-2\pi irc(t/4)/6}\frac{\rho^{1/2-k}}{\rho-\overline\rho},
  $$
  where $\rho=(t+\sqrt{t^2-4n^2})/2$. Now if $3\nmid f$, then $3\nmid
  c$ and
  \begin{equation*}
  \begin{split}
   &e^{2\pi irc/8}e^{-2\pi irc(t/4)/6} \\
  &\qquad=\frac1{2\sqrt 2}\left(\JS8{rc}+i\JS{-8}{rc}\right)
    \left(1-i\sqrt 3\JS{-3}{rc(t/4)}\right) \\
  &\qquad=\frac{\sqrt2}4\JS8{rc}+i\frac{\sqrt2}4\JS{-8}{rc}
   +\frac{\sqrt6}4\JS{24}{rc}\JS{-3}t
  -i\frac{\sqrt6}4\JS{-24}{rc}\JS{-3}t
  \end{split}
  \end{equation*}
  If $3|f$, then $3|c$ and
  $$
    e^{2\pi irc/8}e^{-2\pi irc(t/4)/6}=-\frac{\sqrt2}2\left(
    \JS8{rc}+i\JS{-8}{rc}\right).
  $$
  Therefore,
  \begin{equation} \label{equation: (t,24)=4 1}
  \begin{split}
   &\sum_{(t,24)=4}\sum_{f}
    \sum_{\gamma\in\Gamma_{n,t,f}/\SL(2,\Z)}J(\gamma) \\
   &\quad=\frac{n^{k+1/2}}4\sum_{(t,24)=4}
    \Bigg(\sqrt 2\JS8rL_{2,0}(n,t)+i\sqrt2\JS{-8}rL_{-2,0}(n,t) \\
   &\quad\qquad+\sqrt 6\JS{24}r\JS{-3}tL_{6,0}(n,t)
  -i\sqrt 6\JS{-24}r\JS{-3}tL_{-6,0}(n,t)\Bigg).
  \end{split}
  \end{equation}
  By Lemma \ref{lemma: elliptic main lemma}, $L_{2,0}(n,t)$ is
  nonzero only when $(t+2n)/2$ is a square. Since $4\|t$, this can
  possibly happen only when $n\equiv 3\mod 4$. Also, since $3\nmid t$,
  when $n\equiv 1\mod 3$ and $(t+2n)/2$ is indeed a square, this
  square must be a multiple of $3$. By the same reasoning,
  $L_{-2,0}(n,t)$ can possibly be nonzero only when $n\equiv 3\mod
  4$. Moreover, if $n\equiv1\mod3$ and $(2n-t)/2$ is a square, then this
  square is a multiple of $3$. Thus, by Lemma \ref{lemma: elliptic
    main lemma}, we have
  \begin{equation} \label{equation: (t,24)=4 sum 1}
  \begin{split}
   &\sum_{(t,24)=4}\left(\sqrt2\JS8rL_{2,0}(n,t)+i\sqrt2\JS{-8}r
    L_{-2,0}(n,t)\right) \\
  &\qquad=\frac{2^k\delta_3(n)}2\JS8r\Bigg(\lambda_2(n)\sum_{(u,6)=1}
    \left(M_{2,0}(n,u)-M'_{2,0}(n,u)\right) \\
  &\qquad\qquad+\sum_{(u,6)=3}\left(1-\JS{n}3\right)
    \left(M_{2,0}(n,u)-M'_{2,0}(n,u)\right)\Bigg),
  \end{split}
  \end{equation}
  where the sums run over all \emph{positive} integers $u$
  satisfying $4u^2<8n$ and the given conditions, and $\delta_j(n)$ and
  $\lambda_j(n)$ are defined by \eqref{equation: delta} and
  \eqref{equation: lambda}, respectively.

  Similarly, $L_{6,0}(n,t)$ (respectively, $L_{-6,0}(n,t)$) can
  possibly be nonzero only when $n\equiv 1\mod 4$ and $(t+2n)/6$
  (respectively, $(2n-t)/6$) is a square. Moreover, when $(t+2n)/6$ is
  a square, we have $\JS{-3}t=\JS{-3}n=\JS{12}n$ and when $(2n-t)/6$
  is a square, we have $\JS{-3}t=-\JS{-3}n=-\JS{12}n$. Then, by Lemma
  \ref{lemma: elliptic main lemma},
  \begin{equation} \label{equation: (t,24)=4 sum 2}
  \begin{split}
   &\sum_{(t,24)=4}\left(\sqrt6\JS{24}r\JS{-3}tL_{6,0}(n,t)
      -i\sqrt6\JS{-24}r\JS{-3}tL_{-6,0}(n,t)\right) \\
   &\qquad=\frac{6^k\delta_1(n)}2\JS{24}r\JS{12}n
    \sum_{(u,2)=1}\left(M_{6,0}(n,u)-M'_{6,0}(n,u)\right),
  \end{split}
  \end{equation}
  where the sum runs over all \emph{positive} odd integer $u$
  satisfying $36u^2<24n$. Substituting
  \eqref{equation: (t,24)=4 sum 1} and
  \eqref{equation: (t,24)=4 sum 2} into \eqref{equation: (t,24)=4 1}
  and simplifying, we obtain the claimed formula. (The reader is
  reminded again that the integers $u$ in the sums in \eqref{equation:
    (t,24)=4 sum 1} and \eqref{equation: (t,24)=4 sum 2} are all 
  positive, while the integers $u$ in the definition of
  $B_{\ell}(n)$ and $D_\ell(n)$ can be positive or negative.)
\end{proof}
\end{subsubsection}

\begin{subsubsection}{Case $(t,24)=8$}
\begin{Lemma} \label{lemma: (t,24)=8 summary}
  We have
  $$
    \sum_{(t,24)=4}\sum_f\sum_{\gamma\in\Gamma_{n,t,f}/\SL(2,\Z)}J(\gamma)
   =-\frac{n^{3/2-k}}{16}
    \begin{cases}
    0, &\text{if }n\equiv 1\mod 12, \\
    B_0(n), &\text{if }n\equiv 5\mod 12, \\
    D_0(n), &\text{if }n\equiv 3\mod 4. \end{cases}
  $$
\end{Lemma}

\begin{proof} The proof is almost identical with the proof of the case
  $(t,24)=4$. We omit the proof.
\end{proof}
\end{subsubsection}

\begin{subsubsection}{Case $(t,24)=12$}
\begin{Lemma} \label{proposition: (t,24)=12 summary}
We have
$$
  \sum_{(t,24)=12}\sum_f\sum_{\gamma\in\Gamma_{n,t,f}/\SL(2,\Z)}=
  -\frac{n^{3/2-k}}8\begin{cases}
   B_{0}(n), &\text{if }n\equiv 7\mod 12, \\
  0, &\text{else}, \end{cases}
$$
where $B_{0}(n)$ is defined as in Notation \ref{notation: ABCD}.
\end{Lemma}

\begin{proof}
We must have $(f,6)=1$. If $\SM abcd\in\Gamma_{n,t,f}$ is a class
representative given in Lemma \ref{lemma: class representatives}, then
$$
  \epsilon(a,b,c,d)^{-r}=-\JS dce^{2\pi irc/8}
=-\frac{\sqrt2}2\JS dc\left(\JS8{rc}+i\JS{-8}{rc}\right),
$$
so that
\begin{equation*}
\begin{split}
  &\sum_{(t,24)=12}\sum_f\sum_\gamma J(\gamma) \\
  &\qquad=-\frac{n^{k+1/2}}2\sum_{(t,24)=12}\left(\sqrt2\JS8rL_{2,0}(n,t)
  +i\sqrt2\JS{-8}rL_{-2,0}(n,t)\right).
\end{split}
\end{equation*}
The rest of computation is similar to that of the case $(t,24)=4$ and
$3|f$. We find that $L_{1,2,0}(n,t)$ is nonzero only when $(t+2n)/2$
is a square. Since $3|t$ and $4\|t$, $(t+2n)/2$ can never be a square
unless $n$ satisfies $n\equiv1\mod 3$ and $n\equiv 3\mod 4$, i.e.,
unless $n\equiv 7\mod 12$. Applying Lemma \ref{lemma: elliptic main
  lemma}, we get the claimed formula.
\end{proof}
\end{subsubsection}

\begin{subsubsection}{Case $(t,24)=24$}
\begin{Lemma} \label{lemma: (t,24)=24 summary}
We have
\begin{equation*}
\begin{split}
 \sum_{(t,24)=24}\sum_f\sum_{\gamma\in\Gamma_{n,t,f}/\SL(2,\Z)}J(\gamma)
=-\frac{n^{3/2-k}}8\begin{cases}
  B_{0}(n), &\text{if }n\equiv 1\mod 12, \\
  0, &\text{else}. \end{cases}
\end{split}
\end{equation*}
where $B_{0}(n)$ is defined as in Notation \eqref{notation: ABCD}.
\end{Lemma}

\begin{proof} This case is almost identical with the case
  $(t,24)=12$. We omit the proof.
\end{proof}
\end{subsubsection}

\begin{subsubsection}{Case $(t,24)=2$ and $2\nmid f$}
\begin{Lemma} \label{lemma: (t,24)=2 f odd summary}
  When $(t,24)=2$ and $f$ is odd, we have
  $$
    \sum_{(t,24)=2}\sum_{f\text{ odd}}\sum_{\gamma\in\Gamma_{n,t,f}/\SL(2,\Z)}J(\gamma)
   =-\frac{n^{3/2-k}}{16}\begin{cases}
    D_1^\ast(n), &\text{if }n\equiv 1\mod 3, \\
    (B_1^\ast(n)+D_1^\ast(n)), &\text{if }n\equiv 2\mod 3.\end{cases}
  $$
\end{Lemma}

\begin{proof} Let the class representatives $\gamma=\SM abcd$ be
  chosen according to Lemma \ref{lemma: class representatives}. Set
  $t'=t/2$. The integers $c$ are odd and regardless of whether $3|f$
  or not, we have $bd(1-c^2)\equiv0\mod24$ so that
  $$
    J(\gamma)=\frac{n^{k+1/2}}{w_{n,t,f}}\JS dce^{2\pi irc/8}
    e^{-2\pi irct'/12}\frac{\rho^{1/2-k}}{\rho-\overline\rho}.
  $$
  If $3\nmid f$, then
  \begin{equation*}
  \begin{split}
    e^{2\pi irc/8}e^{-2\pi irct'/12}
  &=\frac{\sqrt2}4\left(\JS8{rc}+i\JS{-8}{rc}\right)
    \left(\sqrt3\JS{12}{rct'}-i\JS{-4}{rct'}\right) \\
  &=\frac{\sqrt2}4\JS8{rc}\JS{-4}{t'}\left(1-i\JS{-4}{rc}\right) \\
  &\qquad\qquad
  +\frac{\sqrt6}4\JS{24}{rc}\JS{12}{t'}\left(1+i\JS{-4}{rc}\right).
  \end{split}
  \end{equation*}
  If $3|f$, then
  \begin{equation*}
  \begin{split}
    e^{2\pi irc/8}e^{-2\pi irct'/12}
  &=\frac{i\sqrt2}2\JS{-4}{-rct'/3}\left(\JS8{rc}+i\JS{-8}{rc}\right)
  \\
  &=\frac{\sqrt2}2\JS8{rc}\JS{-4}{t'}\left(-1+i\JS{-4}{rc}\right).
  \end{split}
  \end{equation*}
  Thus,
  \begin{equation} \label{equation: (t,24)=2 f odd temp}
  \begin{split}
   &\sum_{(t,24)=2}\sum_{(f,2)=1}\sum_\gamma J(\gamma) \\
   &=\frac{n^{k+1/2}}4\sum_{(t,24)=2}\Bigg(
    \sqrt2\JS{8}r\JS{-4}{t'}L_{2,0}(n,t)
  -i\sqrt2\JS{-8}r\JS{-4}{t'}L_{-2,0}(n,t) \\
   &\qquad\qquad+\sqrt6\JS{12}r\JS{12}{t'}L_{6,0}(n,t)
   +i\sqrt6\JS{-3}r\JS{12}{t'}L_{-6,0}(n,t)\Bigg).
  \end{split}
  \end{equation}
  By Lemma \ref{lemma: elliptic main lemma}, $L_{2,0}(n,t)$ is nonzero
  only when $(t+2n)/2=t'+n$ is a square. Since $2|(t'+n)$, if $t'+n$
  is indeed a square, then we have $t'+n\equiv0\mod4$ and
  $\JS{-4}{t'}=-\JS{-4}n$. Note also that since $3\nmid t$, if
  $n\equiv 1\mod 3$, then this square is necessarily a multiple of
  $3$. Likewise, $L_{-2,0}(n,t)$ is nonzero only
  when $n-t'$ is a square. If $n-t'$ is indeed a square, then we have
  $\JS{-4}{t'}=\JS{-4}n$. Then by Lemma \ref{lemma: elliptic main
    lemma},
  \begin{equation} \label{equation: (t,24)=2 f odd 1}
  \begin{split}
   &\sum_{(t,24)=2}\left(\JS{8}r\JS{-4}{t'}L_{2,0}(n,t)
    -i\JS{-8}r\JS{-4}{t'}L_{-2,0}(n,t)\right) \\
   &\qquad=-\JS8r\JS{-4}n\Bigg(\lambda_2(n)\sum_{2|u,3\nmid u}
    \left(M_{2,0}(n,u)-M'_{2,0}(n,u)\right) \\
   &\qquad\qquad+\sum_{6|u}\left(1-\JS{n}3\right)
    \left(M_{2,0}(n,u)-M'_{2,0}(n,u)\right)\Bigg),
  \end{split}
  \end{equation}
  where the sums run over all positive integers $u$ such that
  $4u^2<8n$ and the given conditions are satisfied, and $\lambda_2(n)$
  is defined by \eqref{equation: lambda}.

  Similarly, by Lemma \ref{lemma: elliptic main lemma}, the term
  $L_{6,0}(n,t)$ (respectively, $L_{-6,0}(n,t)$) is nonzero only when
  $(t+2n)/6$ (respectively, $(2n-t)/6$) is a square. If $(t+2n)/6$
  (respectively, $(2n-t)/6$) is indeed a square, this square must be a
  multiple of $4$ since $2\nmid t'$. It follows that $t'+n\equiv 0\mod12$
  (respectively, $n-t'\equiv0\mod 12$) and $\JS{12}{t'}=\JS{12}n$.
  Thus, by Lemma \ref{lemma: elliptic main lemma},
  \begin{equation} \label{equation: (t,24)=2 f odd 2}
  \begin{split}
   &\sum_{(t,24)=2}\left(\sqrt6\JS{24}r\JS{12}{t'}L_{6,0}(n,t)
    +i\sqrt6\JS{-24}r\JS{12}{t'}L_{-6,0}(n,t)\right) \\
   &\qquad=\JS{24}r\JS{12}n\sum_{2|u}\left(M_{6,0}(n,u)-M'_{6,0}(n,u)\right),
  \end{split}
  \end{equation}
  where the sum runs over all positive integers $u$ satisfying
  $36u^2<24n$ and $2|u$.

  Inserting \eqref{equation: (t,24)=2 f odd 1} and \eqref{equation:
    (t,24)=2 f odd 2} into \eqref{equation: (t,24)=2 f odd temp} and
  simplifying, we get the claimed formula.
\end{proof}
\end{subsubsection}

\begin{subsubsection}{Case $(t,24)=2$ and $2|f$}

\begin{Lemma} \label{lemma: eta 2||t} Assume that $2\|t$ and write
  $t=2t'$. For class representatives $\SM abcd$ of
  $\Gamma_{n,t,f}/\SL(2,\Z)$ given in Lemma \ref{lemma: class
    representatives}, define $\epsilon(a,b,c,d)$ as in Lemma
  \ref{lemma: eta multiplier}. Denote the odd part of $c$ by $c'$.
  \begin{enumerate}
  \item If $2\|f$, then
  $$
    \epsilon(a,b,c,d)=\JS2n\JS{c'}d
    e^{-2\pi ic't'/12}e^{2\pi i(t'-1)/8},
  $$
  \item If $4\|f$, then
  $$
    \epsilon(a,b,c,d)=-(-1)^{(t^2-4n^2)/64}
    \JS{c'}de^{2\pi ic't'/3}e^{2\pi i(t'-1)/8},
  $$
  \item If $2^v\|f$, $v\ge 3$, then
  $$
    \epsilon(a,b,c,d)=\JS2{t'}^v\JS{c'}d
    e^{2\pi i2^{v-2}c't'/3}e^{2\pi i(t'-1)/8}.
  $$
  \end{enumerate}
\end{Lemma}

\begin{proof} We have
\begin{equation*}
\begin{split}
 &\left(ac(1-d^2)+d(b-c+3)-3\right)-
  \left(bd(1-c^2)+c(a+d)\right) \\
 &\qquad=-(n^2+2)cd+3d-3\equiv -3cd+3d-3 \mod 24.
\end{split}
\end{equation*}
Thus,
\begin{equation} \label{equation: 2|f epsilon}
  \epsilon(a,b,c,d)=\JS cde^{2\pi i\left(bd(1-c^2)+ct\right)/24}
  e^{2\pi i(-cd+d-1)/8}.
\end{equation}
From now on, we let $c'$, $f'$, and $t'$ be the odd parts of $c$, $f$,
and $t$, respectively.

Consider the case $2\|f$. By Lemma \ref{lemma: class representatives},
we may assume that $2\|c$ and $d\equiv t'\mod 8$. For such
representatives, we have $a=t-d=2t'-d\equiv t'\mod 8$ and
$$
  t^2-4n^2=(d-a)^2+4bc\equiv 4bc\mod 64.
$$
Now
$$
  t^2-4n^2=32\frac{(t')^2-n^2}8\equiv
  \begin{cases}
  32\mod 64, &\text{if }\JS2{t'}\JS2n=-1, \\
  0 \mod 64, &\text{if }\JS2{t'}\JS2n=1. \end{cases}
$$
which shows that $b$ is divisible by $4$ and $8|b$ if and only if
$\JS2{t'}\JS2n=1$. Now if $3\nmid f$, then $c^2\equiv 4\mod 24$ so
that
$$
  bd(1-c^2)\equiv \begin{cases}
  12 \mod 24, &\text{if }4\|b, \\
  0\mod 24, &\text{if }8|b. \end{cases}
$$
The same congruences also hold when $3|f$. Therefore, from
\eqref{equation: 2|f epsilon}, we have
\begin{equation*}
\begin{split}
  \epsilon(a,b,c,d)&=\JS2d\JS{c'}d\JS2n\JS2{t'}
  e^{2\pi ic't'/6}e^{2\pi i(-ct'+t'-1)/8} \\
  &=\JS2n\JS{c'}de^{-2\pi ic't'/12}e^{2\pi i(t'-1)/8}.
\end{split}
\end{equation*}
This proves Part (1).

We next consider the case $4\|f$. By Lemma \ref{lemma: class
  representatives}, we may assume that $d\equiv
t'\mod 16$. For such representatives, we also have $a=2t'-d\equiv
t'\mod 16$. Thus, from $t^2-4n^2=(d-a)^2+4bc$, we see that
$4bc\equiv t^2-4n^2\mod 256$. That is,
$$
  \frac b4\frac c4\equiv\frac{t^2-4n^2}{64}\mod 4,
$$
from which we obtain $e^{2\pi ibd(1-c^2)/24}=(-1)^{b/4}=(-1)^{(t^2-4n^2)/64}$.
It follows that
\begin{equation*}
\begin{split}
  \epsilon(a,b,c,d)&=(-1)^{(t^2-4n^2)/64}\JS{c}de^{2\pi ict/24}
  e^{-2\pi icd/8}e^{2\pi i(d-1)/8} \\
 &=-(-1)^{(t^2-4n^2)/64}\JS{c'}de^{2\pi ic't'/3}e^{2\pi i(t'-1)/8}.
\end{split}
\end{equation*}
This proves Part (2).

For $8\|f$, class representatives given in Lemma \ref{lemma: class
  representatives} satisfy either $d\equiv t'+4\mod 8$ or $d\equiv
t'\mod 8$. In either case, we have
$$
  \JS2{d}e^{2\pi i(d-1)/8}=\JS{2}{t'}e^{2\pi i(t'-1)/8}.
$$
Then from \eqref{equation: 2|f epsilon}, we get
$$
  \epsilon(a,b,c,d)=\JS2d\JS{c'}de^{2\pi ict/24}e^{2\pi i(d-1)/8}
 =\JS2{t'}\JS{c'}de^{4\pi ic't'/3}e^{2\pi i(t'-1)/8}.
$$
This proves the case $8\|f$. The proof of the case $16|f$ is similar
and is omitted.
\end{proof}

\begin{Lemma} \label{lemma: (t,24)=2 2||f summary} We have
\begin{equation*}
\begin{split}
 &\sum_{(t,24)=2}\sum_{2\|f}\sum_{\gamma\in\Gamma_{n,t,f}/\SL(2,\Z)}J(\gamma)
 \\
 &\qquad=-\frac{n^{3/2-k}}{16}
  \begin{cases}
  A_{2,0}^\ast(n)+C_{1,0}(n),
    &\text{if }n\equiv 1\mod 12, \\
  C_{1,0}(n), &\text{if }n\equiv 5\mod 12, \\
  A_{1,0}(n)+C_{2,0}^\ast(n),
    &\text{if }n\equiv 7\mod 12, \\
  C_{2,0}^\ast(n), &\text{if }n\equiv 11\mod 12. \end{cases}
\end{split}
\end{equation*}
\end{Lemma}

\begin{proof} Let class representatives $\gamma=\SM abcd$ be chosen as per
  Lemma \ref{lemma: class representatives}. In particular, we have
  $d\equiv(t/2)\mod 8$. Let $c'$ and $t'$ denote
  the odd parts of $c$ and $t$, respectively. By \eqref{equation:
    J(gamma)} and Lemma \ref{lemma: eta 2||t}, we have
  $$
    J(\gamma)=\frac{n^{k+1/2}}{w_{n,t,f}}\JS2n\JS{c'}d
    e^{2\pi irc't'/12}e^{-2\pi ir(t'-1)/8}\frac{\rho^{1/2-k}}{\rho-\overline\rho},
  $$
  where $\rho=(t+\sqrt{t^2-4n^2})/2$. We check directly using the
  quadratic reciprocity law that
  \begin{equation} \label{equation: reciprocity}
    \JS{c'}de^{-2\pi ir(t'-1)/8}=\JS d{c'}\JS8{t'}\begin{cases}
    1, &\text{if }t'\equiv 1\mod 4, \\ \displaystyle
    i\JS{-4}{rc'}, &\text{if }t'\equiv 3\mod 4. \end{cases}
  \end{equation}
  If $3\nmid f$, then
  \begin{equation*}
  \begin{split}
   &\JS{c'}de^{2\pi irc't'/12}e^{-2\pi ir(t'-1)/8} \\
   &\qquad=\frac12\JS d{c'}\JS8{t'}\left(\sqrt3\JS{12}{rc't'}+i\JS{-4}{rc't'}
    \right)\left(\delta_1(t')+i\delta_3(t')\JS{-4}{rc'}\right) \\
   &\qquad=\frac12\JS d{c'}\JS8{t'}\left(
    \delta_3(t')+i\delta_1(t')\JS{-4}{rc'}\right) \\
   &\qquad\qquad+\frac{\sqrt3}2\JS d{c'}\JS{24}{t'}\JS{12}{rc'}
    \left(\delta_1(t')+i\delta_3(t')\JS{-4}{rc'}\right)
  \end{split}
  \end{equation*}
  where $\delta_1(t')$ and $\delta_3(t')$ are defined by
  \eqref{equation: delta}.
  If $3|f$, then $3|c'$ and
  \begin{equation*}
  \begin{split}
    \JS{c'}de^{2\pi irc't'/12}e^{-2\pi ir(t'-1)/8}
   &=i\JS d{c'}\JS8{t'}\JS{-4}{rc't'/3}
    \left(\delta_1(t')+i\delta_3(t')\JS{-4}{rc'}\right) \\
   &=-\JS d{c'}\JS8{t'}\left(\delta_3(t')+i\delta_1(t')\JS{-4}{rc'}\right).
  \end{split}
  \end{equation*}
  It follows that
  \begin{equation} \label{equation: (t,24)=2 2||f J}
  \begin{split}
   &\sum_{(t,24)=2}\sum_{2\|f}\sum_\gamma J(\gamma) \\
   &\quad=\frac{n^{k+1/2}}2\JS2n\sum_{(t,24)=2}\Bigg(\JS8{t'}\left(
    \delta_3(t')L_{1,1}(n,t)+i\delta_1(t')\JS{-4}rL_{-1,1}(n,t)\right)
  \\
   &\quad\quad+\sqrt3\JS{24}{t'}\JS{12}r\left(
    \delta_1(t')L_{3,1}(n,t)+i\delta_3(t')\JS{-4}rL_{-3,1}(n,t)\right)
    \Bigg)
  \end{split}
  \end{equation}

  By Lemma \ref{lemma: elliptic main lemma}, $L_{1,1}(n,t)$ is nonzero
  only when $t+2n$ is a square. If $t+2n=u^2$ is indeed a square,
  then we have $u^2/2=t'+n\equiv 0,2\mod 8$. Then the condition
  $\delta_3(t')\neq 0$ forces $u$ to satisfy
  \begin{equation} \label{equation: (t,24)=2 2||f temp 1}
    \begin{cases}
    4|u, &\text{if }n\equiv 1\mod 4, \\
    2\|u, &\text{if }n\equiv 3\mod 4, \end{cases}
  \end{equation}
  and also
  \begin{equation} \label{equation: (t,24)=2 2||f temp 2}
    \JS8{t'}=\JS{-4}n\JS8n.
  \end{equation}
  Furthermore, since $3\nmid t$, when $n\equiv2\mod 3$, we must have
  $3|u$.

  Similarly, $L_{-1,1}(n,t)$ is nonzero only when $2n-t$ is a
  square. If $2n-t=u^2$ is indeed a square, then the condition
  $\delta_1(t)=1$ forces \eqref{equation: (t,24)=2 2||f temp 1} and
  \eqref{equation: (t,24)=2 2||f temp 2} to hold. 
  Furthermore, if $n\equiv 2\mod 3$, then we must have $3|u$. Thus,
  by Lemma \ref{lemma: elliptic main lemma}, when $n\equiv 1\mod 4$,
  \begin{equation} \label{equation: (t,24)=2 2||f sum 1 n=1}
  \begin{split}
   &\JS2n\sum_{(t,24)=2}\JS8{t'}\left(
    \delta_3(t')L_{1,1}(n,t)+i\delta_1(t')\JS{-4}rL_{-1,1}(n,t)\right) \\
   &\qquad=\frac14\lambda_1(n)\sum_{4|u,3\nmid u}\left(
    M_{1,0}(n,u)-M'_{1,0}(n,u)\right) \\
   &\qquad\qquad+\frac14\sum_{12|u}\left(1+\JS{n}3\right)
    \left(M_{1,0}(n,u)-M'_{1,0}(n,u)\right)
  \end{split}
  \end{equation}
  and when $n\equiv 3\mod 4$,
  \begin{equation} \label{equation: (t,24)=2 2||f sum 1 n=3}
  \begin{split}
   &\JS2n\sum_{(t,24)=2}\JS8{t'}\left(
    \delta_3(t')L_{1,1}(n,t)+i\delta_1(t')\JS{-4}rL_{-1,1}(n,t)\right) \\
   &\qquad=-\frac14\lambda_1(n)\sum_{2\|u,3\nmid u}\left(
    M_{1,0}(n,u)-M'_{1,0}(n,u)\right) \\
   &\qquad\qquad-\frac14\sum_{2\|u,3|u}\left(1+\JS{n}3\right)
    \left(M_{1,0}(n,u)-M'_{1,0}(n,u)\right),
  \end{split}
  \end{equation}
  where the sums run over all positive integers $u$ satisfying
  $u^2<4n$ and the specified conditions and $\lambda_1(n)$ is defined
  by \eqref{equation: lambda}.

  Likewise, $L_{3,1}(n,t)$ (respectively, $L_{-3,1}(n,t)$) is nonzero
  only when $(t+2n)/3$ (respectively, $(2n-t)/3$) is a 
  square. If $(t+2n)/3=u^2$ (respectively, $(2n-t)/3=u^2$ is indeed a
  square, then $t'+n\equiv 0,6\mod 8$ (respectively, $n-t'\equiv
  0,6\mod 8$). The condition $\delta_1(t')\neq 0$ (respectively,
  $\delta_3(t')\neq 0$) forces that
  $$
    \begin{cases}
    2\|u, &\text{if }n\equiv 1\mod 4, \\
    4|u, &\text{if }n\equiv 3\mod 4, \end{cases}
  $$
  and also checking case by case, we find
  $$
    \JS{24}{t'}=\JS{24}n.
  $$
  Then by Lemma \ref{lemma: elliptic main lemma}, for $n\equiv 1\mod 4$,
  \begin{equation} \label{equation: (t,24)=2 2||f sum 2 n=1}
  \begin{split}
   &\JS2n\sum_{(t,24)=2}\JS{24}{t'}\left(
    \delta_1(t')L_{3,1}(n,t)+i\delta_3(t')\JS{-4}rL_{-3,1}(n,t)\right) \\
   &\qquad=\frac14\JS{12}n\sum_{2\|u}\left(
    M_{3,0}(n,u)-M'_{3,0}(n,u)\right),
  \end{split}
  \end{equation}
  and for $n\equiv 3\mod 4$,
  \begin{equation} \label{equation: (t,24)=2 2||f sum 2 n=3}
  \begin{split}
   &\JS2n\sum_{(t,24)=2}\JS{24}{t'}\left(
    \delta_1(t')L_{3,1}(n,t)+i\delta_3(t')\JS{-4}rL_{-3,1}(n,t)\right) \\
   &\qquad=\frac14\JS{12}n\sum_{4|u}\left(
    M_{3,0}(n,u)-M'_{3,0}(n,u)\right),
  \end{split}
  \end{equation}
  where the sums run over all positive integers $u$ such that
  $9u^2<12n$ and the specified conditions are met. Substituting
  \eqref{equation: (t,24)=2 2||f sum 1 n=1},
  \eqref{equation: (t,24)=2 2||f sum 1 n=3},
  \eqref{equation: (t,24)=2 2||f sum 2 n=1}, and
  \eqref{equation: (t,24)=2 2||f sum 2 n=3} into
  \eqref{equation: (t,24)=2 2||f J} and simplifying, we obtain the
  claimed formula.
\end{proof}
\end{subsubsection}

\begin{subsubsection}{Case $(t,24)=2$ and $4\|f$}
\begin{Lemma} We have
\begin{equation*}
\begin{split}
 &\sum_{(t,24)=2}\sum_{4\|f}\sum_{\gamma\in\Gamma_{n,t,f}/\SL(2,\Z)}J(\gamma) \\
 &=-\frac{n^{3/2-k}}{32}\begin{cases}
  \left(2A_{1,1}(n)-C_{2,0}(n)+C_{3,0}^\ast(n)\right),
    &\text{if }n\equiv 1\mod 24, \\
  \left(C_{2,0}(n)-C_{3,0}^\ast(n)\right),
    &\text{if }n\equiv 5\mod 24, \\
  \left(A_{2,0}(n)-A_{3,0}^\ast(n)+2C_{1,1}(n)\right),
    &\text{if }n\equiv 7\mod 24, \\
  2C_{1,1}(n), &\text{if }n\equiv 11\mod 24, \\
  \left(2A_{1,1}(n)+C_{2,0}(n)-C_{3,0}^\ast(n)\right),
    &\text{if }n\equiv 13\mod 24, \\
  \left(-C_{2,0}(n)+C_{3,0}^\ast(n)\right),
    &\text{if }n\equiv 17\mod 24, \\
  \left(-A_{2,0}(n)+A_{3,0}^\ast(n)+2C_{1,1}(n)\right),
    &\text{if }n\equiv 19\mod 24, \\
  2C_{1,1}(n), &\text{if }n\equiv 23\mod 24.
  \end{cases}
\end{split}
\end{equation*}
\end{Lemma}

\begin{proof} Let class representatives $\gamma=\SM abcd$ be chosen as
  in Lemma \ref{lemma: class representatives}. For convenience, set
  $$
    \mu_{n,t}=(-1)^{(t^2-4n^2)/64}.
  $$
  By Lemma \ref{lemma: eta 2||t}, \eqref{equation: J(gamma)}, and
  \eqref{equation: reciprocity}, the term $J(\gamma)$ in Proposition
  \ref{proposition: Shimura trace} is equal to
  \begin{equation} \label{equation: (t,24)=2 4||f J 1}
  \begin{split}
    J(\gamma)&=-\frac{n^{k+1/2}}{w_{n,t,f}}\mu_{n,t}
      \JS{c'}de^{-2\pi irc't'/3}e^{-2\pi ir(t'-1)/8}
      \frac{\rho^{1/2-k}}{\rho-\overline\rho} \\
   &=-\frac{n^{k+1/2}}{w_{n,t,f}}\mu_{n,t}
      \JS d{c'}\JS8{t'}e^{-2\pi irc't'/3}
      \left(\delta_1(t')+i\JS{-4}{rc'}\delta_3(t')\right)
      \frac{\rho^{1/2-k}}{\rho-\overline\rho},
  \end{split}
  \end{equation}
  where $c'$ and $t'$ denote the odd parts of $c$ and $t$,
  respectively, $\rho=(t+\sqrt{t^2-4n^2})/2$, and $\delta_j(t')$ are
  defined by \eqref{equation: delta}.

  When $3\nmid f$, we have $3\nmid c'$ and
  \begin{equation} \label{equation: (t,24)=2 4||f cubic}
    e^{-2\pi irc't'/3}=-\frac12-\frac{i\sqrt3}2\JS{-3}{rc't'}.
  \end{equation}
  When $3|f$, we have $3|c$ and $e^{-2\pi irc't'/3}=1$. Thus,
  \begin{equation} \label{equation: (t,24)=2 4||f J}
  \begin{split}
  &\sum_{(t,24)=2}\sum_{4\|f}\sum_\gamma J(\gamma) \\
  &\quad=\frac{n^{k+1/2}}2\sum_{(t,24)=2}
   \mu_{n,t}\JS8{t'}\left(\delta_1(t')L_{1,2}(n,t)
    +i\delta_3(t')\JS{-4}rL_{-1,2}(n,t)\right) \\
  &\quad\qquad+\frac{\sqrt3}2n^{k+1/2}\JS{12}r\sum_{(t,24)=2}
   \mu_{n,t}\JS{-24}{t'} \\
  &\quad\qquad\qquad\times\left(-\delta_3(t')L_{3,2}(n,t)
   +i\delta_1(t')\JS{-4}rL_{-3,2}(n,t)\right)
  \end{split}
  \end{equation}
  By Lemma \ref{lemma: elliptic main lemma}, $L_{1,2}(n,t)$ is nonzero
  only when $t+2n=u^2$ is a square. Then $\delta_1(t')\neq 0$ forces
  \begin{equation} \label{equation: (t,24)=2 4||f u}
    \begin{cases}
    2\|u, &\text{if }n\equiv 1\mod 4, \\
    4|u, &\text{if }n\equiv 3\mod 4. \end{cases}
  \end{equation}
  Likewise, $L_{-1,2}(n,t)$ is nonzero only when $2n-t=u^2$ is a
  square. Then the condition $\delta_3(t')\neq 0$ implies
  \eqref{equation: (t,24)=2 4||f u} as well. Note that since $3\nmid
  t$, when $n\equiv 2\mod 3$, the integer $u$ must be a multiple of $3$.

  When $n\equiv 1\mod 4$, $\delta_1(t')=1$, and $t+2n$ is indeed a
  square, we have $t'+n\equiv 2\mod 8$ so that
  \begin{equation} \label{equation: (t,24)=2 4||f temp 1}
    \JS8{t'}=\JS8n
  \end{equation}
  and
  \begin{equation} \label{equation: (t,24)=2 4||f temp 2}
    \mu_{n,t}=(-1)^{(t^2-4n^2)/64}=(-1)^{(t'-n)/8}
   =(-1)^{(1-n)/4}=\JS8n.
  \end{equation}
  Similarly, \eqref{equation: (t,24)=2 4||f temp 1} and
  \eqref{equation: (t,24)=2 4||f temp 2} also hold when $n\equiv 1\mod
  4$, $\delta_3(t')=1$, and $2n-t$ is indeed a square. Then, by
  Lemma \ref{lemma: elliptic main lemma}, when $n\equiv 1\mod 4$,
  \begin{equation} \label{equation: (t,24)=2 4||f sum 1 n=1}
  \begin{split}
  &\sum_{(t,24)=2}\mu_{n,t}\JS8{t'}\left(\delta_1(t')L_{1,2}(n,t)
   +i\delta_3(t')\JS{-4}rL_{-1,2}(n,t)\right) \\
  &\qquad=\frac{\lambda_1(n)}4\sum_{2\|u,3\nmid u}
   \left(M_{1,1}(n,u)-M'_{1,1}(n,u)\right) \\
  &\qquad\qquad+\frac14\sum_{2\|u,3|u}
   \left(1+\JS n3\right)\left(M_{1,1}(n,u)-M'_{1,1}(n,u)\right),
  \end{split}
  \end{equation}
  where $\lambda_1(n)$ is defined by \eqref{equation: lambda}.

  When $n\equiv 3\mod 4$, $\delta_1(t')=1$, and $t+2n=u^2$ is
  indeed a square, we have $t'+n\equiv 0\mod 8$ so that
  \begin{equation} \label{equation: (t,24)=2 4||f temp 100}
    \JS8{t'}=\JS8n
  \end{equation}
  and
  \begin{equation} \label{equation: (t,24)=2 4||f mu}
    \mu_{n,t}=(-1)^{(t'+n)/8}=\begin{cases}
    -1, &\text{if }4\|u, \\
    1,  &\text{if }8|u. \end{cases}
  \end{equation}
  Similar conclusions hold for $\delta_3(t')L_{-1,2}(n,t)$. Thus, by
  Lemma \ref{lemma: elliptic main lemma}, when $n\equiv 3\mod 4$,
  \begin{equation} \label{equation: (t,24)=2 4||f sum 1 n=3}
  \begin{split}
  &\sum_{(t,24)=2}\mu_{n,t}\JS8{t'}\left(\delta_1(t')L_{1,2}(n,t)
   +i\delta_3(t')\JS{-4}rL_{-1,2}(n,t)\right) \\
  &\qquad=-\frac{\lambda_1(n)}8\JS8n\sum_{4\|u,3\nmid u}
   \left(M_{1,0}(n,u)-M'_{1,0}(n,u)\right) \\
  &\qquad\qquad+\frac{\lambda_1(n)}8\JS8n\sum_{8|u,3\nmid u}
   \left(M_{1,0}(n,u)-M'_{1,0}(n,u)\right) \\
  &\qquad\qquad-\frac18\JS8n\sum_{4\|u,3|u}
   \left(1+\JS n3\right)\left(M_{1,0}(n,u)-M'_{1,0}(n,u)\right) \\
  &\qquad\qquad+\frac18\JS8n\sum_{8|u,3|u}
   \left(1+\JS n3\right)\left(M_{1,0}(n,u)-M'_{1,0}(n,u)\right).
  \end{split}
  \end{equation}

  We now consider the second sum in \eqref{equation: (t,24)=2 4||f
    J}. By Lemma \ref{lemma: elliptic main lemma}, $L_{3,2}(n,t)$ is
  nonzero only when $(t+2n)/3=u^2$ is a square. The condition
  $\delta_3(t')=1$ implies that
  \begin{equation} \label{equation: (t,24)=2 4||f uu}
    \begin{cases}
    4|u, &\text{if }n\equiv 1\mod 4, \\
    2\|u, &\text{if }n\equiv 3\mod 4. \end{cases}
  \end{equation}
  Similarly, $L_{-3,2}(n,t)$ is nonzero only when $(2n-t)/3=u^2$ is a
  square and the integer $u$ has the same property as above.

  When $n\equiv 1\mod 4$, $\delta_3(t')=1$, and $(t+2n)/3$ is
  indeed a square, we have
  $$
    \JS{-24}{t'}=-\JS{-24}n
  $$
  since $t'+n\equiv 0\mod 24$. Moreover, \eqref{equation: (t,24)=2
    4||f mu} remains valid in this case. For the sum $L_{-3,2}(n,t)$,
  when $\delta_1(t')=1$ and $(2n-t)/3$ is indeed a square, we have
  $$
    \JS{-24}{t'}=\JS{-24}n
  $$
  since $n-t'\equiv 0\mod 24$. Thus, by Lemma \ref{lemma:
    elliptic main lemma}, when $n\equiv 1\mod 4$,
  \begin{equation} \label{equation: (t,24)=2 4||f sum 2 n=1}
  \begin{split}
   &\sum_{(t,24)=2}\mu_{n,t}\JS{-24}{t'}
    \left(-\delta_3(t')L_{3,2}(n,t)+i\delta_1(t')\JS{-4}rL_{-3,2}(n,t)\right)\\
   &\qquad=-\frac18\JS{-24}n\sum_{4\|u}\left(M_{3,0}(n,u)-M'_{3,0}(n,u)\right)\\
   &\qquad\qquad+\frac18\JS{-24}n\sum_{8|u}
    \left(M_{3,0}(n,u)-M'_{3,0}(n,u)\right).
  \end{split}
  \end{equation}

  Now assume that $n\equiv 3\mod 4$. For the sum $L_{3,2}(n,t)$, if
  $\delta_3(t')=1$ and $(t+2n)/3$ is indeed a square, then
  $t'+n\equiv 6\mod 48$ and
  $$
    \JS{-24}{t'}=\JS{-3}{t'}\JS8{t'}=-\JS{-3}n\JS8n=-\JS{-24}n
  $$
  and
  $$
    \mu_{n,t}=(-1)^{(t'-n)/8}=(-1)^{(6-2n)/8}=\JS{-8}n.
  $$
  For the sum $L_{-3,2}(n,t)$, if $\delta_1(t')=1$ and $(2n-t)/3$
  is indeed a square, then $n-t'\equiv 6\mod 24$ and
  $$
    \JS{-24}{t'}=\JS{-24}n
  $$
  and
  $$
    \mu_{n,t}=(-1)^{(t'+n)/8}=(-1)^{(n-3)/4}=\JS{-8}n.
  $$
  Then by Lemma \ref{lemma: elliptic main lemma}
  \begin{equation} \label{equation: (t,24)=2 4||f sum 2 n=3}
  \begin{split}
   &\sum_{(t,24)=2}\mu_{n,t}\JS{-24}{t'}
    \left(-\delta_3(t')L_{3,2}(n,t)+i\delta_1(t')\JS{-4}rL_{-3,2}(n,t)\right)\\
   &\qquad=\frac1{4}\JS{12}n\sum_{2\|u}\left(M_{3,1}(n,u)-M'_{3,1}(n,u)\right).
  \end{split}
  \end{equation}
  Substituting \eqref{equation: (t,24)=2 4||f sum 1 n=1} and
  \eqref{equation: (t,24)=2 4||f sum 2 n=1} into
  \eqref{equation: (t,24)=2 4||f J} in the case $n\equiv 1\mod 4$,
  \eqref{equation: (t,24)=2 4||f sum 1 n=3} and
  \eqref{equation: (t,24)=2 4||f sum 2 n=3} into
  \eqref{equation: (t,24)=2 4||f J} in the case $n\equiv 3\mod 4$, and
  simplifying, we get the lemma. (The reader is reminded again that
  all the sums over $u$ in the proof are running over only positive
  $u$, but the sums defining $A_{\ell,m}(n)$ and $C_{\ell,m}(n)$ are running
  over both positive and negative $u$.)
\end{proof}
\end{subsubsection}

\begin{subsubsection}{Case $(t,24)=2$ and $8|f$}
  \begin{Lemma} \label{lemma: (t,24)=2 8|f summary}
  Assume that $v\ge 3$. When $n\equiv 1\mod 4$,
  \begin{equation*}
  \begin{split}
   &\sum_{(t,24)=2}\sum_{2^v\|f}\sum_{\gamma\in\Gamma_{n,t,f}/\SL(2,\Z)}J(\gamma)\\
   &\qquad=n^{3/2-k}\JS8n^{v-1}\left(\frac{\lambda_1(n)}{16}A_{1,v-1}(n)
   -\frac{(-1)^{v-1}}{2^{v+2}}C_{v-1,1}(n)
   -\frac{(-1)^{v-1}}{2^{v+3}}C_{v,0}^\ast(n)\right),
  \end{split}
  \end{equation*}
  and when $n\equiv 3\mod 4$,
  \begin{equation*}
  \begin{split}
   &\sum_{(t,24)=2}\sum_{2^v\|f}\sum_{\gamma\in\Gamma_{n,t,f}/\SL(2,\Z)}J(\gamma)\\
   &\qquad=n^{3/2-k}\JS8n^{v-1}\left(-\frac{\lambda_1(n)}{2^{v+2}}A_{v-1,1}(n)
   -\frac{\lambda_1(n)}{2^{v+3}}A_{v,0}^\ast(n)
   +\frac{(-1)^{v-1}}{16}C_{1,v-1}(n)\right),
  \end{split}
  \end{equation*}
  where $\lambda_1(n)$ is defined by \eqref{equation: lambda}.
  \end{Lemma}

  \begin{proof} The computation is almost the same as the case $4\|f$.
    Let class representatives $\gamma=\SM abcd$ be chosen according to
    Lemma \ref{lemma: class representatives}. By Lemma \ref{lemma: eta
      2||t} and \eqref{equation: reciprocity}, instead of
    \eqref{equation: (t,24)=2 4||f J 1}, we have
  $$
    J(\gamma)=\frac{n^{k+1/2}}{w_{n,t,f}}\JS d{c'}\JS8{t'}^{v-1}
    e^{-2\pi i2^{v-2}rc't'/3}\left(\delta_1(t')+i\JS{-4}{rc'}
    \delta_3(t')\right).
  $$
  Also, \eqref{equation: (t,24)=2 4||f cubic} is replaced by
  $$
    e^{-2\pi i2^{v-2}rc't'/3}
  =-\frac12-\frac{i\sqrt3}2\JS{-3}{2^vrc't'}
  =-\frac12-(-1)^v\frac{i\sqrt3}2\JS{-3}{rc't'}
  $$
  so that
  \begin{equation} \label{equation: (t,24)=2 8|f J}
  \begin{split}
  &\sum_{(t,24)=2}\sum_{2^v\|f}\sum_\gamma J(\gamma) \\
  &\quad=-\frac{n^{k+1/2}}2\sum_{(t,24)=2}
   \JS8{t'}^{v-1}\left(\delta_1(t')L_{1,v}(n,t)
    +i\delta_3(t')\JS{-4}rL_{-1,v}(n,t)\right) \\
  &\quad\qquad-(-1)^v\frac{\sqrt3}2n^{k+1/2}\JS{12}r\sum_{(t,24)=2}
   \JS{8}{t'}^{v-1}\JS{-3}{t'} \\
  &\quad\qquad\qquad\times\left(-\delta_3(t')L_{3,v}(n,t)
   +i\delta_1(t')\JS{-4}rL_{-3,v}(n,t)\right).
  \end{split}
  \end{equation}

  Consider first the sums involving $L_{1,v}(n,t)$. We have
  $$
    \begin{cases}2\|u, &\text{if }n\equiv 1\mod 4, \\
    2^{v-1}|u, &\text{if }n\equiv 3\mod 4, \end{cases}
  $$
  in place of \eqref{equation: (t,24)=2 4||f u}. For the case
  $n\equiv 1\mod 4$, \eqref{equation: (t,24)=2 4||f temp 1} remains
  valid. The same things hold for $L_{-1,v}(n,t)$. Then we have
  \begin{equation} \label{equation: (t,24)=2 8|f sum 1 n=1}
  \begin{split}
  &\sum_{(t,24)=2}\JS8{t'}^{v-1}\left(\delta_1(t')L_{1,v}(n,t)
   +i\delta_3(t')\JS{-4}rL_{-1,v}(n,t)\right) \\
  &\qquad=\frac{\lambda_1(n)}4\JS8n^{v-1}\sum_{2\|u,3\nmid u}
   \left(M_{1,v-1}(n,u)-M'_{1,v-1}(n,u)\right) \\
  &\qquad\qquad+\frac14\JS8n^{v-1}\sum_{2\|u,3|u}
   \left(1+\JS n3\right)\left(M_{1,v-1}(n,u)-M'_{1,v-1}(n,u)\right).
  \end{split}
  \end{equation}
  For the case $n\equiv 3\mod 4$, \eqref{equation: (t,24)=2 4||f temp
    100} is also valid. Then instead of \eqref{equation: (t,24)=2 4||f
    sum 1 n=3}, we have
  \begin{equation} \label{equation: (t,24)=2 8|f sum 1 n=3}
  \begin{split}
  &\sum_{(t,24)=2}\JS8{t'}^{v-1}\left(\delta_1(t')L_{1,v}(n,t)
   +i\delta_3(t')\JS{-4}rL_{-1,v}(n,t)\right) \\
  &\qquad=\frac{\lambda_1(n)}{2^v}\JS8n^{v-1}\sum_{2^{v-1}\|u,3\nmid u}
   \left(M_{1,1}(n,u)-M'_{1,1}(n,u)\right) \\
  &\qquad\qquad+\frac{\lambda_1(n)}{2^{v+1}}\JS8n^{v-1}\sum_{2^v|u,3\nmid u}
   \left(M_{1,0}(n,u)-M'_{1,0}(n,u)\right) \\
  &\qquad\qquad+\frac1{2^v}\JS8n^{v-1}\sum_{2^{v-1}\|u,3|u}\left(1+\JS n3\right)
   \left(M_{1,1}(n,u)-M'_{1,1}(n,u)\right) \\
  &\qquad\qquad+\frac1{2^{v+1}}\JS8n^{v-1}\sum_{2^v|u,3|u}\left(1+\JS n3\right)
   \left(M_{1,0}(n,u)-M'_{1,0}(n,u)\right).
  \end{split}
  \end{equation}

  We now consider the sums involving $L_{3,v}(n,t)$. We have
  $$
    \begin{cases}
    2^{v-1}|u, &\text{if }n\equiv 1\mod 4, \\
    2\|u, &\text{if }n\equiv 3\mod 4 \end{cases}
  $$
  instead of \eqref{equation: (t,24)=2 4||f uu}. For the case $n\equiv
  1\mod 4$, we have $24|(t'+n)$ so that
  $$
    \JS8{t'}=\JS8n, \qquad\JS{-3}{t'}=-\JS{-3}n=-\JS{12}n.
  $$
  For the case $n\equiv 3\mod 4$, we have $t'+n\equiv 6\mod 24$ and
  $$
    \JS8{t'}=\JS8n, \qquad\JS{-3}{t'}=-\JS{-3}n=\JS{12}n.
  $$
  For the sum $L_{-3,v}(n,t)$, when $n\equiv 1\mod 4$, we have
  $24|(n-t')$ and
  $$
    \JS8{t'}=\JS8n, \qquad \JS{-3}{t'}=\JS{-3}n=\JS{12}n,
  $$
  and when $n\equiv 3\mod 4$, we have $n-t'\equiv 6\mod 24$ and
  $$
    \JS8{t'}=\JS8n, \qquad\JS{-3}{t'}=\JS{-3}n=-\JS{12}n.
  $$
  Then when $n\equiv 1\mod 4$,
  \begin{equation} \label{equation: (t,24)=2 8|f sum 2 n=1}
  \begin{split}
   &\sum_{(t,24)=2}\JS{8}{t'}^{v-1}\JS{-3}{t'}
    \left(-\delta_3(t')L_{3,v}(n,t)+i\delta_1(t')\JS{-4}rL_{-3,v}(n,t)\right)\\
   &\qquad=\frac1{2^v}\JS{8}n^{v-1}\JS{12}n
    \sum_{2^{v-1}\|u}\left(M_{3,1}(n,u)-M'_{3,1}(n,u)\right)\\
   &\qquad\qquad+\frac1{2^{v+1}}\JS{8}n^{v-1}\JS{12}n\sum_{2^v|u}
    \left(M_{3,0}(n,u)-M'_{3,0}(n,u)\right),
  \end{split}
  \end{equation}
  and when $n\equiv 3\mod 4$,
  \begin{equation} \label{equation: (t,24)=2 8|f sum 2 n=3}
  \begin{split}
  &\sum_{(t,24)=2}\JS{8}{t'}^{v-1}\JS{-3}{t'}
    \left(-\delta_3(t')L_{3,v}(n,t)+i\delta_1(t')\JS{-4}rL_{-3,v}(n,t)\right)\\
  &=-\frac14\JS8n^{v-1}\JS{12}n
     \sum_{2\|u}\left(M_{3,v-1}(n,u)-M_{3,v-1}'(n,u)\right).
  \end{split}
  \end{equation}
  Substituting \eqref{equation: (t,24)=2 8|f sum 1 n=1} and
  \eqref{equation: (t,24)=2 8|f sum 2 n=1} into
  \eqref{equation: (t,24)=2 8|f J} in the case $n\equiv 1\mod 4$, and
  \eqref{equation: (t,24)=2 8|f sum 1 n=3} and
  \eqref{equation: (t,24)=2 8|f sum 2 n=3} into
  \eqref{equation: (t,24)=2 8|f J} in the case $n\equiv 3\mod 4$ and
  simplifying, we obtain the formulas.
  \end{proof}
\end{subsubsection}

\begin{subsubsection}{Case $(t,24)=6$ and $2\nmid f$}
\begin{Lemma}
  We have
  $$
    \sum_{(t,24)=6}\sum_{f\text{ odd}}
    \sum_{\gamma\in\Gamma_{n,t,f}/\SL(2,\Z)}J(\gamma)
  =-\frac{n^{3/2-k}}{8}\begin{cases}
     B_1^\ast(n), &\text{if }n\equiv 1\mod 3, \\
     0, &\text{if }n\equiv 2\mod 3. \end{cases}
  $$
\end{Lemma}

\begin{proof} If class representatives $\gamma=\SM abcd$ are chosen
  according to Lemma \ref{lemma: class representatives}, then by Lemma
  \ref{lemma: eta multiplier},
  $$
    J(\gamma)=\frac{n^{k+1/2}}{w_{n,t,f}}\JS dce^{2\pi irc/8}
    e^{-2\pi irc(t'/3)/4}\frac{\rho^{1/2-k}}{\rho-\overline\rho},
  $$
  where $t'=t/2$. Now
  $$
    e^{2\pi irc/8}e^{-2\pi irc(t'/3)/4}
   =\frac i{\sqrt 2}\JS{-4}{rct'}\left(\JS8{rc}+i\JS{-8}{rc}\right)
  $$
  and consequently,
  \begin{equation*}
  \begin{split}
   &\sum_{(t,24)=6}\sum_{(f,2)=1}\sum_\gamma J(\gamma) \\
   &\qquad=\frac{n^{k+1/2}}{\sqrt2}\JS8r\sum_{(t,24)=6}
    \JS{-4}{t'}\left(-L_{2,0}(n,t)+i\JS{-4}rL_{-2,0}(n,t)\right).
  \end{split}
  \end{equation*}
  The rest of computation is similar to \eqref{equation: (t,24)=2 f
    odd 1}. The only difference is that here because $3|t$, if
  $(t+2n)/2$ is a square, then $n$ has to be congruent to $1$ modulo
  $3$. We omit the details.
\end{proof}
\end{subsubsection}

\begin{subsubsection}{Case $(t,24)=6$ and $2\|f$}
\begin{Lemma} \label{lemma: (t,24)=6, 2||f} We have
  $$
    \sum_{(t,24)=6}\sum_{2\|f}\sum_{\gamma\in\Gamma_{n,t,f}/\SL(2,\Z)}J(\gamma)
  =-\frac{n^{3/2-k}}8
    \begin{cases}
    0, &\text{if }n\equiv 1\mod 3, \\
    A_{2,0}^\ast(n), &\text{if }n\equiv 5\mod 12,  \\
    A_{1,0}(n), &\text{if }n\equiv 11\mod 12. \end{cases}
  $$
\end{Lemma}

\begin{proof} Let class representatives $\gamma=\SM abcd$ of
  $\Gamma_{n,t,f}$ be chosen as in Lemma \ref{lemma: class representatives}.
  By Lemma \ref{lemma: eta 2||t},
  $$
    J(\gamma)=\frac{n^{k+1/2}}{w_{n,t,f}}\frac{\rho^{1/2-k}}
    {\rho-\overline\rho}\JS2n\JS{c'}de^{2\pi irc'(t'/3)/4}
    e^{-2\pi ir(t'-1)/8},
  $$
  where $t'=t/2$. Using \eqref{equation: reciprocity}, we find
  $$
    \JS{c'}de^{2\pi irc'(t'/3)/4}e^{-2\pi ir(t'-1)/8}
   =-i\JS d{c'}\JS8{t'}\JS{-4}{rc't'}\left(\delta_3(t')+i\JS{-4}{rc'}
    \delta_1(t')\right)
  $$
  and
  \begin{equation*}
  \begin{split}
   &\sum_{(t,24)=6}\sum_{2\|f}\sum_\gamma J(\gamma) \\
   &\qquad=-n^{k+1/2}\JS2n\sum_{(t,24)=6}\JS8{t'}\left(
    \delta_3(t')L_{1,1}(n,t)+i\delta_1(t')\JS{-4}rL_{-1,1}(n,t)\right).
  \end{split}
  \end{equation*}
  The rest of computation is similar to
  \eqref{equation: (t,24)=2 2||f sum 1 n=1} and
  \eqref{equation: (t,24)=2 2||f sum 1 n=3}. The only major difference
  is that here the sums are nonzero only when $n\equiv 2\mod 3$
  because $3|t$. We omit the details.
\end{proof}
\end{subsubsection}

\begin{subsubsection}{Case $(t,24)=6$ and $4\|f$}
\begin{Lemma} We have
\begin{equation*}
\begin{split}
  \sum_{(t,24)=6}\sum_{4\|f}\sum_\gamma J(\gamma)
=-\frac{n^{3/2-k}}{16}
  \begin{cases}
  0, &\text{if }n\equiv 1\mod 3, \\
  2A_{1,1}(n), &\text{if }n\equiv 5\mod 12, \\
  \left(-A_{2,0}(n)+A_{3,0}^\ast(n)\right), &\text{if }n\equiv 11\mod
  24, \\
  \left(A_{2,0}(n)-A_{3,0}^\ast(n)\right), &\text{if }n\equiv 23\mod
  24. \end{cases}
\end{split}
\end{equation*}
\end{Lemma}

\begin{proof} If class representatives $\gamma=\SM abcd$ are chosen
  according to Lemma \ref{lemma: class representatives}, then by Lemma
  \ref{lemma: eta 2||t} and \eqref{equation: reciprocity},
  $$
    J(\gamma)=-\frac{n^{k+1/2}}{w_{n,t,f}}\mu_{n,t}
    \JS d{c'}\JS8{t'}\left(\delta_1(t')
  +i\JS{-4}{rc'}\delta_3(t')\right)
    \frac{\rho^{1/2-k}}{\rho-\overline\rho},  
  $$
  where $\mu_{n,t}=(-1)^{(t^2-4n^2)/64}$, $c'$ and $t'$ denote the odd
    parts of $c$ and $t$, respectively. (Cf. \eqref{equation: (t,24)=2
      4||f J 1}.) Thus,
  \begin{equation*}
  \begin{split}
   &\sum_{(t,24)=6}\sum_{4\|f}\sum_\gamma J(\gamma) \\
   &\qquad=-n^{k+1/2}\sum_{(t,24)=6}\mu_{n,t}\JS8{t'}
    \left(\delta_1(t')L_{1,2}(n,t)+i\delta_3(t')\JS{-4}r
    L_{-1,2}(n,t)\right).
  \end{split}
  \end{equation*}
  Following the computation in
  \eqref{equation: (t,24)=2 4||f u}--\eqref{equation: (t,24)=2 4||f sum
    1 n=3} and noting that here the sums are nonzero only when
  $n\equiv 2\mod 3$, we get the claimed formula.
\end{proof}
\end{subsubsection}

\begin{subsubsection}{Case $(t,24)=6$ and $8|f$}
\begin{Lemma} \label{lemma: (t,24)=6 8|f summary}
Let $v\ge 3$ be an integer.
When $n\equiv 1\mod 3$,
$$
  \sum_{(t,24)=6}\sum_{2^v\|f}\sum_\gamma J(\gamma)=0.
$$
When $n\equiv 5\mod 12$,
$$
  \sum_{(t,24)=6}\sum_{2^v\|f}\sum_\gamma J(\gamma)=n^{3/2-k}\JS8n^{v-1}
  \frac{A_{1,v-1}(n)}8.
$$
When $n\equiv 11\mod 12$,
$$
  \sum_{(t,24)=6}\sum_{2^v\|f}\sum_\gamma J(\gamma)=-n^{3/2-k}\JS8n^{v-1}
  \left(\frac{A_{v-1,1}(n)}{2^{v+1}}+\frac{A_{v,0}(n)}{2^{v+2}}\right).
$$
\end{Lemma}

\begin{proof} We have
  $$
    J(\gamma)=\frac{n^{k+1/2}}{w_{n,t,f}}\JS d{c'}\JS8{t'}^{v-1}
    \left(\delta_1(t')+i\JS{-4}{rc'}\delta_3(t')\right)
  $$
  so that
  \begin{equation*}
  \begin{split}
   &\sum_{(t,24)=6}\sum_{2^v\|f}\sum_\gamma J(\gamma) \\
   &\qquad=n^{k+1/2}\sum_{(t,24)=6}\JS8{t'}^{v-1}
    \left(\delta_1(t')L_{1,v}(n,t)+i\delta_3(t')\JS{-4}r
    L_{-1,v}(n,t)\right).
  \end{split}
  \end{equation*}
  The computation is similar to that in Lemma \ref{lemma: (t,24)=2 8|f
    summary} up to \eqref{equation: (t,24)=2 8|f sum 1 n=3}. We skip
  the details.
\end{proof}
\end{subsubsection}
\end{subsection}
\end{section}

\begin{section}{Traces of Hecke operators on $S_{2k}^\new(6)$}
\label{section: trace, integral weight}

Let $W_e$, $e=1,2,3,6$, be the Atkin-Lehner involutions on $S_{2k}(6)$.
For $\epsilon_2,\epsilon_3\in\{\pm 1\}$, let
$S_{2k}(6,\epsilon_2,\epsilon_3)$ be the Atkin-Lehner eigensubspace of
$S_{2k}(6)$ with eigenvalues $\epsilon_2$ and $\epsilon_3$ for $W_2$
and $W_3$, respectively, and let
$S_{2k}^\new(6,\epsilon_2,\epsilon_3)$ be its newform subspace. We
have
\begin{equation} \label{equation: trace, integral weight 1}
\begin{split}
 &\tr(T_n|S_{2k}^\new(6,\epsilon_2,\epsilon_3)) \\
=&\tr(T_n|S_{2k}(6,\epsilon_2,\epsilon_3))
 -\tr(T_n|S_{2k}(3,\epsilon_3))
 -\tr(T_n|S_{2k}(2,\epsilon_2))+\tr(T_n|S_{2k}(1)) \\
=&\frac14\Big(\tr(T_n|S_{2k}(6))
   +\epsilon_2\tr(T_nW_2|S_{2k}(6))+\epsilon_3\tr(T_nW_3|S_{2k}(6)) \\
&\qquad+\epsilon_6\tr(T_nW_6|S_{2k}(6))\Big)
-\frac12\Big(\tr(T_n|S_{2k}(3))+\epsilon_3\tr(T_nW_3|S_{2k}(3))\Big)\\
&\qquad-\frac12\Big(\tr(T_n|S_{2k}(2))+\epsilon_2\tr(T_nW_2|S_{2k}(2))\Big)
  +\tr(T_n|S_{2k}(1)),
\end{split}
\end{equation}
where $\epsilon_6=\epsilon_2\epsilon_3$. Thus, in order to obtain a
formula for the trace of $T_n$ on
$S_{2k}^\new(6,\epsilon_2,\epsilon_3)$, we need to evaluate
$\tr(T_nW_e|S_{2k}(N))$ for various $N$ and $e$.

The traces of $\tr(T_n|S_{2k}(N))$ have been computed by
\cite{Eichler-LNM,Hijikata}, while those of $\tr(T_nW_e|S_{2k}(N))$
were given by \cite{Yamauchi}. The formulas are summarized in the
following proposition. Note that in \cite{Yamauchi}, the numbers of
optimal embeddings are simply expressed as the number of solutions to
some congruence equations. Here we use the formula for the numbers
of optimal embedding from \cite{Ogg}. Note also that here we only give
the formulas in cases relevant to our discussion; the general cases
are much more complicated.

\begin{Proposition}
  \label{proposition: trace formula}
  Let $N$ be a positive squarefree integer, $e$ a
  positive divisor of $N$, and $k$ an integer greater than $1$. For a
  discriminant $\Delta<0$ of an imaginary quadratic order, we let
  $\Delta_0$ denote the discriminant of the field $\Q(\sqrt\Delta)$
  and for a prime $p$, let
  $$
    \alpha(\Delta,p)=\begin{cases} \displaystyle
    1+\JS{\Delta_0}p, &\text{if }p\nmid(\Delta/\Delta_0), \\
    2, &\text{if }p|(\Delta/\Delta_0). \end{cases}
  $$
  Then for a positive divisor $e$ of $N$ and a positive integer $n$
  relatively prime to $N$, we have
  \begin{equation*}
  \begin{split}
 &\tr(T_nW_e|S_{2k}(N)) \\
 &\quad=-\frac1{2e^{k-1}}\sum_{\Delta=e^2u^2-4en<0}\sum_{\substack{g^2|(\Delta/\Delta_0),\\
     (g,e)=1}}
    \prod_{p|(N/e)}\alpha(g^2\Delta_0,p)\cdot
  H(g^2\Delta_0)\frac{\tau^{2k-1}-\overline\tau^{2k-1}}{\tau-\overline\tau}
  \\
 &\qquad\qquad-2^{\omega(N)-1}\delta_1(e)\sum_{a|n}\min(a,n/a)^{2k-1}
 +\frac{2k-1}{12}\delta_2(e,n)n^{k-1}\prod_{p|N}(p+1),
  \end{split}
  \end{equation*}
  where $\tau=(eu+\sqrt{\Delta})/2$, $\omega(N)$ is the number of
  prime divisors of $N$,
  $$
    \delta_1(e)=\begin{cases}
    1, &\text{if }e=1, \\
    0, &\text{else}, \end{cases} \qquad
    \delta_2(e,n)=\begin{cases}
    1, &\text{if }e=1\text{ and }n\text{ is a square}, \\
    0, &\text{else}. \end{cases}
  $$
\end{Proposition}

Using this formula, we now compute the trace of $T_n$ on
$S_{2k}(6,\epsilon_2,\epsilon_3)$.

\begin{Proposition} \label{proposition: trace, integral weight}
  Let $n$ be a positive integer relatively prime to $6$.
  Then for $\epsilon_2,\epsilon_3\in\{\pm1\}$, the trace of $T_n$ on
  $S_{2k}^\new(\Gamma_0(6),\epsilon_2,\epsilon_3)$ is
  \begin{equation*}
  \begin{split}
  &-\frac18\sum_{u^2<4n,(g,6)=1}\left(1-\JS{\Delta_0}2\right)
    \left(1-\JS{\Delta_0}3\right)H(g^2\Delta_0)P_k(1,n,u) \\
  &\qquad+\frac{\epsilon_2}{8\cdot2^{k-1}}\sum_{4u^2<8n,(g,3)=1}
    \left(1-\JS{\Delta_0}3\right)H(g^2\Delta_0)P_k(2,n,u) \\
 &\qquad+\frac{\epsilon_3}{8\cdot3^{k-1}}\sum_{9u^2<12n,(g,2)=1}
    \left(1-\JS{\Delta_0}2\right)H(g^2\Delta_0)P_k(3,n,u) \\
 &\qquad-\frac{\epsilon_2\epsilon_3}{8\cdot6^{k-1}}
    \sum_{36u^2<24n,g}H(g^2\Delta_0)P_k(6,n,u)+\frac{2k-1}{24}\delta(n)n^{k-1},
  \end{split}
  \end{equation*}
  where
  $$
    \delta(n)=\begin{cases}
  1, &\text{if }n\text{ is a square}, \\
  0, &\text{else}. \end{cases}
  $$
\end{Proposition}

\begin{proof}
  We first consider the terms
  \begin{equation} \label{equation: e=1, all 1}
    \frac14\tr(T_n|S_{2k}(6))-\frac12\tr(T_n|S_{2k}(3))-\frac12\tr(T_n|S_{2k}(2))
   +\tr(T_n|S_{2k}(1))
  \end{equation}
  in \eqref{equation: trace, integral weight 1}. By Proposition
  \ref{proposition: trace formula}, we have
  \begin{equation} \label{equation: e=1, N=1}
  \begin{split}
    \tr(T_n|S_{2k}(1))
  &=-\frac12\sum_{u}\sum_gH(g^2\Delta_0)P_k(1,n,u)
    -\frac12\sum_{a|n}\min(a,n/a)^{2k-1} \\
  &\qquad\qquad+\frac{2k-1}{12}\delta(n)n^{k-1},
  \end{split}
  \end{equation}
  \begin{equation} \label{equation: e=1, N=2}
  \begin{split}
    \tr(T_n|S_{2k}(2))
  &=-\frac12\sum_u\left(2\sum_{2|g}H(g^2\Delta_0)+
    \left(1+\JS{\Delta_0}2\right)\sum_{2\nmid g}H(g^2\Delta_0)\right) \\
  &\qquad\times P_k(1,n,u)-\sum_{a|n}\min(a,n/a)^{2k-1}
   +\frac{2k-1}4\delta(n)n^{k-1},
  \end{split}
  \end{equation}
  \begin{equation} \label{equation: e=1, N=3}
  \begin{split}
    \tr(T_n|S_{2k}(3))
  &=-\frac12\sum_u\left(2\sum_{3|g}H(g^2\Delta_0)+
    \left(1+\JS{\Delta_0}3\right)\sum_{3\nmid g}H(g^2\Delta_0)\right) \\
  &\qquad\times P_k(1,n,u)-\sum_{a|n}\min(a,n/a)^{2k-1}
   +\frac{2k-1}3\delta(n)n^{k-1},
  \end{split}
  \end{equation}
  \begin{equation} \label{equation: e=1, N=6}
  \begin{split}
    \tr(T_n|S_{2k}(6))
  &=-\frac12\sum_u\Bigg(4\sum_{(g,6)=6}H(g^2\Delta_0)
    +2\sum_{(g,6)=3}\left(1+\JS{\Delta_0}2\right)H(g^2\Delta_0) \\
  &\qquad+2\sum_{(g,6)=2}\left(1+\JS{\Delta_0}3\right)H(g^2\Delta_0)
  \\
  &\qquad+\sum_{(g,6)=1}\left(1+\JS{\Delta_0}2\right)
   \left(1+\JS{\Delta_0}3\right)H(g^2\Delta_0)\Bigg)P_k(1,n,u) \\
  &\qquad+2\sum_{a|n}\min(a,n/a)^{2k-1}+(2k-1)\delta(n)n^{k-1},
  \end{split}
  \end{equation}
  where
  $$
    \delta(n)=\begin{cases}
    1, &\text{if }n\text{ is a square}, \\
    0, &\text{else}. \end{cases}
  $$
  Substituting \eqref{equation: e=1, N=1}--\eqref{equation: e=1, N=6}
  into \eqref{equation: e=1, all 1} and simplifying, we get
  \begin{equation} \label{equation: e=1, all 2}
  \begin{split}
   &\frac14\tr(T_n|S_{2k}(6))-\frac12\tr(T_n|S_{2k}(3))-\frac12\tr(T_n|S_{2k}(2))
   +\tr(T_n|S_{2k}(1)) \\
   &\qquad=-\frac18\sum_u\sum_{(g,6)=1}\left(1-\JS{\Delta_0}2\right)
    \left(1-\JS{\Delta_0}3\right)H(g^2\Delta_0)P_k(1,n,u) \\
   &\qquad\qquad+\frac{2k-1}{24}\delta(n)n^{k-1}.
  \end{split}
  \end{equation}
  We next consider the terms
  $$
    \frac{\epsilon_2}4\tr(T_nW_2|S_{2k}(6)-\frac{\epsilon_2}2
    \tr(T_nW_2|S_{2k}(2))
  $$
  in \eqref{equation: trace, integral weight 1}. We have, by
  Proposition \ref{proposition: trace formula},
  \begin{equation*}
  \begin{split}
    \tr(T_nW_2|S_{2k}(6))&=-\frac1{2\cdot2^{k-1}}\sum_{4u^2<8n}
    \Bigg(2\sum_{3|g}H(g^2\Delta_0) \\
  &\qquad\qquad+\sum_{3\nmid g}
    \left(1+\JS{\Delta_0}3\right)H(g^2\Delta_0)\Bigg)P_k(2,n,u)
  \end{split}
  \end{equation*}
  and
  $$
    \tr(T_nW_2|S_{2k}(2))=-\frac1{2\cdot2^{k-1}}\sum_{4u^2<8n}
    \sum_gH(g^2\Delta_0)P_k(2,n,u).
  $$
  Thus,
  \begin{equation} \label{equation: e=2, all}
  \begin{split}
   &\frac{\epsilon_2}4\tr(T_nW_2|S_{2k}(6)-\frac{\epsilon_2}2
    \tr(T_nW_2|S_{2k}(2)) \\
   &\qquad=\frac{\epsilon_2}{8\cdot2^{k-1}}\sum_{4u^2<8n}
    \sum_{(g,3)=1}\left(1-\JS{\Delta_0}3\right)H(g^2\Delta_0)
    P_k(2,n,u).
  \end{split}
  \end{equation}
  Similarly, we have
  \begin{equation} \label{equation: e=3, all}
  \begin{split}
   &\frac{\epsilon_3}4\tr(T_nW_3|S_{2k}(6)-\frac{\epsilon_3}2
    \tr(T_nW_3|S_{2k}(3)) \\
   &\qquad=\frac{\epsilon_2}{8\cdot3^{k-1}}\sum_{9u^2<12n}
    \sum_{(g,2)=1}\left(1-\JS{\Delta_0}2\right)H(g^2\Delta_0)
    P_k(3,n,u)
  \end{split}
  \end{equation}
  and
  \begin{equation} \label{equation: e=6}
    \frac{\epsilon_6}4\tr(T_nW_6|S_{2k}(6))
  =-\frac{\epsilon_6}{8\cdot6^{k-1}}\sum_{36u^2<24n}\sum_g
    H(g^2\Delta_0)P_k(6,n,u).
  \end{equation}
  Summarizing \eqref{equation: trace, integral weight 1},
  \eqref{equation: e=1, all 2}, \eqref{equation: e=2, all},
  \eqref{equation: e=3, all}, and \eqref{equation: e=6}, we obtain the
  claimed formula. This proves the proposition.
\end{proof}
\end{section}

\begin{section}{Comparison of traces}
\label{section: comparison}

In this section, we will prove
$$
  \tr(T_{n^2}|\S_{r,s}(1))=\JS{12}n\tr\left(T_n\Big|S_{2k}^\new\left(6,
  -\JS8r,-\JS{12}r\right)\right)
$$
for positive integers $n$ relatively prime to $6$ and thereby
establish Theorem \ref{theorem: 1}. The verification is done case by
case according to the residue of $n$ modulo $24$. Here we only work
out the cases $n\equiv 1\mod 24$ and $n\equiv 11\mod 24$ and omit the
proof for the other cases.

Recall that the Hecke operator $T_{n^2}$ on $\S_{r,s}(1)$ is defined by
$$
  T_{n^2}:f\longmapsto
  n^{k-3/2}\sum_{ad=n,a|d}af\big|[\MM_{(d/a)^2}^\ast]
 =n^{k-3/2}\sum_{a|m}af\big|[\MM_{(n/a^2)^2}^\ast],
$$
where $m$ is the positive integer such that $n/m^2$ is squarefree and
the action of $[\MM_{(n/a^2)^2}^\ast]=[\MM_{(n/a^2)^2}(1)^\ast]$ is
defined as in Lemma \ref{lemma: action of M}. We first consider the
case $n\equiv 1\mod 24$. Write $\ell=n/a^2$. Note that we have
$\ell\equiv 1\mod 24$. According to
Propositions \ref{proposition: scalar}, \ref{proposition: parabolic},
\ref{proposition: hyperbolic}, Lemmas \ref{lemma: elliptic prelim},
\ref{lemma: (t,24)=1 summary}--\ref{lemma: (t,24)=2 f odd summary} and
\ref{lemma: (t,24)=2 2||f summary}--\ref{lemma: (t,24)=6 8|f summary},
we have
\begin{equation} \label{equation: trace ell 0}
\begin{split}
  \tr[\MM_{\ell^2}^\ast]
&=\frac{\sqrt\ell}8\sum_{e=1,2,3,6}\JS{-4e}r
  \left(1-\JS{-e}3\right)(H(-4e\ell)-H(-e\ell)) \\
&\qquad-\frac{\ell^{3/2-k}}4(A_{0,0}(\ell)+C_{0,0}(\ell))-\frac{\ell^{3/2-k}}{8}D_0(\ell)
 -\frac{\ell^{3/2-k}}4B_0(\ell) \\
&\qquad-\frac{\ell^{3/2-k}}{8}D_1^\ast(\ell)
  -\frac{\ell^{3/2-k}}{8}(A_{2,0}^\ast(\ell)+C_{1,0}(\ell)) \\
&\qquad-\frac{\ell^{3/2-k}}{16}(2A_{1,1}(\ell)-C_{2,0}(\ell)+C_{3,0}^\ast(\ell))
\\
&\qquad+\ell^{3/2-k}\sum_{v\ge 3}\left(\frac1{8}A_{1,v-1}(\ell)
  +\frac{(-1)^v}{2^{v+1}}C_{v-1,1}(\ell)+\frac{(-1)^v}{2^{v+2}}C_{v,0}^\ast(\ell)\right) \\
&\qquad-\frac{\ell^{3/2-k}}4B_1^\ast(\ell)+\delta_1(\ell)\frac{2k-1}{24},
\end{split}
\end{equation}
where
$$
  \delta_1(\ell)=\begin{cases}
  1, &\text{if }\ell=1, \\
  0, &\text{else}. \end{cases}
$$
For the first sum above, we observe that
$$
  P_k(e,\ell,0) 
 =(e\ell)^{k-1}\frac{i^{2k-1}-i^{1-2k}}{2i}=(-e\ell)^{k-1}=-\JS{-4}r(e\ell)^{k-1}
$$
and the sum, including the factor $\sqrt\ell/8$ in front, can be written as
$$
  -\frac{\ell^{3/2-k}}8\sum_{e=1,2,3,6}\frac1{e^{k-1}}\JS
  er\left(1-\JS{-e}3\right)(H(-4e\ell)-H(-e\ell))P_k(e,\ell,0).
$$
Now we have $H(-e\ell)=0$ except for $e=3$, in which case we have
$$
  H(-12\ell)-H(-3\ell)=\left(1-\JS{-3\ell}2\right)H(-3\ell).
$$
Thus,
\begin{equation} \label{equation: trace ell 1}
\begin{split}
 &\frac{\sqrt\ell}8\sum_{e=1,2,3,6}\JS{-4e}r
  \left(1-\JS{-e}3\right)(H(-4e\ell)-H(-e\ell)) \\
 &\qquad=-\frac{\ell^{3/2-k}}8\Bigg(
  \left(1-\JS{-4\ell}2\right)\left(1-\JS{-4\ell}3\right)H(-4\ell)P_k(1,\ell,0)
  \\
 &\qquad\qquad+\frac1{2^{k-1}}\JS8r\left(1-\JS{-8\ell}3\right)H(-8\ell)P_k(2,\ell,0)
  \\
 &\qquad\qquad+\frac1{3^{k-1}}\JS{12}r\left(1-\JS{-3\ell}2\right)H(-3\ell)P_k(3,\ell,0)
  \\
 &\qquad\qquad+\frac1{6^{k-1}}\JS{24}rH(-24\ell)P_k(6,\ell,0)\Bigg).
\end{split}
\end{equation}
We next consider the terms $A_{i,j}(\ell)$ in \eqref{equation: trace
  ell 0}. Here we remind the reader that the summations $\sum_u\sum_g$
in the subsequent discussion are all subject to the condition
$(\ell,u,\Delta/g^2\Delta_0)=1$ inherited from the definition of $A_{i,j}(\ell)$.

For an integer $u$ contributing to the sums $A_{i,j}(\ell)$,
let $\Delta=u^2-4\ell$ and $\Delta_0$ be the discriminant of the field
$\Q(\sqrt{\Delta})$. For $A_{0,0}(\ell)$, we have
$$
  \JS{\Delta_0}2=\JS52=-1,
$$
and by Lemma \ref{lemma: alternative AB},
\begin{equation} \label{equation: trace ell A00}
  A_{0,0}(\ell)=\frac12\sum_{2\nmid u}
  \left(1-\JS{\Delta_0}2\right)\left(1-\JS{\Delta_0}3\right)
  \sum_{(g,6)=1}H(g^2\Delta_0)P_k(1,\ell,u).
\end{equation}
For $A_{2,0}^\ast(\ell)$, we have $4|\Delta_0$ and
\begin{equation} \label{equation: trace ell A20}
  A_{2,0}^\ast(\ell)=\sum_{4|u}\left(1-\JS{\Delta_0}2\right)
  \left(1-\JS{\Delta_0}3\right)\sum_{(g,6)=1}H(g^2\Delta_0)
  P_k(1,\ell,u).
\end{equation}
For $A_{1,j}(\ell)$, we let $v=\ord_2(\Delta/\Delta_0)/2$, i.e., the
$2$-adic valuation of the conductor of $\Delta$. If $v=1$, then since
$32|(u^2-4n)$ for $u$ with $2\|u$, we have $8|\Delta_0$ and
\begin{equation} \label{equation: trace ell A1j 1}
  -\sum_{2\|g}H\JS{\Delta}{g^2}=-\left(1-\JS{\Delta_0}2\right)
   \sum_{2\nmid g}H(g^2\Delta_0).
\end{equation}
If $v\ge 2$, then
\begin{equation} \label{equation: trace ell A1j 2}
\begin{split}
 &-\sum_{2\|g}H\JS{\Delta}{g^2}+\sum_{4|g}H\JS\Delta{g^2}
 =-\sum_{2^{v-1}\|g}H(g^2\Delta_0)+\sum_{j=0}^{v-2}\sum_{2^j\|g}H(g^2\Delta_0) \\
&\qquad=\sum_{2\nmid g}\left(-2^{v-2}\left(2-\JS{\Delta_0}2\right)
  +\sum_{j=1}^{v-2}2^{j-1}\left(2-\JS{\Delta_0}2\right)+1\right)H(g^2\Delta_0)\\
&\qquad=-\sum_{2\nmid g}\left(1-\JS{\Delta_0}2\right)H(g^2\Delta_0).
\end{split}
\end{equation}
It follows from \eqref{equation: trace ell A1j 1}, \eqref{equation:
  trace ell A1j 2}, and Lemma \ref{lemma: alternative AB} that
\begin{equation} \label{equation: trace ell A1j}
\begin{split}
 &-A_{1,1}(\ell)+\sum_{2\le j<\infty}A_{1,j}(\ell) \\
 &\qquad=-\sum_{2\|u}\sum_{(g,6)=1}\left(1-\JS{\Delta_0}2\right)\left(1-\JS{\Delta_0}3\right)
  H(g^2\Delta_0)P_k(1,\ell,u).
\end{split}
\end{equation}
In summary, from \eqref{equation: trace ell A00}, \eqref{equation:
  trace ell A20}, and \eqref{equation: trace ell A1j}, we get
\begin{equation} \label{equation: trace ell A}
\begin{split}
  &-\frac14A_{0,0}(\ell)-\frac18A_{2,0}^\ast(\ell)-\frac18A_{1,1}(\ell)
   +\frac18\sum_{j\ge 2}A_{1,j}(\ell) \\
  &\qquad=-\frac18\sum_{u\neq 0}\sum_{(g,6)=1}
   \left(1-\JS{\Delta_0}2\right)\left(1-\JS{\Delta_0}3\right)
   H(g^2\Delta_0)P_k(1,\ell,u),
\end{split}
\end{equation}
where the outer summation runs over all nonzero integers $u$ with
$u^2<4\ell$ and the inner summation runs over all positive integers
$g$ dividing the conductor of $\Delta=u^2-4\ell$ such that $(g,6)=1$ and
$(\ell,u,\Delta/g^2\Delta_0)=1$.

The terms $B_j(\ell)$ in \eqref{equation: trace ell 0}
are easy to deal with. By Lemma \ref{lemma: alternative AB}, we have
\begin{equation} \label{equation: trace ell B}
  -B_0(\ell)-B_1^\ast(\ell)=-\frac1{2\cdot2^{k-1}}\JS8r\sum_{u\neq 0}\sum_{(g,3)=1}
  \left(1-\JS{\Delta_0}3\right)H(g^2\Delta_0)P_k(2,\ell,u)
\end{equation}
Here, as above, the outer summation runs over all nonzero integers $u$
with $4u^2<8\ell$ and the inner summation runs over all positive
integers $g$ dividing the conductor of $\Delta=4u^2-8\ell$ such that
$(g,3)=1$ and $(\ell,u,\Delta/g^2\Delta_0)=1$.

We next consider $C_{i,j}(\ell)$ in \eqref{equation: trace ell 0}. If
$\Delta=9u^2-12n$ with $2\nmid u$, then $\Delta_0\equiv 5\mod 8$ and
we have
\begin{equation} \label{equation: trace ell C00}
  C_{0,0}(\ell)=\frac1{2\cdot3^{k-1}}\JS{12}r\sum_{2\nmid u}\left(1-\JS{\Delta_0}2\right)
  \sum_gH(g^2\Delta_0)P_k(3,\ell,u).
\end{equation}
If $\Delta=9u^2-12n$ with $2\|u$, then $4|\Delta_0$ and
\begin{equation} \label{equation: trace ell C10}
  C_{1,0}(\ell)=\frac1{2\cdot3^{k-1}}\JS{12}r\sum_{2\|u}\left(1-\JS{\Delta_0}2\right)
  \sum_gH(g^2\Delta_0)P_k(3,\ell,u).
\end{equation}
If $\Delta=9u^2-12n$ with $2^2\|u$, then $\Delta_0\equiv 1\mod 8$ and
for any odd integer $g$, we have $H(4g^2\Delta_0)=H(g^2\Delta_0)$. Thus,
\begin{equation} \label{equation: trace ell C20}
  C_{2,0}(\ell)-C_{2,1}(\ell)=0=\frac1{3^{k-1}}\JS{12}r\sum_{2^2\|u}
  \left(1-\JS{\Delta_0}2\right)\sum_{(g,2)=1}H(g^2\Delta_0)P_k(3,\ell,u).
\end{equation}
If $\Delta=9u^2-12n$ with $8|u$, then $\Delta_0\equiv 5\mod 8$ and for
any odd integer $g$, we have $H(4g^2\Delta_0)=3H(g^2\Delta_0)$. It
follows that $C_{j,0}(\ell)=3C_{j,1}(\ell)$ and
\begin{equation} \label{equation: trace ell C30}
\begin{split}
 &-\frac1{16}C_{3,0}^\ast(\ell)
  -\sum_{3\le j<\infty}\frac{(-1)^j}{2^{j+2}}C_{j,1}(\ell)
  +\sum_{3\le j<\infty}\frac{(-1)^j}{2^{j+2}}C_{j,0}^\ast(\ell) \\
 &\qquad=-\frac3{16}C_{3,1}^\ast(\ell)
  -\sum_{3\le j<\infty}\frac{(-1)^j}{2^{j+2}}C_{j,1}(\ell)
  +\sum_{3\le j<\infty}\frac{3(-1)^j}{2^{j+2}}
   \sum_{j\le i<\infty}C_{i,1}(\ell) \\
 &\qquad=\sum_{3\le i<\infty}C_{i,1}(\ell)\left(-\frac3{16}
  -\frac{(-1)^i}{2^{i+2}}+\sum_{3\le j\le
    i}\frac{3(-1)^j}{2^{j+2}}\right)
  =-\frac14\sum_{3\le i<\infty}C_{i,1}(\ell) \\
 &\qquad=-\frac1{8\cdot3^{k-1}}\JS{12}r\sum_{8|u,u\neq0}\left(1-\JS{\Delta_0}2\right)
  \sum_{(g,2)=1}H(g^2\Delta_0)P_k(3,\ell,u).
\end{split}
\end{equation}
Combining \eqref{equation: trace ell C00}--\eqref{equation: trace ell
  C30}, we get
\begin{equation} \label{equation: trace ell C}
\begin{split}
 &-\frac14C_{0,0}(\ell)-\frac18C_{1,0}(\ell)+\frac1{16}C_{2,0}(\ell)
  -\frac1{16}C_{3,0}^\ast(\ell) \\
 &\qquad\qquad+\sum_{v\ge 3}\left(\frac{(-1)^v}{2^{v+1}}
   C_{v-1,1}(\ell)+\frac{(-1)^v}{2^{v+2}}C_{v,0}^\ast(\ell)\right) \\
 &\qquad=-\frac1{8\cdot3^{k-1}}\JS{12}r\sum_{u\neq 0}\sum_{(g,2)=1}\left(1-\JS{\Delta_0}2\right)
  H(g^2\Delta_0)P_k(3,\ell,u).
\end{split}
\end{equation}
Again, here the outer summation runs over all nonzero integers $u$
with $9u^2<12\ell$ and the inner summation runs over all positive
divisors of the conductor of $\Delta$ satisfying $(g,2)=1$ and
$(\ell,u,\Delta/g^2\Delta_0)=1$.

The terms $D_j(\ell)$ in \eqref{equation: trace ell 0} are easy. We have
\begin{equation} \label{equation: trace ell D}
  -D_0(\ell)-D_1^\ast(\ell)=-\frac1{6^{k-1}}\JS{24}r\sum_{u\neq 0}\sum_g
  H(g^2\Delta_0)P_k(6,\ell,u).
\end{equation}
Here the summation over $g$ is subject to the condition
$(\ell,u,\Delta/g^2\Delta_0)=1$.

Substituting \eqref{equation: trace ell 1}, \eqref{equation: trace ell
  A}, \eqref{equation: trace ell B}, \eqref{equation: trace ell C}, and
\eqref{equation: trace ell D} into \eqref{equation: trace ell 0}, we
get
\begin{equation} \label{equation: trace ell temp}
\begin{split}
  \tr[\MM_{\ell^2}]&=
 -\frac{\ell^{3/2-k}}8\sum_u\sum_{(g,6)=1}\left(1-\JS{\Delta_0}2\right)
  \left(1-\JS{\Delta_0}3\right)H(g^2\Delta_0)P_k(1,\ell,u) \\
&\qquad-\frac{\ell^{3/2-k}}{8\cdot2^{k-1}}\JS8r\sum_u\sum_{(g,3)=1}
   \left(1-\JS{\Delta_0}3\right)H(g^2\Delta_0)P_k(2,\ell,u) \\
&\qquad-\frac{\ell^{3/2-k}}{8\cdot3^{k-1}}\JS{12}r\sum_u\sum_{(g,2)=1}
   \left(1-\JS{\Delta_0}2\right)H(g^2\Delta_0)P_k(3,\ell,u) \\
&\qquad-\frac{\ell^{3/2-k}}{8\cdot6^{k-1}}\JS{24}r\sum_u\sum_g
   H(g^2\Delta_0)P_k(6,\ell,u)+\delta_1(\ell)\frac{2k-1}{24}.
\end{split}
\end{equation}
Here the summations $\sum_g$ are all subject to the condition
$(\ell,u,\Delta/g^2\Delta_0)=1$. Now recall from the definition of
$T_{n^2}$ that
\begin{equation*}
 \tr(T_{n^2}|\S_{r,s}(1))=n^{k-3/2}\sum_{a|m}a\tr[\MM_{(n/a^2)^2}],
\end{equation*}
where $m$ is the integer such that $n/m^2$ is squarefree. Thus, we
need to sum up \eqref{equation: trace ell temp} over all $\ell=n/a^2$.
Observe that
$$
  P_k(e,n/a^2,u)=a^{2-2k}P_k(e,n,au).
$$
Take the fourth sum in \eqref{equation: trace ell temp} for
example. We find
\begin{equation*}
\begin{split}
 &n^{k-3/2}\sum_{a|m}a(n/a^2)^{3/2-k}\sum_u\sum_{(n/a^2,u,\Delta/g^2\Delta_0)=1}H(g^2\Delta_0)
  P_k(6,n/a^2,u) \\
 &\qquad=\sum_{a|m}\sum_u\sum_{(n,au,ah)=a}H\JS{a^2u^2-24n}{(ah)^2}P_k(6,n,au)
 \\
 &\qquad=\sum_u\sum_gH(g^2\Delta_0)P_k(6,n,u),
\end{split}
\end{equation*}
where there is no longer any restriction to $g$ in the inner
sum. Similar formulas hold for other three sums in \eqref{equation:
  trace ell temp}. Also,
$$
  n^{k-3/2}\sum_{a|m}a\delta_1(n/a^2)=
  \begin{cases}
  n^{k-1}, &\text{if }n\text{ is a square}, \\
  0, &\text{else}. \end{cases}
$$
Therefore, we find
\begin{equation*}
\begin{split}
  \tr(T_{n^2}|\S_{r,s}(1))&=
 -\frac18\sum_u\sum_{(g,6)=1}\left(1-\JS{\Delta_0}2\right)
  \left(1-\JS{\Delta_0}3\right)H(g^2\Delta_0)P_k(1,n,u) \\
&\quad-\frac1{8\cdot2^{k-1}}\JS8r\sum_u\sum_{(g,3)=1}
   \left(1-\JS{\Delta_0}3\right)H(g^2\Delta_0)P_k(2,n,u) \\
&\quad-\frac1{8\cdot3^{k-1}}\JS{12}r\sum_u\sum_{(g,2)=1}
   \left(1-\JS{\Delta_0}2\right)H(g^2\Delta_0)P_k(3,n,u) \\
&\quad-\frac1{8\cdot6^{k-1}}\JS{24}r\sum_u\sum_g
   H(g^2\Delta_0)P_k(6,n,u)+\frac{2k-1}{24}\delta(n)n^{k-1},
\end{split}
\end{equation*}
Comparing the trace of $T_n$ on $S_{2k}(6,\epsilon_2,\epsilon_3)$
given in Proposition \ref{proposition: trace, integral weight}, we
find
$$
  \tr(T_{n^2}|\S_{r,s}(1))=\JS{12}n\tr\left(T_n\Big|S_{2k}\left(6,-\JS8r,-\JS{12}r\right)
  \right).
$$
This proves the case $n\equiv 1\mod 24$. (Note that in this case
$\JS{12}n=1$.) We now consider the case $n\equiv 11\mod 24$.

Assume that $n\equiv 11\mod 24$. By
Propositions \ref{proposition: scalar}, \ref{proposition: parabolic},
\ref{proposition: hyperbolic}, Lemmas \ref{lemma: elliptic prelim},
\ref{lemma: (t,24)=1 summary}--\ref{lemma: (t,24)=2 f odd summary} and
\ref{lemma: (t,24)=2 2||f summary}--\ref{lemma: (t,24)=6 8|f summary},
we have
\begin{equation*}
\begin{split}
  \tr[\MM_{\ell^2}^\ast]
  &=\frac{\sqrt\ell}8\sum_{e=1,2,3,6}\JS{-4e}r\left(1-\JS{-4e\ell}3\right)
    (H(-4e\ell)-H(-e\ell)) \\
  &\qquad+\ell^{3/2-k}\Bigg(-\frac12A_{0,0}(\ell)-\frac14A_{1,0}(\ell)
  +\frac18A_{2,0}(\ell)-\frac18A_{3,0}^\ast(\ell) \\
  &\qquad+\sum_{3\le v<\infty}\left(\frac{(-1)^v}{2^v}
   A_{v-1,1}(\ell)+\frac{(-1)^v}{2^{v+1}}A_{v,0}(\ell)\right)-\frac18B_0^\ast(\ell)\\
  &\qquad-\frac14C_{0,0}(\ell)-\frac18C_{2,0}^\ast(\ell)-\frac18C_{1,1}(\ell)
   +\frac18\sum_{v\ge 2}C_{1,v}(\ell)-\frac18D_0^\ast(\ell)\Bigg).
\end{split}
\end{equation*}
The computation for $A_{i,j}(\ell)$ is parallel to that for
$C_{i,j}(\ell)$ in the case of $n\equiv 1\mod 24$, while the
computation for $C_{i,j}(\ell)$ is parallel to that for
$A_{i,j}(\ell)$ in the case of $n\equiv 1\mod 24$. The computation for
$B_0^\ast(\ell)$ and $D_0^\ast(\ell)$ is almost the same as
before. (The reader is reminded that there is a difference of $1/2$
between the case of $n\equiv 1\mod 3$ and the case of $n\equiv 2\mod
3$ in the formulas for $A_{i,j}(\ell)$ and $B_j(\ell)$ in Lemma
\ref{lemma: alternative AB}.) For the first sum above, instead of
\eqref{equation: trace ell 1}, we have
\begin{equation*}
\begin{split}
 &\frac{\sqrt\ell}8\sum_{e=1,2,3,6}\JS{-4e}r
  \left(1-\JS{-e\ell}3\right)(H(-4e\ell)-H(-e\ell)) \\
 &\qquad=-\frac{\ell^{3/2-k}}8\Bigg(
  \left(1-\JS{-\ell}2\right)\left(1-\JS{-\ell}3\right)H(-\ell)P_k(1,\ell,0)
  \\
 &\qquad\qquad+\frac1{2^{k-1}}\JS8r\left(1-\JS{-8\ell}3\right)H(-8\ell)P_k(2,\ell,0)
  \\
 &\qquad\qquad+\frac1{3^{k-1}}\JS{12}r\left(1-\JS{-12\ell}2\right)H(-12\ell)P_k(3,\ell,0)
  \\
 &\qquad\qquad+\frac1{6^{k-1}}\JS{24}rH(-24\ell)P_k(6,\ell,0)\Bigg).
\end{split}
\end{equation*}
Therefore, \eqref{equation: trace ell temp} remains valid, from which
we conclude that
$$
  \tr(T_{n^2}|\S_{r,s}(1))=\tr\left(T_n\Big|S_{2k}\left(6,-\JS8r,-\JS{12}r\right)\right).
$$
This proves the case $n\equiv 11\mod 24$. We skip the proof of the
other cases.
\end{section}

\end{document}